\documentclass[11pt,reqno]{amsart}

\usepackage{todonotes}
\usepackage{amsmath}
\usepackage{amsthm}
\usepackage{amssymb}
\usepackage{amsfonts}
\usepackage{amsxtra}
\usepackage{fullpage}
\usepackage{amscd}
\usepackage{mathrsfs}
\usepackage[all]{xypic}
\usepackage{enumerate}
\usepackage{graphicx}
\usepackage[mathscr]{eucal}
\usepackage{tikz}
\usepackage{tikz-cd}
\usepackage{verbatim}

\let\nc\newcommand
\let\renc\renewcommand

\theoremstyle{plain}
\newtheorem{thm}{Theorem}
\newtheorem{prop}[thm]{Proposition}
\newtheorem{cor}[thm]{Corollary}
\newtheorem{lem}[thm]{Lemma}

\theoremstyle{definition}
\newtheorem{defn}[thm]{Definition}
\newtheorem{example}{Example}
\newtheorem{remark}[thm]{Remark}

\newtheorem{rem}[thm]{Remark}

\numberwithin{thm}{section}
\nc{\bdm}{\begin{displaymath}}
\nc{\edm}{\end{displaymath}}
\nc{\bthm}{\begin{thm}}
\nc{\ethm}{\end{thm}}
\nc{\blem}{\begin{lem}}
\nc{\elem}{\end{lem}}
\nc{\bcor}{\begin{cor}}
\nc{\ecor}{\end{cor}}
\nc{\bprop}{\begin{prop}}
\nc{\eprop}{\end{prop}}
\nc{\bdef}{\begin{defn}}
\nc{\eddef}{\end{defn}}

\makeatletter
\renewcommand{\subsection}{\@startsection{subsection}{2}{0pt}{-3ex
plus -1ex minus -0.2ex}{-2mm plus -0pt minus
-2pt}{\normalfont\bfseries}} \makeatother

\def\hp{\hphantom{x}}

\DeclareMathOperator{\Ker}{\mathrm{Ker}}

\DeclareMathOperator{\End}{\mathrm{End}}

\DeclareMathOperator{\gr}{\mathrm{gr}}
\DeclareMathOperator{\Hilb}{{\mathrm{Hilb}}}

\DeclareMathOperator{\Lie}{\mathrm{Lie}}

\DeclareMathOperator{\Rep}{\mathrm{Rep}}

\DeclareMathOperator{\codim}{\mathrm{codim}}

\newcommand{\beq}{\begin{equation}\label}
\newcommand{\eeq}{\end{equation}}

\DeclareMathOperator{\Spec}{\mathrm{Spec}}

\newcommand{\iso}{{\;\stackrel{_\sim}{\to}\;}}

\DeclareMathOperator{\Hom}{\mathrm{Hom}}
\DeclareMathOperator{\GL}{\mathrm{GL}}
\DeclareMathOperator{\SL}{\mathrm{SL}}

\nc{\Z}{\mathbb{Z}}
\newcommand{\N}{\mathbb{N}}
\newcommand{\Q}{\mathbb{Q}}
\newcommand{\R}{\mathbb{R}}
\newcommand{\C}{\mathbb{C}}

\newcommand{\Fun}{\mbox{\mathrm{Fun}}\,}

\nc{\rank}{\textrm{rank} \,}
\nc{\ds}{\dots}

\let\mc\mathcal
\let\mf\mathfrak

\nc{\mbf}{\mathbf}
\nc{\LK}{\textsf{Irr}(K)}
\nc{\LW}{\textsf{Irr}(W)}
\nc{\Res}{\mathsf{Res} \, }
\nc{\Ind}{\mathsf{Ind} \, }

\nc{\cont}{\textrm{cont}}

\nc{\msf}{\mathsf}
\newcommand{\git}{\ensuremath{/\!\!/\!}}

\nc{\minusone}{-1}
\nc{\minustwo}{-2}
\nc{\Mod}{\mathrm{Mod} \,}
\nc{\ms}{\mathscr}
\nc{\Frac}{\mathrm{Frac} \,}
\nc{\ra}{\rightarrow}
\nc{\hra}{\hookrightarrow}
\nc{\lab}{\label}
\renc{\O}{\mc{O}}
\nc{\Tan}{\mc{T}}
\nc{\ul}{\underline}
\nc{\s}{\mathfrak{S}}
\nc{\g}{\mf{g}}
\nc{\pg}{\mf{pg}}
\nc{\pa}{\partial}
\nc{\tit}{\textit}
\nc{\Maxspec}{\mathrm{Maxspec} \, }
\nc{\gldim}{\mathrm{gl.dim}}
\nc{\rkm}{\mathrm{rk} \, (\mf{m})}
\nc{\sm}{\mathrm{sm}}
\nc{\PD}{\mathbb{PD}}
\nc{\hilb}{\textrm{Hilb}}
\nc{\<}{\langle}
\renc{\>}{\rangle}
\nc{\T}{\mathbb{T}}
\nc{\X}{\mathfrak{X}}
\nc{\W}{\mathscr{W}}
\nc{\kt}{\mbf{k}}
\nc{\ko}{\mbf{k}(0)}
\nc{\Ok}{\mc{O}_G \boxtimes \kt_X}
\nc{\Oko}{\mc{O}_G \boxtimes \ko_X}
\nc{\OYk}{\mc{O}_Y \boxtimes \kt_X}
\nc{\id}{\msf{id}}
\nc{\A}{\mathbb{A}}
\nc{\Supp}{\mathrm{Supp}}
\nc{\Grat}{\mc{Grat}}
\nc{\Squo}[1]{\A^{(#1)}}
\nc{\twist}{\mathrm{twist}}
\nc{\Cd}{\mc{C}}
\nc{\Span}{\mathrm{Span}}
\nc{\Grass}{\mathrm{Gr}}
\nc{\Fr}{\mathrm{Fr}}
\nc{\pco}[1]{k[V]^{p\mathrm{co} #1}}
\nc{\Irr}{\mathsf{Irr} }
\renc{\o}{\otimes}
\renc{\gr}{\mathsf{gr}}
\nc{\U}{\mathsf{U}}
\nc{\algD}{\mf{D}}

\nc{\hr}{\mf{h}_{\textrm{reg}}}
\nc{\D}{\mathscr{D}}
\nc{\PIdeg}{\mathrm{P.I.-degree}}
\nc{\ch}{\mathrm{ch}}
\nc{\Fl}{\mathcal{F}\mathrm{l}}
\nc{\Stab}{\mathrm{Stab}}
\nc{\Der}{\mathrm{Der}}
\nc{\rightsim}{\stackrel{\sim}{\longrightarrow}}
\nc{\HZ}{H_{\mbf{h},\Z}(\Z_m)}
\nc{\sing}{\mathrm{sing}}
\nc{\dd}{\mathscr{D}}
\nc{\GKdim}{\mathrm{G.K. dim}}
\nc{\PIdegree}{\mathrm{P.I. degree}}
\nc{\Du}{\mathbb{D}}
\renc{\Fun}{\mathrm{Fun}}
\nc{\Xo}{\mathbb{X}}
\nc{\Cs}{\C^{\times}}
\nc{\dQ}{\overline{Q}}
\nc{\bM}{\mathbf{M}}
\nc{\bv}{\mathbf{v}}
\nc{\bw}{\mathbf{w}}
\nc{\Qv}{\mathfrak{M}}
\nc{\barOm}{\overline{\Omega}}
\nc{\Om}{\Omega}
\nc{\Lag}{\mathfrak{L}}
\nc{\free}{\mathrm{free}}
\nc{\Fix}{\mathfrak{F}}
\nc{\IC}{\mathrm{IC}}
\nc{\Ham}{\mathrm{Ham}}
\nc{\bnu}{\boldsymbol{\nu}}
\nc{\tX}{\mc{X}}
\nc{\tY}{\mc{Y}}
\nc{\reg}{\mathrm{reg}}
\nc{\Char}{\mathfrak{X}}
\nc{\Qo}{Q_0}
\nc{\Qn}{Q_1}
\nc{\ba}{\boldsymbol{\alpha}}
\nc{\Nak}[3]{\mf{M}_{{#1}} ({#2},{#3}) }
\nc{\Nakq}[4]{\mf{M}^{{#1}}_{{#2}} ({#3},{#4}) }
\renc{\a}{\alpha}
\nc{\hcf}[1]{\mathsf{hcf}({#1})}
\nc{\stable}{\mathrm{stable}}

\newcommand{\CharG}{\mathscr{X}}
\newcommand{\CharS}{\mathscr{Y}}
\newcommand{\PGL}{\mathrm{PGL}}
\newcommand{\G}{\mathrm{G}}
\newcommand{\PG}{\mathrm{PG}}
\newcommand{\Sp}{\mathrm{Sp}}

\begin{document}

\title{Symplectic resolutions of quiver varieties}

\begin{abstract}
  In this article, we consider Nakajima quiver varieties from the
  point of view of symplectic algebraic geometry. We prove that they
  are all symplectic singularities in the sense of Beauville and
  completely classify which admit symplectic resolutions.  Moreover we
  show that the smooth locus coincides with the locus of canonically
  $\theta$-polystable points, generalizing a result of Le Bruyn; we
  study their \'etale local structure and find their symplectic leaves.
  An interesting consequence
  of our results is that not all symplectic resolutions of quiver
  varieties appear to come from variation of GIT.  


\end{abstract}

\author{Gwyn Bellamy}

\address{School of Mathematics and Statistics, University of Glasgow, University Gardens, Glasgow, G12 8QW, UK}

\email{gwyn.bellamy@glasgow.ac.uk}

\author{Travis Schedler}

\address{Department of Mathematics, Imperial College London, South Kensington Campus, London, SW7 2AZ, UK}


\email{trasched@gmail.com}

\subjclass[2010]{16G20;17B63,14D25,58F05,16S80}

\keywords{symplectic resolution, quiver variety, Poisson variety}


\maketitle

\centerline{\it Dedicated, with admiration and thanks,}
\centerline{\it to Victor Ginzburg, on the occasion of his $60$th Birthday.}

$\hspace{50mm}$

\centerline{\bf Table of Contents}


$\hspace{50mm}$ {\parbox[t]{115mm}{
\hp${}_{}$\!\hp1.{ $\;\,$} {Introduction}\newline
\hp2.{ $\;\,$} {Quiver varieties}\newline
\hp3.{ $\;\,$} {Canonical Decompositions}\newline
\hp4.{ $\;\,$} {Smooth vs. stable points}\newline
\hp5.{ $\;\,$} {The $(2,2)$ case}\newline
\hp6.{ $\;\,$} {Factoriality of quiver varieties} \newline
\hp7.{ $\;\,$} {Namikawa's Weyl group}\newline
}

$\vspace{3mm}$

\section{Introduction}

Nakajima's quiver varieties \cite{Nak1994}, \cite{NakDuke98}, have
become ubiquitous throughout representation theory. For instance, they
play a key role in the categorification of representations of
Kac-Moody Lie algebras and the corresponding theory of canonical
bases. They provide \'etale-local models of singularities appearing in
important moduli spaces, together with, in most cases, a
canonical symplectic resolution given by varying the stability
parameter. They give global constructions of certain moduli spaces,
such as
resolutions
of du Val singularities \cite{Nak1994}, Hilbert schemes of points on them
\cite{Kuznetsov},
and
Uhlenbeck and Gieseker instanton moduli spaces \cite{Kronheimer-Nakajima,NakajimaBook,Nakajima-ALE-Quiver}.

Surprisingly, there seems to be no explicit criterion in the
literature for when a quiver variety admits a symplectic
resolution; often, in applications, suitable sufficient conditions for
their existence are provided, but they do not appear always to be necessary.
The main motivation of this article is to give such an explicit
criterion. Following arguments of Kaledin, Lehn and Sorger (who
consider the related case of moduli spaces of semistable
sheaves on a $K3$ or abelian surface), our classification
applies  Drezet's criteria to show that certain GIT quotients are locally factorial. To do so we undertake a careful study of the local and global algebraic symplectic geometry of quiver varieties.

Our classification begins by generalizing Crawley-Boevey's
decomposition theorem \cite{CBdecomp} of affine quiver varieties into
products of such varieties, which we will call \emph{indecomposable},
to the non-affine case; i.e., to quiver varieties with nonzero
stability condition (Theorem \ref{thm:main0}). Along the way, we also
generalize Le Bruyn's theorem, \cite[Theorem 3.2]{LeBCoadjoint}, which computes
the smooth locus of these varieties, again from the affine to
nonaffine setting (Theorem \ref{thm:stablesmooth}).

Then, our main result, Theorem \ref{thm:main1}, states that those
quiver varieties admitting resolutions are exactly those whose
indecomposable factors, as above, are one of the following types of
varieties:
\begin{itemize}
\item[(a)] Varieties whose dimension vectors are indivisible roots;
\item[(b)] Symmetric powers of deformations or partial resolutions of du
  Val singularities ($\C^2/\Gamma$ for $\Gamma < \SL_2(\C)$);
\item[(c)] Varieties whose dimension vector are twice a 
root whose Cartan pairing with itself is $-2$ (i.e., the
variety has 
  dimension ten).
\end{itemize}
Here, a dimension vector $\alpha \in \N^{I}$
is called \emph{indivisible} if $\gcd(\alpha_i) = 1$ for $i \in I$. 
The last type (c) is perhaps surprising: it is closely related to
O'Grady's examples \cite{OGr-K3,OGr-sixirr,KaledinLehn,LSOGrady}. In this case, one cannot fully
resolve or smoothly deform via a quiver variety, but after maximally
smoothing in this way, the remaining singularities are
\'etale-equivalent to the product of $V=\C^4$ with the locus of
square-zero matrices in $\mathfrak{sp}(V)$ (as considered in preceding articles). Via the partial Springer resolution \cite{BorhoMacPherson}, the latter is resolved by
the cotangent bundle of the Lagrangian Grassmannian of $V$.  As explained in \cite[Remark 4.6]{KaledinLehn}, \cite{LSOGrady}, this resolution can also
be obtained by
blowing up the reduced singular locus (once), which makes sense globally on the quiver variety.

In the case of type (a), one can resolve or deform by varying the
quiver (GIT) parameters. In fact, (for $\lambda = 0$) it is shown in \cite{BCSquiver} that all symplectic resolutions can be realised in this way. On the other hand, for quiver varieties of type (b), one cannot resolve
in this way, but the variety is well-known to be isomorphic to another
quiver variety (whose quiver is obtained by adding an additional
vertex, usually called a framing, and arrows from it to the other
vertices), which does admit a resolution via varying the parameters.
Moreover, in this case, if the stability parameter is chosen to lie in
the appropriate chamber, then the resulting resolution is a punctual
Hilbert scheme of the minimal resolution of the original du
Val singularity; see \cite{Kuznetsov}. The other chambers give in general different resolutions: in fact, thanks to \cite{BellamyCrawQuotient}, they again produce all symplectic resolutions of symmetric powers of du Val singularities.

\subsection{Symplectic resolutions} In order to state precisely our
main results, we will require some notation, which we will restate
in more detail in Section \ref{sec:quiver}.  Let $Q = (\Qo,\Qn)$ be
a quiver with finitely many vertices and arrows. We fix a dimension
vector $\alpha \in \N^{\Qo}$, deformation parameter
$\lambda \in \C^{\Qo}$, and stability parameter $\theta \in \Z^{\Qo}$,
such that $\lambda \cdot \alpha = \theta \cdot \alpha = 0$. Unless
otherwise stated, we make the following assumption throughout the
paper:
\begin{equation}\label{eq:assume}
\textrm{\textit{If $\theta \neq 0$ then $\lambda \in \R^{Q_0}$.}}
\end{equation}
Nakajima associated to this data the (generally singular) variety,
called a ``quiver variety.'' We briefly recall the definition; see
Section \ref{sec:quiver} for more details. Let $\Rep(Q,\a)$ be the
vector space of representations of $Q$ of dimension $\a$. The group
$\G(\alpha) := \prod_{i \in \Qo} \GL_{\a_i}(\C)$ acts on $\Rep(Q,\a)$;
write $\g(\alpha) = \Lie \G(\alpha)$. Then $\G(\alpha)$ also acts on
$T^*\Rep(Q,\a) \cong \Rep(\dQ,\a)$ with a moment map
$\mu: T^* \Rep(Q,\a) \to \g(\alpha)^* \cong \g(\alpha)$; here $\dQ$ is
the doubled quiver, obtained by adding reverse arrows to $Q$. To $\lambda \in \C^{\Qo}$ we can associate $(\lambda \operatorname{Id}_i)_{i \in \Qo} \in \g(\alpha)$. By abuse
of notation we will consider $\C^{\Qo}$ to be a subset of $\g(\alpha)$
in this way and write $\mu^{-1}(\lambda)$ for the fiber over
$(\lambda \operatorname{Id}_i)_{i \in \Qo}$.  Let
$\mu^{-1}(\lambda)^{\theta} \subseteq \mu^{-1}(\lambda)$ be the
$\theta$-semistable locus; this is the locus corresponding to
representations of $\dQ$ such that the dimension vector $\beta$ of
every subrepresentation satisfies $\theta \cdot \beta \leq 0$. Then
Nakajima defined the variety $\Nak\lambda\a\theta$ as:
$$
\Nak\lambda\a\theta := \mu^{-1}(\lambda)^{\theta} \git \, \G({\a}).
$$
It does not seem to be known  whether $\Nak\lambda\a\theta$, equipped with its natural scheme structure, is reduced (though we expect it is the case). Therefore, following Crawley-Boevey \cite{CBnormal}, we will consider \textit{throughout the paper} all quiver varieties as \textit{reduced schemes}.  

\begin{remark}\label{r:fr-unfr}
  The construction in \cite{Nak1994,NakDuke98} is apparently more
  general, depending on an additional dimension vector, called the
  framing.  However, as observed by Crawley-Boevey \cite{CBmomap},
  every framed variety can be identified with an unframed one.  In
  more detail, for the variety as in \cite{Nak1994,NakDuke98} with
  framing $\beta \in \N^{Q_0}$, it is observed in \cite[Section 1]{CBmomap}
  that the resulting variety can alternatively be constructed by
  replacing $Q$ by the new quiver $(\Qo \cup \{\infty\},
  \widetilde{\Qn})$, where $\widetilde{\Qn}$ consists of $\Qn$
  together with, for every $i \in \Qo$, $\beta_i$ new arrows from
  $\infty$ to $i$; then Nakajima's $\beta$-framed variety is the same
  as $\Nak{(\lambda,0)}{(\alpha,1)}{(\theta,-\alpha \cdot \theta)}$.
    Thus, for the purposes of the questions addressed in this article,
    it is sufficient to consider the unframed varieties.
\end{remark}
Let $R_{\lambda,\theta}^+$ denote those positive roots of $Q$ that
pair to zero with both $\lambda$ and $\theta$. If $\alpha \notin \N
R^+_{\lambda,\theta}$ then $\Nak\lambda\a\theta = \emptyset$,
therefore we assume $\a \in \N R^+_{\lambda,\theta}$. As defined by Beauville \cite{Beauville}, a normal variety $X$ is said to be a symplectic singularity if there
exists an (algebraic) symplectic $2$-form $\omega$ on the smooth locus
of $X$ such that $\pi^* \omega$ extends to a regular $2$-form on the
whole of $Y$, for any resolution of singularities $\pi : Y \rightarrow
X$. We say that $\pi$ is a symplectic resolution if $\pi^* \omega$
extends to a non-degenerate $2$-form on $Y$.  Note that \emph{a symplectic resolution does not always exist,} and when it does exist, it is not always unique.

\begin{thm}\label{thm:sympsing}
The variety $\Nak\lambda\a\theta$ is an irreducible symplectic singularity. 
\end{thm} 

This theorem is important because symplectic singularities have
become important in representation theory: on the one hand they
include many of the most important examples (aside from quiver
varieties, they include linear quotient singularities, nilpotent
cones, orbit closures, Slodowy slices, hypertoric varieties, and
so on), and on the other hand they exhibit important properties, at least in the
conical case, such as the existence of a nice universal family of
deformations \cite{KaledinVerbitsky,Namikawa,Namikawa2} and of quantizations
\cite{BezKaledinQuantization,BPWQuant1,LosevQuantOrbit}.

From both the representation theoretic and the geometric point of
view, it is important to know when the variety $\Nak\lambda\a\theta$
admits a symplectic resolution.  In this article, we address this
question, giving a complete answer. The first step is to reduce to the
case where $\alpha$ is a root for which there exists a $\theta$-stable point
in $\mu^{-1}(\lambda)$.
This is done via the \textit{canonical decomposition} of $\alpha$, as
described by Crawley-Boevey; it is analogous to Kac's canonical
decomposition. In this article, the term canonical decomposition will
only refer to the former, which we now recall.  Associated to
$\lambda,\theta$ is a combinatorially defined set
$\Sigma_{\lambda,\theta} \subset R_{\lambda,\theta}^+$; see Section
\ref{sec:quiver} below. Then $\a$ admits a canonical decomposition
\begin{equation}\label{eq:alphadecomp}
\alpha = n_1 \sigma^{(1)} + \cdots + n_k \sigma^{(k)}
\end{equation}
with $\sigma^{(i)} \in \Sigma_{\lambda,\theta}$ pairwise distinct, such that any other
decomposition of $\a$ into a sum of roots belonging to
$\Sigma_{\lambda,\theta}$ is a refinement of the decomposition
\eqref{eq:alphadecomp}. Closed points in $\mu^{-1}(\lambda)$ correspond to representations of the so-called deformed preprojective algebra $\Pi^{\lambda}(Q)$; see section~\ref{sec:notationpreproj} for details. Then points of $\Nak\lambda\a\theta$ are in bijection with isomorphism classes of $\theta$-polystable representations of $\Pi^{\lambda}(Q)$  (equivalently, representations of the doubled quiver of moment $\lambda$
 which decompose as direct sums of $\theta$-stable representations) of dimension $\alpha$.
 Generalizing \cite[Theorem 1.2]{CBmomap}, Proposition \ref{prop:diffeo} implies

\begin{thm}\label{t:stab-exists}
	There exists a $\theta$-stable representation of the deformed preprojective algebra $\Pi^{\lambda}(Q)$ of dimension $\a$ if and only if $\a \in \Sigma_{\lambda,\theta}$. 
\end{thm}

Crawley-Boevey's Decomposition Theorem \cite{CBdecomp}, which we will show holds in somewhat greater generality, then implies that the canonical decomposition gives a decomposition of the quiver variety as a product of varieties for each of the summands (the first statement of the next theorem). 
We show that 
 the question of existence of symplectic resolutions of $\Nak\lambda\a\theta$ can be reduced to the analogous question for each factor.   

\begin{thm}\label{thm:main0}
	With respect to the canonical decomposition \eqref{eq:alphadecomp}:
	\begin{enumerate}
		\item[(a)] The symplectic variety $\Nak\lambda\a\theta$ is isomorphic to
		$S^{n_1} \Nak\lambda{\sigma^{(1)}}\theta \times \cdots \times
		S^{n_k} \Nak\lambda{\sigma^{(k)}}\theta$.
		\item[(b)] $\Nak\lambda\a\theta$ admits a projective symplectic resolution if and only
		if each $\Nak\lambda{\sigma^{(i)}}\theta$ admits a projective symplectic
		resolution. 
	\end{enumerate}
\end{thm}

Here $S^n X$ denotes the $n$th symmetric product of $X$. 

To finish the classification, it suffices to describe the case
$\alpha \in \Sigma_{\lambda,\theta}$. Write $\gcd (\alpha)$ for the greatest common divisor of the integers $\{ \alpha_i \}_{i \in \Qo}$; it is divisible if $gcd(\alpha)> 1$, and otherwise indivisible. Let $p(\alpha) := 1-\frac{1}{2}(\alpha,\alpha)$ where $(-,-)$ is the Cartan pairing associated to the undirected graph underlying the quiver, i.e.,
$(e_i, e_j) = 2 - |\{a \in \Qn \, | \, a \colon i \to j \text{ or } a\colon j \to i\} |$, 
for elementary vectors $e_i, e_j$. As we will show below (in Corollary
\ref{cor:dimsympsing}), $2 p(\alpha) = \dim \Nak\lambda\a\theta$.
Finally, as we will recall in Section \ref{sec:quiver}, elements
$\alpha \in \Sigma_{\lambda,\theta}$ are divided into real roots (when
$p(\alpha) = 0$) and imaginary roots (when $p(\alpha) > 0$). The case
$p(\alpha)=1$ is particularly important and called \emph{isotropic},
since it means $(\alpha,\alpha)=0$. When $p(\alpha) > 0$ we say that $\alpha$ is \textit{anisotropic}.  Note that, when $\sigma^{(i)}$ is anisotropic in the canonical decomposition \eqref{eq:alphadecomp}, then $n_i=1$ (see Corollary \ref{c:ni=1} below).

Our main theorem is then:
\begin{thm}\label{thm:main1}
  Let $\alpha \in \Sigma_{\lambda,\theta}$.
  Then 
  $\Nak\lambda\a\theta$ admits a projective symplectic resolution if and only if
  $\alpha$ is indivisible or
%
  $\left( \gcd(\a),p\left( \gcd(\a)^{-1} \a \right) \right) = (2,2)$.
\end{thm}
The latter case in the theorem will be referred to as ``the $(2,2)$ case''.

If $\alpha \in \Sigma_{\lambda,\theta}$ is
indivisible and anisotropic,
then a projective symplectic resolution
of $\Nak\lambda\a\theta$ is given by moving $\theta$ to a generic
stability parameter. However, this fails in the $(2,2)$ case.
It seems unlikely that $\Nak\lambda\a\theta$ can be resolved by another quiver variety in this case. Instead, we show that the $10$-dimensional symplectic
singularity $\Nak\lambda\a\theta$ can be resolved by blowing up the singular locus. We will need the partial ordering $\ge$ on stability conditions,
where $\theta' \ge \theta$ if every $\theta'$-semistable representation
is $\theta$-semistable; see Section \ref{sec:stab} below.

\begin{thm}\label{thm:blowup22Q} Let $\alpha \in \Sigma_{\lambda,\theta}$, and
  suppose   $\left( \gcd(\a),p\left( \gcd(\a)^{-1} \a \right) \right) = (2,2)$.
Let $\theta'$ be a generic stability parameter such that $\theta' \ge \theta$.
 If $\widetilde{\mathfrak{M}}_{\lambda}(\alpha,\theta')$ is the blowup
  of $\Nak\lambda{\alpha}{\theta'}$ along 
the reduced singular locus, then the canonical
  morphism $\pi : \widetilde{\mathfrak{M}}_{\lambda}(\alpha,\theta')
  \rightarrow \Nak\lambda{\alpha}{\theta}$ is a projective symplectic
  resolution of singularities.
\end{thm} 



In most cases where a projective symplectic resolution does not exist,
we can prove that neither does a proper one exist (note that every projective resolution is proper but not conversely). We say that $\a \in \Sigma_{\lambda,\theta}$ is ``$\Sigma$-divisible'' if $\a = m\beta$ for $m \geq 2$ and $\beta \in \Sigma_{\lambda,\theta}$. This is a slightly stronger condition than being divisible, although they coincide in most cases: see Theorem \ref{thm:Sigmadivisiblevsdivisible} below.
\begin{thm}\label{thm:proper}
  If $\alpha \in \Sigma_{\lambda,\theta}$ is $\Sigma$-divisible, and $\left( \gcd(\a),p\left( \gcd(\a)^{-1} \a \right) \right) \neq (2,2)$, then $\Nak\lambda\a\theta$ does not admit a proper symplectic resolution.

  Moreover, for general $\alpha \in \N^{Q_0}$, if there exists a $\Sigma$-divisible factor $\sigma^{(i)}$ in the decomposition \eqref{eq:alphadecomp} that satisfies $\left( \gcd(\sigma^{(i)}),p\left( \gcd(\sigma^{(i)})^{-1} \sigma^{(i)} \right) \right) \neq (2,2)$, then $\Nak\lambda\a\theta$ does not admit a proper symplectic resolution.  
  \end{thm}
\begin{remark}\label{r:proper-proj}
  Most of the literature deals with projective rather than proper
  resolutions.  However, there are interesting examples of proper
  symplectic resolutions that are not projective. For example,
  in \cite{ArPr-ProperHypertoric} such examples are constructed admitting
  Hamiltonian torus actions of maximal dimension (this condition is
  called hypertoric there, which generalizes the usual definition of
  hypertoric variety).
  It seems to be an interesting question if, whenever a proper
  symplectic resolution exists, also a projective symplectic
  resolution exists. More generally, it seems reasonable to ask
  whether, if a proper symplectic resolution exists, then every
  proper $\Q$-factorial terminalization is symplectic; if we restrict to
  projective resolutions and terminalizations, then the proof of
  \cite[Theorem 5.5]{Namikawa} shows that this holds at least when the
  singularity is conical with homogeneous generic symplectic form.
\end{remark}

\subsection{Symplectic leaves and the \'etale local structure}

As Hamiltonian reductions, quiver varieties have a natural Poisson structure. The symplectic leaves of this Poisson structure are the maximal connected (analytic immersed) submanifolds on which the Poisson bracket is non-degenerate. Put differently, 
the reduction naturally is foliated by symplectic submanifolds.   For example, the locus of stable representations of the doubled quiver inside $\mu^{-1}(\lambda)$ consists of free closed orbits under the group $\PG(\alpha):=\G(\alpha)/\Cs$, hence its Hamiltonian reduction here is well known to be symplectic. If nonempty, this forms an open dense symplectic leaf of the quiver variety.

Since we have shown that quiver varieties have symplectic singularities, thanks to \cite[Theorem 2.3]{Kaledinsympsingularities},  they must necessarily be a finite union of symplectic leaves and the latter are algebraic. It has long been assumed that the leaves are precisely the strata $\Nak{\lambda}{\a}{\theta}_{\tau}$  given by the representation-type stratification. Here $\tau$ is a decomposition of $\alpha$ in $\Sigma_{\lambda,\theta}$. Since this explicit identification of the symplectic leaves is crucial later in the article, we provide a complete proof that this is indeed the case.

\begin{thm}\label{thm:symplecticleavesarerepstrata}
	The symplectic leaves of $\Nak{\lambda}{\a}{\theta}$ are the representation-type strata $\Nak{\lambda}{\a}{\theta}_{\tau}$ . 
\end{thm} 

This result follows from Proposition~\ref{prop:strata} and Corollary~\ref{cor:strataNakaconnected}. The classification of symplectic leaves already appears in \cite{MarsdenWeinsteinStratification}, however there seems to be a gap in the proof given there; see remark~\ref{rem:Martinoconnectedleaves}. 

Theorem~\ref{thm:symplecticleavesarerepstrata} allows us to give a combinatorial classification, in Corollary~\ref{cor:smoothminimal} below, of those quiver varieties that are smooth.   

An important tool in both the proof of Theorem~\ref{thm:symplecticleavesarerepstrata} and later results on the factoriality of quiver varieties is an \'etale local description of the varieties. In the case of trivial stability parameter $\theta = 0$, this \'etale local picture was described by Crawley-Boevey in \cite{CBnormal}, where it was used to prove that those quiver varieties are normal. In section~\ref{sec:etale} we show that this \'etale local description holds for all stability parameters. See Theorem~\ref{thm:etalelocalgeneral} for the precise statement.     

A relative version of Theorem~\ref{thm:etalelocalgeneral}, proving an \'etale local description of the morphism $\mf{M}(\a,\theta) \to \mf{M}(\a,\theta')$ is given in \cite{BellamyCrawQuotient}. This result allows the authors to completely classify (in the case of a framed affine Dynkin quiver) those walls in the space of stability parameters that are flops; resp. are divisorial contractions. This is a key step in showing (as mentioned above) that all symplectic resolutions of symmetric powers of du Val singularities are given by variation of GIT. 

\subsection{Factoriality of quiver varieties} 
The real difficulty in the proof of Theorem \ref{thm:main1} is in
showing that if $\a \in \Sigma_{\lambda,\theta}$ is
divisible and anisotropic,
\begin{equation}\label{eq:not22pn}
\left( \gcd(\a),p\left(\gcd(\a)^{-1} \a \right) \right) \neq (2,2),
\end{equation}
then $\Nak{\lambda}{\a}{\theta}$ does not admit a projective symplectic
resolution. Based upon a result of Drezet \cite{Drezet}, who
considered instead the moduli space of semistable sheaves on a
rational surface, we show in Corollary \ref{cor:factorial} the following result. Recall that a variety is locally factorial if all of its local rings are unique factorization domains.
\begin{thm}\label{t:fact} 
  Assume that $\a \in \Sigma_{\lambda,\theta}$ is an anisotropic root satisfying condition \eqref{eq:not22pn}, and that
  $\theta$ is generic. Then the quiver variety
  $\Nak{\lambda}{\a}{\theta}$ is locally factorial.
\end{thm}
Observe that we did not require $\a$ to be
divisible, although if were
indivisible then we already noted that $\Nak\lambda\a\theta$ is
smooth for generic $\theta$. On the other hand, in the
divisible case,
we will see that, for $\theta$ generic, the variety
$\Nak{\lambda}{\a}{\theta}$ has terminal singularities, using that, by
\cite{NamikawaNote}, this is equivalent to having singularities in
codimension at least four.
Therefore, by a
well-known fact, the above theorem implies that it cannot admit a
proper symplectic resolution.

In fact, we prove 
in Corollary \ref{cor:factorial} 
a more precise statement than Theorem \ref{t:fact}
which does not require that
$\theta$ be generic.
By the argument given in the proof of Theorem \ref{thm:nonisonores},
we see that the corollary implies that this statement holds for open
subsets of $\Nak{\lambda}{\a}{\theta}$. Therefore we conclude the
following strengthening of the nonexistence direction of Theorem
\ref{thm:main1}:
\begin{cor}\label{cor:openfact}
  Assume that $\a \in \Sigma_{\lambda,\theta}$ is
  divisible,
it satisfies condition \eqref{eq:not22pn},
  and $\theta$ is generic.  Under the assumptions of Theorem
  \ref{t:fact}, if $U \subseteq \Nak{\lambda}{\a}{\theta}$ is any
  singular Zariski open subset, then $U$ does not admit a proper symplectic
  resolution.
\end{cor}
In fact, by Corollary \ref{cor:factorial} below, we can drop in
Corollary \ref{cor:openfact} the assumption that $\theta$ is generic,
at
the price of replacing $\Nak{\lambda}{\a}{\theta}$ by a certain
canonical open set: the locus of direct sums of stable representations
of dimension vector proportional to $\alpha$.

In particular, in many cases, there are open subsets
$U \subseteq \Nak{\lambda}{\a}{\theta}$ which formally locally admit
symplectic resolutions everywhere, but do not admit one globally.  For
example, if $\a = 2\beta$ for some $\beta \in \Sigma_{\lambda,\theta}$
with $p(\beta) \geq 3$ (cf.~the definition of $p$ above Theorem
\ref{thm:main1}), then we can let $U$ be the locus of representations
which are either stable or decompose as $X = Y \oplus Y'$ for $Y,Y'$
\emph{nonisomorphic} $\theta$-stable representations of dimension
vectors equal to $\beta$.

There is one quiver variety in particular that captures the
``unresolvable'' singularities of $\Nak{\lambda}{\a}{\theta}$. This
variety, which we denote $\Char(g,n)$ with $g,n \in \N$, has been
studied in the works of Lehn, Kaledin and
Sorger. Concretely,
$$ \Char(g,n) := \left\{ (X_1,Y_1, \ds, X_g,Y_g) \in \End_{\C}(\C^n)
\ \left| \ \sum_{i = 1}^d [X_i,Y_i] = 0 \right. \right\} \git 
\GL(n,\C),
$$
Viewed as a special case of Corollary \ref{cor:factorial}, we see that
$\Char(g,n)$ does not admit a proper symplectic resolution if $g, n \ge 2$ and $(g,n) \neq (2,2)$.


When $g = 1$, the Hilbert scheme of $n$ points in the plane provides a
symplectic resolution of $\Char(g,n) \simeq S^n \C^2$; see
\cite[Theorem 1.2.1, Lemma 2.8.3]{AlmostCommutingVariety}. When
$n = 1$, one has $\Char(g,n) \simeq \C^{2g}$.

\begin{remark}
  It is interesting to note that \cite[Theorem 1.1]{CBmomap} implies
  that the moment map
$$
 (X_1,Y_1, \ds, X_g,Y_g) \mapsto \sum_{i = 1}^g [X_i,Y_i]
$$
is flat when $g > 1$, in contrast to the case $g = 1$, which is easily
seen not to be flat.
\end{remark}

\begin{remark}
  Generalizing the Geiseker moduli spaces that arise from framings of
  the Jordan quiver, it seems likely that the framed versions of
  $\Char(g,n)$, which are smooth for generic stability parameters,  should have interesting combinatorial and representation theoretic properties.
\end{remark}

\begin{remark}\label{r:use-formal}
  One does not need the full strength of Theorem \ref{t:fact} to prove
  that $\Nak{\lambda}{\a}{\theta}$ does not admit a symplectic
  resolution: it suffices to show that a formal neighborhood of some
  point does not admit a symplectic resolution.  This reduces the
  problem to the one-vertex case, i.e., to $\Char(g,n)$. However, the
  techniques (following \cite{KaledinLehnSorger}) do not actually
  simplify in this case.
 Moreover, this would not be enough to imply Corollary
  \ref{cor:openfact}.
\end{remark}

\subsection{Smooth versus canonically polystable points} In order to
decide when the variety $\Nak\lambda\a\theta$ is smooth, we describe
the smooth locus in terms of $\theta$-stable representations. Write
the canonical decomposition
$n_1 \sigma^{(1)} + \cdots + n_k \sigma^{(k)}$ of
$\alpha \in \N R^+_{\lambda,\theta}$ as
$\beta^{(1)} + \cdots + \beta^{(\ell)}$, where a given
$\beta \in \Sigma_{\lambda,\theta}$ may appear multiple times. Recall
that a representation is said to be $\theta$-polystable if it is a
direct sum of $\theta$-stable representations. We say that a
representation $x$ is \textit{canonically $\theta$-polystable} if
$x = x_1 \oplus \cdots \oplus x_{\ell}$ where each $x_i$ is
$\theta$-stable, $\dim x_i = \beta^{(i)}$ and $x_i \not\simeq x_j$ for
$i \neq j$, unless $\beta^{(i)} = \beta^{(j)}$ is a real root, i.e.,
$p(\beta^{(i)}) = 0$. Observe that the notion of canonical $\theta$-polystability reduces to $\theta$-stability precisely in the case that $\a \in \Sigma_{\lambda,\theta}$. In general, the set of points of $\Nak\lambda\a\theta$ which
are the image of canonically $\theta$-polystable representations is a
dense open subset. When $\theta = 0$, the result below is due to Le
Bruyn \cite[Theorem 3.2]{LeBCoadjoint} (whose arguments we
generalize).
 
\begin{thm}\label{thm:stablesmooth}
  A point $x \in \Nak\lambda\a\theta$ belongs to the smooth locus if
  and only if it is canonically $\theta$-polystable.
\end{thm}

\begin{rem}
  Theorem \ref{thm:stablesmooth} confirms the expectation stated after
  Lemma 4.4 of \cite{GordonQuiver}.
\end{rem}

An element $\sigma \in \Sigma_{\lambda,\theta}$ is said to be
\textit{minimal} if there are no
$\beta^{(1)}, \ds , \beta^{(r)} \in \Sigma_{\lambda,\theta}$, with
$r \ge 2$, such that $\sigma = \beta^{(1)} + \cdots + \beta^{(r)}$.

\begin{cor}\label{cor:smoothminimal}
The variety $\Nak\lambda\a\theta$ is smooth if, and only if, in the canonical decomposition $\alpha = n_1 \sigma^{(1)} + \cdots + n_k \sigma^{(k)}$ of $\alpha$, each $\sigma^{(i)}$ is minimal, and the multiplicity $n_i$ is one if $\sigma^{(i)}$ is isotropic.
\end{cor}
Since, as recalled in Corollary \ref{c:ni=1} below, $n_i$ is always one if $\sigma^{(i)}$ is aniostropic, we could equivalently drop the assumption ``is isotropic'' at the end of the corollary.

Corollary~\ref{cor:smoothminimal} is a crucial ingredient in the proof of the main result of \cite{BellamyCrawQuotient}, where a key step is the classification of stability parameters for which the corresponding quiver variety (associate to a framed affine Dynkin quiver) is smooth.

\subsection{Namikawa's Weyl group} When both $\lambda$ and $\theta$ are zero, $\Nak{0}{\a}{0}$ is an affine conic  symplectic singularity.  Associated to $\Nak{0}{\a}{0}$ is \emph{Namikawa's Weyl group} $W$ \cite{Namikawa2}, a finite reflection group. In order to compute $W$, one needs to describe the codimension two symplectic leaves of $\Nak{0}{\a}{0}$. More generally, we consider the codimension two leaves in a general quiver variety $\Nak{\lambda}{\a}{\theta}$. It is enough by Crawley--Boevey's canonical decomposition to consider the case $\a \in \Sigma_{\lambda,\theta}$. We show that the codimension two symplectic leaves are parameterized by \textit{isotropic decompositions} of $\a$. 

\begin{defn}\label{defn:isodecomp} 
The decomposition $\alpha = \beta^{(1)} + \cdots + \beta^{(s)} + m_1
\gamma^{(1)} + \cdots m_t \gamma^{(t)}$ is said to be an isotropic decomposition if
\begin{enumerate}
\item[(a)] $\beta^{(i)}, \gamma^{(j)} \in \Sigma_{\lambda,\theta}$. 
\item[(b)] The $\beta^{(i)}$ are  imaginary roots.
\item[(c)] The $\gamma^{(i)}$ are \textit{pairwise distinct} real roots.
\item[(d)] If $\overline{Q}''$ is the quiver with $s + t$ vertices without
  loops and $-(\alpha^{(i)}, \alpha^{(j)})$ arrows from vertex $i$ to vertex $j \neq i$, where $\alpha^{(i)},\alpha^{(j)} \in \{ \beta^{(1)}, \ds,
  \beta^{(s)}, \gamma^{(1)} , \ds , \gamma^{(t)} \}$, then $Q''$ is an
  affine Dynkin quiver.
\item[(e)] The dimension vector $(1,\ds, 1, m_1, \ds, m_t)$ of $Q''$ (where there are $s$ ones) equals $\delta$, the minimal imaginary root. 
\end{enumerate}
\end{defn}
\begin{remark}
  In fact, as we will show in Lemma \ref{l:codim2} below, in an isotropic decomposition of
  $\a \in \Sigma_{\lambda,\theta}$, all of the anisotropic
  $\beta^{(i)}$ are pairwise distinct, except\footnote{Thanks to Jasper van de Kreeke for pointing out this exceptional case.} in the case where $Q''$ is of affine type $A_1$, so that $s=2$, $t=0$, and $\beta:=\beta^{(1)}=\beta^{(2)}$ has self pairing $(\beta,\beta)=-2$.  This may help in finding these
  decompositions.

However, the isotropic $\beta^{(i)}$ need not be distinct.  As an
example, when $Q$ is the quiver with two vertices $1,2$ and two
arrows, one loop at $1$ and one arrow from $1$ to $2$, then we can
take $\a=(4,2)$,
$\beta^{(1)}=(1,0)=\beta^{(2)}=\beta^{(3)}=\beta^{(4)}$, and
$\gamma^{(1)}=(0,1)$. Then $p(\a)=5$ and $\a \in \Sigma_{0,0}$, and
the quiver $\overline{Q}''$ is of affine $D_4$ type with central
vertex corresponding to $\gamma^{(1)}$ and external vertices
corresponding to the $\beta^{(i)}$.  This example is also interesting since $\a \in \Sigma_{0,0}$ is divisible, but not $\Sigma$-divisible (as $\frac{1}{2} \a \notin \Sigma_{0,0}$).
\end{remark}
Given an isotropic decomposition with affine Dynkin quiver $Q''$, let $Q''_f$ be the finite part, which is a Dynkin diagram.

\begin{thm}\label{t:codim-two-strata}
  Let $\alpha \in \Sigma_{\lambda,\theta}$ be imaginary. Then the
  codimension two strata of $\Nak{\lambda}{\a}{\theta}$ are in
  bijection with the isotropic decompositions of $\alpha$. The
  singularity along each such stratum is \'etale-equivalent to the du
  Val singularity of the type $A_n, D_n, E_n$ corresponding to $Q''_f$.
\end{thm}
As a consequence, for $\lambda=0=\theta$, by \cite[Theorem 1.1]{Namikawa2} the Namikawa Weyl
group is a product over all isotropic decompositions $B$ of a group
$W_B$.  This group $W_B$ is either the Weyl group of the corresponding
Dynkin diagram $Q''_f$, or else the centralizer therein of an
automorphism of this diagram, corresponding to the monodromy around
the fiber over a point of the stratum under a crepant resolution of
the complement of the codimension $>2$ strata.

\subsection{Character varieties}\label{sec:charintro}

The methods we use seem to be applicable to many other
situations. Indeed, as we have noted previously, they were first
developed by Kaledin-Lehn-Sorger in the context of semistable sheaves
on a $K3$ or abelian surface. Any situation where the symplectic
singularity is constructed as a Hamiltonian reduction with respect to
a reductive group of type $A$ is amenable to this sort of
analysis. One such situation, which is of crucial importance
throughout geometry, topology, and group theory, is that of character
varieties of a Riemannian surface.

Let $\Sigma$ be a compact Riemannian surface of genus $g > 0$ and $\pi$ its fundamental group. The \textit{$\SL$-character variety} of $\Sigma$ is the affine quotient 
$$
\CharS(g,n) := \Hom(\pi,\SL(n,\C)) \git \SL(n,\C). 
$$
Similarly, the $\GL$-character variety is
$$
\CharG(g,n) =  \Hom(\pi,\GL(n,\C)) \git \GL(n,\C).
$$
In the article \cite{BSCharacter} we show that $\CharG(g,n)$ and $\CharS(g,n)$ are irreducible symplectic singularities. Moreover, we show:

\begin{thm}\label{thm:slsymp}\cite{BSCharacter}
	Assume that $g > 1$ and $(g,n) \neq (2,2)$. Then the varieties $\CharG(g,n)$ and $\CharS(g,n)$ are locally factorial with terminal singularities and hence do not admit proper symplectic resolutions. The same holds for any singular open subset. 
\end{thm}







Another very similar situation is that of moduli spaces of Higgs bundles on a genus $g$ curve. The symplectic singularities of these moduli spaces are considered by A. Tirelli in \cite{TirelliHiggs}. 

\subsection{Applications}

In joint work with A. Craw, the first author studies the symplectic resolutions of the symplectic quotient singularities $\C^{2n} / (\s_n \wr \Gamma)$, where $\Gamma \subset \mathrm{SL}(2,\C)$ is a finite group and $\s_n \wr \Gamma = \Gamma^n \rtimes \s_n$ is the associated wreath product. It is well-known that $\C^{2n} / (\s_n \wr \Gamma)$ is a quiver variety and symplectic resolutions of the quotient singularity can be realised using variation of GIT for quiver varieties. Using the results from this article, it is shown in \cite{BellamyCrawQuotient} that in fact all projective symplectic resolutions of the quotient singularity can be realised using quiver varieties. Moreover, one can say when stability parameters lying in different chambers give rise to the same symplectic resolutions. To prove these statements, it is crucial to have (a) the characterization of smooth quiver varieties given by Corollary \ref{cor:smoothminimal} (b) the classification of symplectic leaves given in Corollary \ref{cor:strataNakaconnected}; and (c) the local normal form given by Theorem \ref{thm:etalelocalgeneral}.   

More generally, in joint work \cite{BCSquiver} with A. Craw, we use results of this paper to give a complete classification of
projective symplectic resolutions of quiver varieties. 

In \cite{Craw-Hilbert}, the authors prove that symmetric powers of minimal resolutions of du Val singularities are also quiver varieties, for non-generic stability parameters on affine Dynkin quivers.  By Theorem \ref{thm:sympsing}, this implies that they are symplectic singularities. Our results in Section \ref{sec:hypertwist} are also employed in the proof of their main result.

\subsection{Other related work}

In joint work \cite{SchedlerTirelli} with Tirelli, the second author has used similar methods to give a classification of those multiplicative quiver varieties and character varieties of open Riemann surfaces that admit symplectic resolutions. Though the methods are similar, the situation considered in \cite{SchedlerTirelli} is considerably more complex that the additive case considered here (owing, for example, to the fact that it is unknown there when the varieties in question are nonempty).

Our classification explains which quiver varieties fall under the
general framework of Springer theory as recently developed by
McGerty--Nevins \cite{McGertyNevinsGalois}. Additionally, similar
questions to ours are analyzed there in greater detail for the Dynkin cases.
 
\subsection{Acknowledgments}

The first author was partially supported by EPSRC grant EP/N005058/1.
The second author was partially supported by NSF Grant DMS-1406553.
The authors are grateful to the University of Glasgow for the
hospitality provided during the workshop ``Symplectic representation
theory'', where part of this work was done, and the second author to
the 2015 Park City Mathematics Institute as well as to the Max Planck
Institute for Mathematics for excellent working environments.

We would like to thank
Yoshinori Namikawa and Nick
Proudfoot for useful discussions. We also thank Andrea Tirelli for
careful reading and useful comments, and Ben Davison and Johannes
Nicaise for useful comments on Tirelli's thesis, dealing with related
material.

We thank the referee for a careful reading of the paper and helpful suggestions.

\subsection{Notation and proof of the main results}

Throughout, a variety will mean a reduced, quasi-projective scheme of finite type over $\C$. If $X$ is a (quasi-projective) variety equipped with the action of a reductive algebraic group $G$, then $X \git \, G$ will denote the good quotient (when it exists). In this case, let $\xi: X \rightarrow X \git \, G$ denote the quotient map. Then each fibre $\xi^{-1}(x)$ contains a unique closed $G$-orbit. Following Luna, this closed orbit is denoted $T(x)$.  

The proof of the theorems and corollaries stated in the introduction can be found in the following sections. 
\begin{center}
\begin{tabular}{lcl}
Theorem \ref{thm:sympsing} & : & Section \ref{sec:thmsympsing} \\
Theorem \ref{thm:main0} & : & Section \ref{sec:Thm01proof} \\
Theorem \ref{thm:main1} & : & Section \ref{sec:Thm01proof} \\
Theorem \ref{thm:blowup22Q} & : & Section \ref{sec:blowup22Q-pf} \\
Theorem \ref{t:fact} & : & Section \ref{sec:localfact62} \\
Corollary \ref{cor:openfact} & : & Section \ref{sec:Thm01proof} \\ 
Theorem \ref{thm:stablesmooth} & : & Section \ref{sec:stablesmooth-pf} \\ 
Corollary \ref{cor:smoothminimal} & : & Section \ref{sec:smoothminimal-pf} \\
Theorem \ref{t:codim-two-strata} & : & Section \ref{sec:codim-two-strata-pf} \\
\end{tabular}
\end{center}




\section{Quiver varieties}\label{sec:quiver}

In this section we fix notation.  

\subsection{Notation}\label{sec:notationpreproj} Let $\N := \Z_{\geq 0}$. We work over $\C$ throughout. All quivers considered will have a finite number of vertices and arrows. We \tit{allow} $Q$ to have loops at vertices. Let $Q = (\Qo,\Qn)$ be a quiver, where $\Qo$ denotes the set of vertices and $\Qn$ denotes the set of arrows. Given $a \in \Qn$, let $a_s, a_t \in \Qo$ be the source and target, so $a: a_s \to a_t$.
For a dimension vector $\alpha \in \N^{\Qo}$, $\Rep(Q,\a) := \prod_{a \in \Qn} \Hom(\C^{\a_{a_s}}, \C^{\a_{a_t}})$ denotes the vector space of representations of $Q$ of dimension $\a$. The group $\G(\alpha) := \prod_{i \in \Qo} GL_{\a_i}(\C)$ acts on $\Rep(Q,\a)$; write $\g(\alpha) = \Lie \G(\alpha)$. The torus $\Cs$ in $\G(\a)$ of diagonal matrices acts trivially on $\Rep(Q,\a)$. Thus, the action factors through $\PG(\a) := \G(\a) / \Cs$. Let $\pg(\a) := \Lie \PG(\a) = \g(\a)/\C$.

Let $\dQ$ be the doubled quiver of $Q$, where for each arrow $a \colon i \to j$ of $Q$ we add a reverse arrow $a^* \colon j \to i$ to form $\dQ$. There is a natural identification $T^* \Rep(Q,\a) = \Rep (\dQ, \alpha)$. The group $\G(\a)$ acts symplectically on $\Rep (\dQ, \alpha)$ and the corresponding moment map is $\mu : \Rep (\dQ, \alpha) \rightarrow \g(\a)$, where we have identified $\g(\a)$ with its dual using the trace form. An element $\lambda \in \C^{Q_0}$ is identified with the tuple of scalar matrices $( \lambda_i \mathrm{Id}_{V_i})_{i \in Q_0} \in \g(\a)$. The affine quotient $\mu^{-1}(\lambda) / \! / \G(\a)$ parameterizes semi-simple representations of the \textit{deformed preprojective algebra} $\Pi^{\lambda}(Q) := \C \dQ / (\sum_{a \in Q_1} (a a^* - a^* a) - \sum_{i \in Q_0} \lambda_i p_i)$, where $p_i$ is the length-zero path at the vertex $i$. See \cite{CBmomap} for details. 

If $M$ is a finite dimensional $\Pi^{\lambda}(Q)$-module, then $\dim M$ will always denote the dimension \emph{vector} of $M$, and not just its total dimension. 

\subsection{Root systems} The coordinate vector at vertex $i$ is denoted $e_i$. The set $\N^{\Qo}$ of dimension vectors is partially ordered by
$\alpha \ge \beta$ if $\alpha_i \ge \beta_i$ for all $i$ and we say that
$\alpha > \beta$ if $\alpha \ge \beta$ with $\alpha \neq
\beta$. The \emph{support} of the vector $\a$ is the subquiver of $Q$ obtained by deleting all vertices $i \in Q_0$ where $\a_i = 0$. Following \cite[Section 8]{CBnormal},  $\alpha$ is called \emph{sincere} if $\alpha_i > 0$ for all $i$ i.e. the support of $\a$ equals $Q$. The Ringel form
on $\Z^{Q_0}$ is defined by
$$
\langle \alpha, \beta \rangle = \sum_{i \in Q_0} \alpha_i \beta_i - \sum_{a \in Q_1} \alpha_{t(a)} \beta_{h(a)}.
$$
Let $(\alpha,\beta) = \langle \alpha, \beta \rangle + \langle \beta,
\alpha \rangle$ denote the corresponding Euler form and set $p(\alpha) = 1 -
\langle \alpha, \a \rangle$. The fundamental region $\mc{F}(Q)$ is the set of $0 \neq \alpha \in
\N^{\Qo}$ with connected support and with $(\alpha, e_i) \le 0$ for all
$i$.

If $i$ is a loopfree vertex, so $p(e_i) = 0$, there is a reflection
$s_i : \Z^{\Qo} \ra \Z^{\Qo}$ defined by
$s_i \alpha = \alpha - (\alpha,e_i)e_i$. 
There is also the dual reflection, $r_i: \Z^{\Qo} \to \Z^{\Qo}, (r_i \lambda)_j = \lambda_j - (e_i, e_j) \lambda_i$.
The real roots (respectively
imaginary roots) are the elements of $\Z^{\Qo}$ which can be obtained
from the coordinate vector at a loopfree vertex (respectively $\pm$ an
element of the fundamental region) by applying some sequence of
reflections at loopfree vertices.  Let $R^+$ denote the set of
positive roots. Recall that a root $\beta$ is \textit{isotropic}
imaginary if $p(\beta) = 1$ (i.e., $(\beta,\beta)=0$) and
\textit{anisotropic} imaginary if $p(\beta) > 1$. Abusing terminology slightly, we will simply say that a root $\a$ is (a) real if $p(\alpha) = 0$, (b) isotropic if $p(\a) = 1$, and (c) anisotropic if $p(\a) > 1$. 


\subsection{The canonical decomposition}\label{sec:candec}

In this section we recall the \textit{canonical decomposition} defined by Crawley-Boevey (not to be confused with Kac's canonical decomposition). Fix $\lambda \in \C^{\Qo}$ and $\theta \in \Z^{Q_0}$. Then $R_{\lambda,\theta}^+ := \{ \alpha \in R^+ \ |
\ \lambda \cdot \alpha = \theta \cdot \a = 0 \}$. Following \cite{CBmomap}, we define
$$
\Sigma_{\lambda,\theta} = \left\{ \alpha \in R_{\lambda,\theta}^+ \ \left| \ p(\alpha) > \sum_{i
  = 1}^r p \left( \beta^{(i)} \right) \textrm{ for any decomposition } \right. \right.
$$
$$
\hspace{65mm} \left. \alpha =
\beta^{(1)} + \ds + \beta^{(r)} \textrm{ with } r \ge 2, \ \beta^{(i)}
\in R_{\lambda,\theta}^+ \right\}.
$$

\begin{example}\label{ex:realrootdecomp} Suppose that $\lambda=0=\theta$ and $\alpha \in \Sigma_{\lambda,\theta}$ is real, i.e., $p(\alpha)=0$. Then $\alpha$ is a coordinate vector. Indeed, if not, by definition there is a vertex $i \in \Qo$ such that $\alpha = s_i \alpha + ke_i$ with $k \geq 1$. Then $0 = p(\alpha) = p(s_i\alpha) + kp(e_i)$ contradicts the fact that $\alpha \in \Sigma_{0,0}$.
\end{example}
\begin{example} \label{ex:sigma00-ad} Again suppose that
  $\lambda=0=\theta$, and now assume that
  $\alpha \in \Sigma_{\lambda,\theta}$ is isotropic i.e.,
  $p(\alpha)=1$.  Then as observed in the proof of \cite[Proposition
  1.2.(2)]{CBdecomp}, $\alpha$ is supported on an affine Dynkin
  subquiver and there is the minimal imaginary root.  We repeat the
  argument for the reader's convenience. First, $\alpha$ is
  indivisible, since $\alpha=k\beta$ would imply
  $p(\alpha) < kp(\beta)$, and as $\beta$ is also a root, this
  contradicts the assumption $\alpha \in \Sigma_{0,0}$.  Next,
  $\alpha$ is in the fundamental region, since otherwise
  $\alpha = s_i \alpha + k e_i$ for some $i \in \Qo$ and $k \geq 1$,
  which implies $1=p(\alpha) = p(s_i \alpha) + kp(e_i)$, again
  contradicting the assumption that $\alpha \in \Sigma_{0,0}$.  Now
  the support of $\alpha$ is connected. Letting $Q'$ be its supporting
  quiver (i.e., the result of discarding all vertices not in the
  support and all incident arrows), we obtain a connected quiver for
  which $\alpha$ is in the kernel of the Cartan pairing.  By
  \cite[Lemma 1.9.(d)]{KacThm}, $Q'$ is affine (ADE) Dynkin and $\alpha$
  is an imaginary root. Since it is also
  indivisible, it is the
  minimal imaginary root $\delta$ of $Q'$.
\end{example}

In several places below, we choose a parameter $\nu \in \C^{Q_0}$ such that $R_{\lambda,\theta}^+ = R_{\nu}^+$ so that we can apply results of \cite{CBdecomp}, where the case $\theta = 0$ is considered. This is only for convenience, since the arguments of \cite{CBdecomp} can also be generalized directly to the context of the pair $(\theta,\lambda)$. Then \cite[Theorem 1.1]{CBdecomp} implies that

\begin{prop}
Let $\a \in \N R_{\lambda,\theta}^+$. Then $\alpha$ admits a unique decomposition $\a = n_1 \sigma^{(1)} + \cdots + n_k \sigma^{(k)}$ as a sum of element $\sigma^{(i)} \in \Sigma_{\lambda,\theta}$ such that any other decomposition of $\a$ as a sum of elements from $\Sigma_{\lambda,\theta}$ is a refinement of this decomposition. 
\end{prop}

As is apparent from the results stated in the introduction, indivisible roots in $\Sigma_{\lambda,\theta}$ play an important role in this paper. Occasionally it is useful to compare this with the condition of being $\Sigma$-indivisible, i.e., being indivisible in  $\Sigma_{\lambda,\theta}$:

\begin{thm}\label{thm:Sigmadivisiblevsdivisible}
	If $\alpha \in \Sigma_{\lambda,\theta}$ is imaginary, with $\alpha = m \beta$ for some indivisible root $\beta$, then one of the following hold:
	\begin{enumerate}
		\item[(a)] $\beta$ is isotropic and $m = 1$, 
		\item[(b)] $\beta$ is anisotropic and $\beta \in \Sigma_{\lambda,\theta}$; or
		\item[(c)] $\beta$ is anisotropic, $\beta \notin \Sigma_{\lambda,\theta}$ and $m > 1$ can be chosen arbitrarily.
                \end{enumerate}
The following converse to (b) holds: if $\beta \in \Sigma_{\lambda,\theta}$ is anisotropic, then $m\beta \in \Sigma_{\lambda,\theta}$ for all $m \geq 1$. 
\end{thm}

\begin{proof}
  Once again, choose once again $\nu \in \C^{Q_0}$ such that $R_{\lambda,\theta}^+ = R_{\nu}^+$ and let $\mc{F}_{\nu}$ be the ``relative fundamental domain'', as defined in \cite[\S 7]{CBmomap}. Then Theorem \ref{thm:Sigmadivisiblevsdivisible} follows from \cite[Theorem 8.1]{CBmomap} provided that $\alpha \in \mc{F}_{\lambda,\theta}$.  Namely, there it is described precisely the set $\mc{F}_{\nu} \setminus \Sigma_{\nu}$, which has a very special form, called types (I), (II), and (III). Type (I) is the isotropic case: namely the multiples by positive integers $m \geq 2$ of the imaginary root of an affine Dynkin  subquiver. They are divisible. Types (II) and (III) are indivisible, and anisotropic.

  If $\alpha$ is not in $\mc{F}_{\nu}$ then, by definition, there is a sequence of admissible reflections (whose product is $w$ say) mapping $\alpha$ to $w(\alpha) \in \mc{F}_{w(\nu)}$ (where $w(\nu)$ uses the action of dual reflections rather than reflections). Moreover, by \cite[Lemma 5.2]{CBmomap}, $w(\alpha)$ also belongs to $\Sigma_{w(\nu)}$. Thus, it suffices to note that if trichotomy of the theorem holds for $w(\alpha)$, then it also holds for the root $\alpha$.

  The final statement follows from  \cite[Proposition 1.2 (3)]{CBdecomp}. For the convenience of the reader we recall the proof, since it is closely related to the above.
 As we mentioned, the anisotropic cases (II) and (III) mentioned above are both indivisible. Thus every divisible anisotropic element of $\mc{F}_{\nu}$ is in $\Sigma_{\nu}$. So the above reductions imply the statement.
\end{proof}
\begin{cor}\label{c:ni=1} In the canonical decomposition \eqref{eq:alphadecomp},
  $n_i = 1$ if $\sigma^{(i)}$ is anisotropic.
  \end{cor}
  \begin{proof} This follows immediately from the final statement of Theorem \ref{thm:Sigmadivisiblevsdivisible}, by the definition of the canonical decomposition.
    \end{proof}
Notice that Theorem \ref{thm:Sigmadivisiblevsdivisible} says that if $\beta$ is an indivisible anisotropic root such that some multiple of $\beta$ belongs to $\Sigma_{\lambda,\theta}$, then every \textit{proper} multiple of $\beta$ belongs to $\Sigma_{\lambda,\theta}$. However, in some cases $\beta$ itself need not belong to $\Sigma_{\lambda,\theta}$. 




\subsection{Stability}\label{sec:stab}

Let $\theta \in \Z^{Q_0}$ be a stability condition. Given a representation $M$ of $\dQ$ (e.g., a module over $\Pi^{\lambda}(Q)$), let $\theta(M) := \theta \cdot \dim M$. Note that a representation $M$ of $\Pi^\lambda(Q)$ is the same as a point in the zero fiber $\mu^{-1}(\lambda)$.
%
Recall that a $\Pi^{\lambda}(Q)$-representation $M$ (hence also a
point in $\mu^{-1}(\lambda)$) such that $\theta(M) = 0$, is said to be
$\theta$-stable, respectively $\theta$-semistable, if
$\theta(M') < 0$, respectively $\theta(M') \le 0$, for all proper
nonzero subrepresentations $M'$ of $M$. A representation $M$ is said
to be $\theta$-polystable if $M = M_1 \oplus \cdots \oplus M_k$ with
$\theta(M_i) = 0$, such that each $M_i$ is $\theta$-stable. The set of
$\theta$-semistable points in $\mu^{-1}(\lambda)$ is denoted
$\mu^{-1}(\lambda)^{\theta}$. We define a partial order on $\Z^{Q_0}$
by setting $\theta' \ge \theta$ if $M$ $\theta'$-semistable implies
that $M$ is $\theta$-semistable, i.e.,
$$
\theta' \ge \theta \quad \Longleftrightarrow \quad
\mu^{-1}(\lambda)^{\theta'} \subset \mu^{-1}(\lambda)^{\theta}.
$$
The space $\Rep (\dQ, \alpha)$ has a natural Poisson structure. Since
the action of $\G(\alpha)$ on $\Rep (\dQ, \alpha)$ is Hamiltonian,
$$
\Nak{\lambda}{\alpha}{\theta} = \mu^{-1}(\lambda)^{\theta} \git
\G(\alpha) := \mathrm{Proj} \ \bigoplus_{k \ge 0}
\C\left[\mu^{-1}(\lambda)\right]^{k \theta}
$$
is a Poisson variety. 

\begin{lem}\label{lem:obvious}
  If $\theta' \ge \theta$, then there is a projective Poisson morphism
  $\Nak{\lambda}{\alpha}{\theta'} \rightarrow
  \Nak{\lambda}{\alpha}{\theta}$.
\end{lem}

\begin{proof}
  By definition, we have a $\G(\alpha)$-equivariant embedding
  $\mu^{-1}(\lambda)^{\theta'} \hookrightarrow
  \mu^{-1}(\lambda)^{\theta}$. This induces a morphism
$$
\Nak{\lambda}{\alpha}{\theta'} = \mu^{-1}(\lambda)^{\theta'} \git
\G(\alpha) \longrightarrow \mu^{-1}(\lambda)^{\theta} \git \, \G(\alpha)
=\Nak{\lambda}{\alpha}{\theta},
$$
between geometric quotients. We need to show that this morphism is
projective. This is local on
$\Nak{\lambda}{\alpha}{\theta}$. Therefore we may choose $n \gg 0$ and
a $n \theta$-semi-invariant $f$ and consider the open subsets
$U \cap \mu^{-1}(\lambda)^{\theta'}$ and
$U \cap \mu^{-1}(\lambda)^{\theta}$, where
$U = (f \neq 0) \subset \Rep(\overline{Q},\alpha)$. Then
$\left(U \cap \mu^{-1}(\lambda)^{\theta}\right) \git \, \G(\alpha) =
\Spec \C \left[U \cap \mu^{-1}(\lambda)\right]^{\G(\alpha)}$
is an open subset of $\Nak{\lambda}{\alpha}{\theta}$ and
$$
\left(U \cap \mu^{-1}(\lambda)^{\theta'}\right) \git \, \G(\alpha) = \mathrm{Proj} \ \bigoplus_{k \ge 0} \C\left[U \cap \mu^{-1}(\lambda)\right]^{k \theta'}
$$
such that
$\left(U \cap \mu^{-1}(\lambda)^{\theta'}\right) \git \, \G(\alpha)
\rightarrow \left(U \cap \mu^{-1}(\lambda)\right) \git
\G(\alpha)$ is the projective morphism
$$
\mathrm{Proj} \ \bigoplus_{k \ge 0} \C\left[U \cap \mu^{-1}(\lambda)\right]^{k \theta'} \longrightarrow \Spec \C\left[U \cap \mu^{-1}(\lambda)\right]^{\G(\alpha)}. 
$$
It is clear that this morphism is Poisson. 
\end{proof} 

It follows from the proof of Lemma \ref{lem:obvious} that if
$\theta'' \ge \theta' \ge \theta$ then the projective morphism
$\Nak{\lambda}{\alpha}{\theta''} \rightarrow
\Nak{\lambda}{\alpha}{\theta}$
factors through $\Nak{\lambda}{\alpha}{\theta'}$. 

We will frequently use the fact that for each point $x \in \Nak{\lambda}{\alpha}{\theta}$, there is a unique closed $\G(\alpha)$-orbit in the fibre over $x$ of the quotient map $\xi : \mu^{-1}(\lambda)^{\theta} \rightarrow \Nak{\lambda}{\alpha}{\theta}$. Recall that this closed orbit is denoted $T(x)$.

\section{Canonical Decompositions of the Quiver Variety}\label{sec:canonical}

In this section we recall the canonical decomposition of quiver varieties described in \cite{CBdecomp}, and show that it holds in slightly greater generality than stated there.

\subsection{A stratification}\label{sec:strata} 

Let $x \in \Nak{\lambda}{\alpha}{\theta}$ be a closed point and $y \in T(x)$. Recall the following
basic fact:

\begin{prop}\cite[Proposition 3.2 (i)]{KingStable} \label{p:polystable}
	A point of a closed $\G(\alpha)$-orbit in
	$\mu^{-1}(\lambda)^{\theta}$ is a $\theta$-polystable representation.
\end{prop}
In more detail, \cite[Proposition 3.2 (ii)]{KingStable} states that
two points of $\mu^{-1}(\lambda)^{\theta}$ determine the same point of
$\Nak{\lambda}{\alpha}{\theta}$ if and only if the corresponding
representations admit filtrations whose associated graded subquotients
are isomorphic $\theta$-polystable representations.

Therefore $y$ decomposes into a direct sum
$y_1^{e_1} \oplus \cdots \oplus y_k^{e_k}$ of $\theta$-stable
representations, with multiplicity. Let $\beta^{(i)} = \dim y_i$. The
point $x$ is said to have \textit{representation type}
$\tau = (e_1, \beta^{(1)}; \ds ; e_k, \beta^{(k)})$. Associated to
this is the stabilizer group
$G_{\tau}=\G(\a)_y$, which is independent of the choice of $y$
up to  conjugation in $\G(\a)$.
Even though $\mu^{-1}(\lambda)^{\theta}$
is not generally affine, the fact that a nonzero morphism between
$\theta$-stable representations is an isomorphism implies:

\begin{lem}\label{lem:redstab}
	The group $G_{\tau}$ is reductive. 
\end{lem}

In fact, it is isomorphic to $\prod_{i = 1}^k GL_{e_i}(\C)$. We denote
the conjugacy class of a closed subgroup $H$ of $\G(\a)$ by $(H)$.
Given a reductive subgroup $H$ of $\G(\a)$, let $\Nak{\lambda}{\alpha}{\theta}_{(H)}$ denote the set of points $x$
such that the stabilizer of any $y \in T(x)$ belongs to $(H)$. We order the conjugacy classes of reductive subgroups of $\G(\a)$ by $(H) \le (L)$ if and only if $L$ is conjugate to a subgroup of $H$.

\subsection{\'Etale local structure}\label{sec:etale}

In this section, we recall the \'etale local structure of $\Nak\lambda\a\theta$, as described in \cite[Section 4]{CBnormal}. Since it is assumed in \textit{op.~cit.} that $\theta = 0$, we provide some details to ensure the results are still applicable in this more general setting. Let $x, y, y_1,\ldots,y_k, \beta^{(1)},\ldots,\beta^{(k)}$, and $\tau$ be as in Section \ref{sec:strata}. Let $Q'$ be the quiver with $k$ vertices whose double has $2 p(\beta^{(i)})$ loops at vertex $i$ and $- (\beta^{(i)},\beta^{(j)})$ arrows from vertex $i$ to $j$. The $k$-tuple $\mathbf{e} = (e_1, \ds, e_k)$ defines a dimension vector for the quiver $Q'$.

If $X$ and $Y$ are Poisson varieties, then we say that there is a \'etale Poisson isomorphism between a neighborhood of $x \in X$ and $y \in Y$ if there exists a Poisson variety $Z$ and Poisson morphisms $Y \stackrel{\psi}{\longleftarrow} Z \stackrel{\phi}{\longrightarrow} X$ and $z \in Z$ such that $\phi(z) = x$,  $\psi(z) = y$ and both $\phi$ and $\psi$ are \'etale at $z$.

\begin{thm}\label{thm:etalelocalgeneral}
	There is an \'etale Poisson isomorphism between a neighborhood of $0$ in $\mu_{Q'}^{-1}(0) \git \, \G(\mathbf{e})$ and a neighborhood of $x \in \Nak\lambda{\alpha}{\theta}$. 
\end{thm}

The proof of Theorem \ref{thm:etalelocalgeneral} is given in section \ref{sec:thmetalelocalgeneral} below. By taking the completion $\widehat{\mathfrak{M}}_{\lambda}(\alpha,\theta)_{x}$ of $\Nak\lambda{\alpha}{\theta}$ at $x$ and the completion $\widehat{\mathfrak{M}}_{0}(\mathbf{e},0)_0$ of $\Nak{0}{\mathbf{e}}{0}$ at $0$, the formal analogue of Theorem \ref{thm:etalelocalgeneral} is:

\begin{cor}\label{cor:formnbd}
	There is an isomorphism of formal Poisson schemes $\widehat{\mathfrak{M}}_{\lambda}(\alpha,\theta)_{x}  \simeq \widehat{\mathfrak{M}}_{0}(\mathbf{e},0)_0$. 
\end{cor}

\begin{remark}\label{rem:etalep}
	An easy calculation shows that $p(\alpha) =p(\mathbf{e})$. It can also be deduced from the fact that $\dim \widehat{\mathfrak{M}}_{\lambda}(\alpha,\theta)_{x}  = \dim \widehat{\mathfrak{M}}_{0}(\mathbf{e},0)_0$. This fact will be useful later. 
\end{remark}


\subsection{The proof of Theorem \ref{thm:etalelocalgeneral}}\label{sec:thmetalelocalgeneral}

Fix $M = \Rep(\overline{Q},\alpha)$ and $G = \G(\alpha)$. Recall that
$M$ has a canonical $G$-invariant symplectic form $\omega$. Since
$y \in M^{\theta}$, there exists some $n > 0$ and
$n \theta$-semi-invariant function $\gamma$ such that
$\gamma(y) \neq 0$. We fix such a $\gamma$, and let
$M_{\gamma} \subset M^{\theta}$ be the \textit{affine} open subset of
$M$ defined by the non-vanishing of $\gamma$. Let $H := \G(\alpha)_y$
be the stabilizer of $y$ in $\G(\alpha)$ and $\mf{h}$ the Lie algebra
of $H$. Since $\mf{h}$ is reductive we can fix a $\mf{h}$-stable
complement $L$ to $\mf{h}$ in $\mf{g}$. By \cite[Lemma 4.1]{CBnormal},
the $H$-submodule $\mf{g} \cdot y \subset M$ is isotropic, and by
\cite[Corollary 2.3]{CBnormal}, there exists a coisotropic $H$-module
complement $C$ to $\mf{g} \cdot y$ in $M$. Let
$W = (\mf{g} \cdot y)^{\perp} \cap C$. The composition of
$\mu: M \rightarrow \mf{g}^*$ with the restriction map
$\mf{g}^* \rightarrow \mf{h}^*$ is denoted $\mu_H$. Notice that
$\mu_H$ is simply the moment map for the action of $H$ on $M$. The
restriction of $\mu_H$ to $W$ is denoted $\hat{\mu}$. There is a
natural identification of $W$ with $\Rep(\overline{Q}',\mathbf{e})$
such that $\hat{\mu} = \mu_{Q'}$.

\begin{lem}\label{lem:Gxred}
	The group $H$ is isomorphic to $\G(\mathbf{e})$ and $\theta |_{H}$ is the trivial character. 
\end{lem}

\begin{proof}
	The isomorphism $H \simeq \G(\mathbf{e})$ follows from the fact that $\Hom_{\Pi^{\lambda}(Q)}(M_1,M_2) = 0$ if $M_1$ and $M_2$ are non-isomorphic $\theta$-stable representations and $\End_{\Pi^{\lambda}(Q)}(M_i,M_i) = \C$. Under this identification, 
	$$
	\theta |_{\G(\mathbf{e})} = (\theta \cdot \beta^{(1)}, \ds, \theta \cdot
	\beta^{(k)}) = (0,\ds,0) \in \Z^{Q_0'}
	$$
	is the trivial stability condition. 
\end{proof}

As in \cite{CBnormal}, define $\nu : C \rightarrow L^*$ by 
$$
\nu(c) (l) = \omega(c,l \cdot y) + \omega(c,l \cdot c) + \omega(y,l \cdot c). 
$$
Theorem \ref{thm:etalelocalgeneral} follows from the following more precise result. 

\begin{thm}\label{thm:etalelocalgeneralgeneral}
	There exists a $G$-saturated affine open set $V \subset M^{\theta}$, and $H$-saturated affine open sets $Z \subset C$ and $U \subset \Rep(\overline{Q}',\mathbf{e})$ such that
	\begin{enumerate}
		\item[(a)] there are \'etale Poisson morphisms
		$$
		\phi : G \times_H Z \rightarrow V, \quad \psi :  Z \cap \nu^{-1}(0) \rightarrow U;
		$$
		\item[(b)] the morphisms $\phi$ and $\psi$ induce \'etale Poisson maps
		\begin{align*}
		(Z \cap \nu^{-1}(0) \cap \mu_H^{-1}(0)) \git H & \rightarrow (U \cap \hat{\mu}^{-1}(0))\git H, \\
		(\phi^* \mu)^{-1}(\lambda)^{\theta} \git \, G & \rightarrow (V \cap \mu^{-1}(\lambda)) \git \, G.
		\end{align*}
		\item[(c)] There is an isomorphism of \textit{Poisson} varieties, 
		$$
		\Phi : (Z \cap \nu^{-1}(0) \cap \mu_H^{-1}(0)) \git H \stackrel{\sim}{\longrightarrow} (\phi^* \mu)^{-1}(\lambda)^{\theta} \git \, G.
		$$
	\end{enumerate} 
\end{thm}

If we assume that $y \in \mu^{-1}(\lambda)$ then for $k \in \mf{h}$ and $l \in L$,  
\begin{align*}
\mu(y + c) (k+l) & = \lambda(k+l) + \nu(c)(l) + \mu_H(c)(k) + \omega(y,k \cdot c) - \omega(k \cdot y, c) \\
& = \lambda(k+l) + \nu(c)(l) + \mu_H(c)(k)
\end{align*}
because $k \cdot y = 0$. We define $\delta : C \rightarrow \C$ by $\delta(c) = \gamma(c + y)$. Then $\delta$ is $H$-invariant. We let $C_{\delta}$ denote the non-vanishing locus of $\delta$. Then 
\begin{equation}\label{eq:pullbackphiCheq}
\{ c \in C \ | \ c + y \in M_{\gamma} \cap \mu^{-1}(\lambda) \} = C_{\delta} \cap \mu_H^{-1}(0) \cap \nu^{-1}(0).
\end{equation}
Let $X = G \times_{H} C_{\delta}$. Since $M = C \oplus \mf{g} \cdot y$, the map $\phi: X \rightarrow M_{\gamma}$, $\phi(g,c) = g \cdot (c + y)$ is \'etale at $(1,0)$. We recall that a $G$-morphism $\phi: X \rightarrow Y$ is said to be \textit{excellent} if 
\begin{enumerate}
	\item[(a)] $\phi$ is \'etale.
	\item[(b)] The induced map $\phi/G : X\git \, G \rightarrow Y\git \, G$ is \'etale.
	\item[(c)] The morphism $X \rightarrow Y \times_{Y \git \, G} X\git \, G$ is an isomorphism. 
\end{enumerate}

\begin{lem}\label{lem:LunaonePoisson}
	There exists an affine, $H$-saturated open neighbourhood $Z$ of $0$ in $C_{\delta}$, such that $\phi$ restricts to an excellent Poisson morphism 
	$$
	\phi: G \times_H Z \rightarrow V := \mathrm{Im} \ \phi \subset M_{\gamma},
	$$
	inducing a \'etale Poisson morphism  
	$$
	(\phi^* \mu)^{-1}(\lambda) \git \, G \rightarrow (\mu^{-1}(\lambda) \cap V) \git \, G.
	$$
\end{lem}

\begin{proof}
	This is a direct consequence of Luna's Fundamental Lemma \cite{Luna}, together with the fact that every $G$-saturated affine open subset of $X$ is of the form $G \times_H Z$ for some $H$-saturated open subset of $C_{\delta}$. Since $\phi: G \times_H Z \rightarrow V$ is excellent, the form $\phi^* \omega$ on $X$ is symplectic, with moment map $\phi^* \mu$. In particular, \cite[Lemma 3.7]{SchwarzHamred} says that the corresponding \'etale morphism of Hamiltonian reductions $(\phi^* \mu)^{-1}(\lambda) \git \, G \rightarrow (\mu^{-1}(\lambda) \cap V) \git \, G$	is Poisson. 
\end{proof}

\begin{prop}\label{prop:morphispsiPoisson}
	There exist $H$-saturated open subsets $Z$ of $\nu^{-1}(0)$ and $U$ of $W$ such that the morphism 
	$$
	(\mu_H^{-1}(0) \cap Z) \git H \rightarrow (\hat{\mu}^{-1}(0) \cap U) \git  H
	$$
	is Poisson and \'etale. 
\end{prop}

\begin{proof}
	Let $\hat{\omega} = \omega |_W$. As in \cite[Lemma 4.3]{CBnormal}, $\hat{\omega}$ is a $H$-invariant symplectic form on $W$, with corresponding moment map $\hat{\mu}$. Write $p : C \rightarrow W$ for the projection map along $C^{\perp}$ and $\overline{p} : \nu^{-1}(0) \rightarrow W$ for the restriction of $p$ to $\nu^{-1}(0)$. We claim that $\overline{p}^* \hat{\omega} = \omega |_{\nu^{-1}(0)}$ and $\overline{p}^* \hat{\mu} = \mu_H |_{\nu^{-1}(0)}$. This follows, by definition, from $p^* \hat{\omega} = \omega |_{C}$ and $p^* \hat{\mu} = \mu_H |_{C}$. The latter two equalities can be checked by a direct computation. 
	
	By \cite[Lemma 4.5]{CBnormal}, the map $\nu$ is smooth at $0$ and $\omega |_{\nu^{-1}(0)}$ is non-degenerate at $0$ with moment map $\mu_H |_{\nu^{-1}(0)}$. Moreover, \textit{loc.~cit.} shows that the kernel of $d_0 \nu$ is $W$, thus $d_0 p : T_0 \nu^{-1}(0) \rightarrow T_0 W$ is the identity map. This implies that $\overline{p} : \nu^{-1}(0) \rightarrow W$ is \'etale at $0$. Applying Luna's Fundamental Lemma once again, we deduce that there are $H$-saturated affine open subset $Z \subset \nu^{-1}(0)$ and $U = \overline{p}(Z)$ such that $\overline{p} : Z \rightarrow U$ and $\overline{p}/H : Z \git H \rightarrow U \git H$ are \'etale. Since $\overline{p}^* \hat{\mu} = \mu_H |_{\nu^{-1}(0)}$, pulling back $\overline{p}/H$ along the closed embedding $\hat{\mu}^{-1}(0) \git H \rightarrow W \git H$ gives an \'etale morphism $(Z \cap \mu_H^{-1}(0)) \git H \rightarrow (U \cap \hat{\mu}^{-1}(0)) \git H$.   
	
	Shrinking $Z$ if necessary, we may assume that $\overline{p}^* \hat{\omega} = \omega |_{\nu^{-1}(0)}$ is non-degenerate on $Z$. Since $\overline{p}^* \hat{\mu} = \mu_H |_{\nu^{-1}(0)}$, it follows from \cite[Lemma 3.7]{SchwarzHamred} that the map $(Z \cap \mu_H^{-1}(0)) \git H \rightarrow (U \cap \hat{\mu}^{-1}(0)) \git H$ is Poisson. 
\end{proof}

	The $H$-equivariant closed embedding $j : \nu^{-1}(0) \cap C_{\delta} \hookrightarrow G \times_H C_{\delta}$ given by $j(c) = (1,c)$ induces an isomorphism 
	\begin{equation}\label{eq:lunaembeddingjiso}
	\Psi : (\mu^{-1}_H(0) \cap \nu^{-1}(0) \cap Z) \git H \iso (\phi^* \mu)^{-1}(\lambda) \git \, G.
	\end{equation}
	We will show later that this isomorphism is Poisson. Let $M_{(H)}$ be the set of points $m$ in $M^{\theta}$ such that 
\begin{enumerate}
	\item[(a)] $G \cdot m$ is closed in $M^{\theta}$; and
	\item[(b)] $G_m$ is conjugate to $H$. 
\end{enumerate}
If $V$ is a $G$-module, then $V_G$ denotes the complement to $V^G$.

\begin{lem}\label{lem:sheetclosedorbitsmooth}
	The set $M_{(H)}$ is a smooth locally closed subset of $M^{\theta}$ with 
	\begin{equation}\label{eq:tangentXsubH}
	T_m M_{(H)} = T_m M^H \oplus (\mf{g} / \mf{h})_H, \quad \forall \ m \in \left( M_{(H)} \right)^H. 
	\end{equation}	
\end{lem}

\begin{proof}
	To show that $M_{(H)}$ is locally closed, it suffices to prove that, for each $m \in M_{(H)}$, there exists some $G$-stable affine open neighbourhood $U$ of $m$ in $M^{\theta}$ such that $U \cap M_{(H)}$ is closed in $U$. By a result of Richardson, \cite[Proposition 3.3]{RichardsonPricipalOrbit}, the fact that all stabilizers are connected implies that there is a $G$-stable open set $U$ such that the stabilizer $G_u$ of each $u \in U$ is conjugate to a subgroup of $H$. In particular, we see that if $n = \dim H$ then $\dim G_u < n$ for all $u \in U \smallsetminus M_{(H)}$. Therefore, $U \cap M_{(H)} = \{ u \in U \ | \ \dim G_u \ge n \}$. This is closed by \cite[Lemma 2.2]{BorhoKraftSheets}. It will follow that $M_{(H)}$ is smooth if we can prove identity \eqref{eq:tangentXsubH}, since $M^H$ is smooth by \cite[Corollary 6.5]{PopovVinberg}. 
	
	In order to prove identity \eqref{eq:tangentXsubH}, we apply Luna's slice theorem \cite{Luna}. There exists an excellent map $\phi: G \times_H S \rightarrow U$, where $S$ is a slice to the $G$-orbit at $m$. Then, $\phi^{-1}(M_{(H)}) = G \times_H S_{(H)} = G / H \times S^H$. Thus, 
	$$
	d_{(1,m)} \phi : T_{(1,m)} \left(G / H \times S^H\right) \stackrel{\sim}{\longrightarrow} T_p M_{(H)}
	$$
	has image $T_m S^H \oplus \mf{g} \cdot m$ in $T_m M$, hence $T_m M_{(H)} = T_m S^H \oplus \mf{g} \cdot m$. Since $S^H \subset M^H \subset M_{(H)}$, we have $T_m S^H \subset T_m M^H \subset T_m M_{(H)}$ and hence 
	$$
	T_m M^H = (T_m M)^H = (T_p M_{(H)})^H = T_m S^H \oplus (\mf{g} \cdot m)^H. 
	$$
	Thus, 
	$$
	T_p M_{(H)} = T_m S^H \oplus (\mf{g} \cdot m)^H \oplus (\mf{g} \cdot m)_H = T_m M^H \oplus (\mf{g} / \mf{h})_H
	$$
	as required.        
\end{proof}

\begin{lem}
	The variety $\mu^{-1}(\lambda)^{\theta} \cap M_{(H)}$ is smooth, with 
	$$
	T_y \left(\mu^{-1}(\lambda)^{\theta} \cap M_{(H)}\right) = (M^H \cap (\mf{g} \cdot y)^{\perp}) \oplus (\mf{g} \cdot y)_H, 
	$$
	for all $y \in \mu^{-1}(\lambda)^{\theta} \cap (M_{(H)})^H$. 
\end{lem}

\begin{proof}
	Note that every point of $\mu^{-1}(\lambda)^{\theta} \cap M_{(H)}$ is conjugate by $G$ to some point in $\mu^{-1}(\lambda)^{\theta} \cap (M_{(H)})^H$. By Lemma \ref{lem:sheetclosedorbitsmooth}, we have 
	\begin{align*}
	T_y \left(\mu^{-1}(\lambda)^{\theta} \cap M_{(H)}\right) & = T_y M_{(H)} \cap \Ker d_y \mu, \\
	& = (M^H \oplus \mf{g} \cdot y ) \cap (\mf{g} \cdot y)^{\perp} \\
	& = (M^H \cap (\mf{g} \cdot y)^{\perp}) \oplus (\mf{g} \cdot y)_H
	\end{align*}
	since $\mf{g} \cdot y \subset (\mf{g} \cdot y)^{\perp}$ is isotropic. Therefore, we just need to show that the dimension of $\mu^{-1}(\lambda)^{\theta} \cap M_{(H)}$, as a reduced variety, is also equal to $\dim ((M^H \cap (\mf{g} \cdot y)^{\perp}) \oplus (\mf{g} \cdot y)_H)$. We have 
	$$
	(\phi^* \mu)^{-1}(\lambda) \cap (G \times_H C)_{(H)} = G \times_H (\nu^{-1}(0) \cap \mu_H^{-1}(0))_{(H)}.
	$$
	Set-theoretically, this equals $G / H \times \nu^{-1}(0)^H$ (which is smooth) and there is an \'etale map from this space to $G / H \times W^H$. Thus, we just need to show that 
	$$
	\dim ((M^H \cap (\mf{g} \cdot y)^{\perp}) \oplus (\mf{g} \cdot y)_H) =  \dim G / H \times \nu^{-1}(0)^H. 
	$$
	If $M = C \oplus (\mf{g} \cdot y)$, then $M^H = C^H \oplus(\mf{g} \cdot y)^H$. The fact that $W = C \cap (\mf{g} \cdot y)^{\perp}$ implies that 
	$$
	C^H \cap (\mf{g} \cdot y)^{\perp} = \left( C \cap (\mf{g} \cdot y)^{\perp} \right)^H =W^H. 
	$$
	Thus, 
	\begin{align*}
	\dim ((M^H \cap (\mf{g} \cdot y)^{\perp}) \oplus (\mf{g} \cdot y)_H) & = \dim C^H \cap (\mf{g} \cdot y)^{\perp} + \dim (\mf{g} \cdot y)^H + \dim (\mf{g} \cdot y)_H \\
	& = \dim W^H + \dim (\mf{g} \cdot y) \\
	& = \dim G / H \times \nu^{-1}(0)^H
	\end{align*}
	as required. 
\end{proof}

\begin{thm}\label{thm:symponstratum}
	There exists a unique symplectic form $\omega_H$ on $\Nak{\lambda}{\alpha}{\theta}_{(H)}$ such that 
	$$
	\pi^* \omega_H = \omega |_{\mu^{-1}(\lambda)^{\theta} \cap M_{(H)}},
	$$
	where $\pi: \mu^{-1}(\lambda)^{\theta} \cap M_{(H)} \rightarrow \Nak{\lambda}{\alpha}{\theta}_{(H)}$ is the quotient map. 
\end{thm}

\begin{proof}
	For brevity, let $Y =\mu^{-1}(\lambda)^{\theta} \cap M_{(H)} \cap V$, where $V$ is the affine open set of Lemma \ref{lem:LunaonePoisson}, and set $\mf{M} = \Nak{\lambda}{\alpha}{\theta}$. Abusing notation, we will also write $\mf{M} \cap V$ for the affine open subset $(V \cap \mu^{-1}(\lambda)) \git \, G$ of $\mf{M}$. We claim that we have a commutative diagram of linear maps
	$$
	\begin{tikzcd}
	\ar[r,"\sim"] W^H \oplus \mf{g}/\mf{h} \ar[d]  & T_y Y \ar[d,"d_y \pi"]  \\
	W^H  \ar[r,"\sim"] & T_y \mf{M}_{(H)},
	\end{tikzcd}
	$$
	where the vertical map on the left is just projection.
	
	Since $\phi$ is excellent, we have an identification 
	$$
	\phi^{-1}(Y) = G / H \times (\nu^{-1}(0)^H \cap U),
	$$
	which means that the diagram 
	$$
	\begin{tikzcd}
	G / H \times \nu^{-1}(0)^H \ar[r,"\phi"] \ar[d,"{\eta}"'] & Y \ar[d,"\pi"] \\
	\left(G / H \times \nu^{-1}(0)^H \right) \git \, G \ar[r,"\phi/G"] & \mf{M}_{(H)}
	\end{tikzcd}
	$$
	commutes, with $\phi$ and $\phi/G$ being \'etale. Under the identification $T_0 \nu^{-1}(0)^H = W^H$, the differential map $d \eta :  T_0 \nu^{-1}(0)^H \oplus \mf{g} / \mf{h} \rightarrow T_0 \nu^{-1}(0)^H$ is the projection map $W^H \oplus \mf{g}/\mf{h}  \rightarrow W^H$, as required. 
	
	We deduce that $\pi$ is a smooth morphism on $Y$. Hence $\pi^* : \Omega^2_{(\mf{M} \cap V)_{(H)}} \rightarrow \Omega^2_{Y}$ is an embedding, with image $\left( \Omega^2_{Y} \right)^G$. Thus, there is a unique (closed) $2$-form $\omega_H$ on $(\mf{M} \cap V)_{(H)}$, whose pull-back along $\pi$ equals $\omega|_{Y}$. 
	
	Finally, to prove that $\omega_H$ is symplectic it suffices to prove that the radical of $\omega|_{Y}$ at $m$ equals $\mf{g}/ \mf{h}$. Clearly the latter is contained in the former. Since $T_m Y = W^H \oplus (\mf{g} / \mf{h})$, it suffices to show that $\omega |_{W^H}$ is non-degenerate. Recall that $\hat{\omega} = \omega |_W$ is non-degenerate. Then $W^H$ is a symplectic subspace since $\hat{\omega}$ is $H$-invariant.  	
\end{proof}

Next, we show that the symplectic forms $\omega_H$ come from the Poisson structure on $\mu^{-1}(\lambda)^{\theta}/\!/G$. 

\begin{lem}\label{lem:tangeHamvectstratum}
	For each $f \in \C[V]^G$, the Hamiltonian vector field $\zeta_f$ is tangent to $M_{(H)}$. 
\end{lem}

\begin{proof}
	 By Lemma \ref{lem:sheetclosedorbitsmooth}, $M_{(H)}$ is smooth, therefore it suffices to show that $(\zeta_f)_y \in T_y M_{(H)}$ for all $y \in (M_{(H)})^H$. Recall from Lemma \ref{lem:sheetclosedorbitsmooth} that $T_y M_{(H)} = M^H \oplus (\mf{g}/ \mf{h})_H$. The canonical map $\mathrm{Der}(V) \rightarrow T_y M_{(H)}$ is $H$-equivariant. Since $\{ - , - \}$ is $G$-invariant, and $f \in \C[V]^G$, the Hamiltonian vector field $\zeta_f$ belongs to $\Der(V)^G \subset \Der(V)^H$. Hence $(\zeta_f)_y \in (T_y M)^H = M^H \subset T_y M_{(H)}$, as required. 
\end{proof}

\begin{thm}\label{thm:stratsymp}
	The space $\Nak{\lambda}{\alpha}{\theta}_{(H)}$ is a locally closed Poisson subvariety, such that the restriction $\{ - , - \} |_{\Nak{\lambda}{\alpha}{\theta}_{(H)}}$ of the Poisson bracket on $\Nak{\lambda}{\alpha}{\theta}$ equals the Poisson structure induced by $\omega_H$. In particular, it is non-degenerate.
\end{thm}

\begin{proof}
	Again, let $\mf{M} = \Nak{\lambda}{\alpha}{\theta}$. First we show that it is a Poisson subvariety. It suffices to show that each Hamiltonian vector field $\zeta_{\bar{f}}$ on $\mf{M} \cap V$ is tangent to $(\mf{M} \cap V)_{(H)}$. Let $f \in \C[V]^G$ be a lift of $\bar{f}$. Then, by Lemma \ref{lem:tangeHamvectstratum}, $\zeta_f$ is tangent to $M_{(H)} \cap \mu^{-1}(\lambda)^{\theta}$. By definition of Hamiltonian reduction, $\zeta_f$ is also tangent to $V \cap \mu^{-1}(\lambda)^{\theta}$. Therefore, it descends to the vector field $\zeta_{\bar{f}}$ on $\mf{M}$, which is tangent to $(M_{(H)} \cap V \cap \mu^{-1}(\lambda)^{\theta}) \git \, G$. But, by Theorem \ref{thm:symponstratum}, 
	$$
	(M_{(H)} \cap V \cap \mu^{-1}(\lambda)^{\theta}) \git \, G = (\mf{M} \cap V)_{(H)},
	$$
	as required. 
	
	Next, we show that the two Poisson structures agree. Once again, we let $Y = V \cap M_{(H)} \cap \mu^{-1}(\lambda)^{\theta}$, and let $\pi : Y \rightarrow (V \cap \mf{M})_{(H)}$ be the quotient map. 
	
	Choose a function $\bar{f}$ defined on $(V \cap \mf{M})_{(H)}$ and denote by the same symbol an arbitrary lift to $V \cap \mf{M}$. Since the form $\omega_H$ is non-degenerate on $(V \cap \mf{M})_{(H)}$ there exists a Hamiltonian vector field $\zeta_{\bar{f}}'$ on $(V \cap \mf{M})_{(H)}$ satisfying the defining equation $\omega_H(\zeta_{\bar{f}}',\eta) = - \eta(\bar{f})$ for all vector fields $\eta$. The non-degeneracy of $\omega_H$ implies that it suffices to prove that $\omega_H(\zeta_{\bar{f}},\eta) = \omega_H(\zeta_{\bar{f}}',\eta)$ for all $\eta$, since $\zeta_{\bar{f}} =\zeta_{\bar{f}}'$ implies that $\{ \bar{f}, g\} = \{ \bar{f},g\}'$ for all functions $g$ on $(V \cap \mf{M})_{(H)}$. Thus, we must show that $\omega_H(\zeta_{\bar{f}},\eta) = - \eta(\bar{f})$. 
	
	Since the quotient map $\pi : Y \rightarrow (V \cap \mf{M})_{(H)}$ is smooth, we can choose a lift of $\eta$. In fact, if we ask that the lift be $G$-invariant, it is unique, and so we will denote it by $\eta$ too. If $f$ is a lift of $\bar{f}$ to $\C[V]^G$, then $\zeta_f$ is tangent to $Y$, and $\zeta_f |_{Y}$ is a lift of $\zeta_{\bar{f}}$. Therefore,  
	$$
	\omega_H(\zeta_{\bar{f}},\eta) = \pi^* \omega_H( \zeta_f |_{Y}, \eta) = \omega|_Y(\zeta_f |_Y, \eta).
	$$
	Finally, if we choose an arbitrary lift $\eta'$ of $\eta$ to $V$, then 
	$$
	\omega|_Y(\zeta_f |_Y, \eta) = \omega (\zeta_f, \eta') |_Y = - \eta'(f)|_Y = - \eta(f |_Y) = - \eta(\bar{f}). \qedhere
	$$
\end{proof}  

Finally, we complete the proof of Theorem \ref{thm:etalelocalgeneralgeneral}. 

\begin{proof}[Proof of Theorem \ref{thm:etalelocalgeneralgeneral}]
	All claims, except for the final one, follow from Lemma \ref{lem:LunaonePoisson} and Proposition \ref{prop:morphispsiPoisson}. Thus, it suffices to note that the isomorphism $\Psi$ of \eqref{eq:lunaembeddingjiso} is Poisson. Choose a generic point $n$ in 
	$$
	(Z \cap \nu^{-1}(0) \cap \mu_H^{-1}(0)) \git  H \simeq (\phi^* \mu)^{-1}(\lambda)^{\theta} \git \, G.
	$$
	Then there exists some $K \subset H$ such that $n \in ((Z \cap \nu^{-1}(0) \cap \mu_H^{-1}(0)) \git H)_{(K)}$. Both Poisson structures on this open stratum are non-degenerate. Therefore, it suffices to show that the corresponding symplectic $2$-forms agree via $\Psi$. Recall that the symplectic form on $((Z \cap \nu^{-1}(0) \cap \mu_H^{-1}(0)) \git H)_{(K)}$ is the unique form such that its pull-back to $Z_{(K)} \cap \nu^{-1}(0) \cap \mu_H^{-1}(0)$ agrees with $\omega |_{Z_{(K)} \cap \nu^{-1}(0) \cap \mu_H^{-1}(0)}$. Similarly, the symplectic form on $((\phi^* \mu)^{-1}(\lambda)^{\theta} \git \, G)_{(K)}$ is the unique symplectic form whose pull-back to $D := (G \times_H Z_{(K)}) \cap (\phi^* \mu)^{-1}(\lambda)^{\theta}$ equals $(\phi^* \omega) |_{D}$. Therefore, since the map $\Psi$ is induced by the closed embedding $j$, it suffices to show that 
	$$
	j^* ((\phi^* \omega) |_{D}) = \omega |_{Z_{(K)} \cap \nu^{-1}(0) \cap \mu_H^{-1}(0)}.
	$$
	But this follows from the fact that 
	$$
	D = \phi^{-1}(V_{(K)} \cap \mu^{-1}(\lambda)^{\theta}), \quad j^{-1}(D) = Z_{(K)} \cap \nu^{-1}(0) \cap \mu_H^{-1}(0), 
	$$
	and $\phi \circ j$ is the map $c \mapsto c + m$, so that $j^* \phi^* \omega = \omega |_{\nu^{-1} (0) \cap C_{\delta}}$, since $\omega$ is invariant under translation. 
\end{proof}

The following result is an important consequence of Theorem~\ref{thm:stratsymp}. 

\begin{prop}\label{prop:strata}
	The strata $\Nak{\lambda}{\alpha}{\theta}_{\tau} := \Nak{\lambda}{\alpha}{\theta}_{(G_\tau)}$ define a finite stratification of $\Nak{\lambda}{\alpha}{\theta}$ into locally closed subsets such that 
	$$
	\Nak{\lambda}{\alpha}{\theta}_{(H)} \subset \overline{\Nak{\lambda}{\alpha}{\theta}_{(L)}} \quad \Leftrightarrow \quad (H) \le (L). 
	$$
	Moreover, the connected components of the strata are precisely the symplectic leaves of $\Nak{\lambda}{\alpha}{\theta}$, with respect to its natural Poisson bracket. 
\end{prop}

\begin{proof}
	It is well-known that the stratification of $\Rep(Q,\alpha)^{\theta}\git \, \G(\alpha)$ by stabilizer type is finite, with smooth locally closed strata. Therefore the stratification $\{ \Nak{\lambda}{\alpha}{\theta}_{\tau} \}$ of $\Nak{\lambda}{\alpha}{\theta}$ is finite with locally closed strata. Thus it suffices to show that (a) each stratum is smooth, and (b) the Poisson structure is non-degenerate on each stratum. In fact, (b) implies (a), and both statements are implied by Theorem~\ref{thm:stratsymp}. 
\end{proof}

We will show in Corollary~\ref{cor:strataNakaconnected} that each stratum $\Nak{\lambda}{\alpha}{\theta}_{\tau}$ is connected.

\subsection{Hyperk\"ahler twisting}\label{sec:hypertwist}

Let $\alpha = m_1 \nu^{(1)} + \cdots + m_t \nu^{(t)}$ be the canonical decomposition of $\alpha$ with respect to $\Sigma_{\lambda}$. It is shown in \cite{CBdecomp} that

\begin{thm}\label{thm:CBdecomp}\cite{CBdecomp}
There is an isomorphism of  varieties $\prod_i   S^{m_i} \left( \Nak{\lambda}{\nu^{(i)}}{0} \right) \simeq \Nak\lambda\a 0$.
\end{thm}

Moreover, if $\nu^{(i)}$ is real then $S^{m_i} \left( \Nak{\nu^{(i)}}{\lambda}{0} \right) = \{ \mathrm{pt} \}$ and if $\nu^{(i)}$ is anisotropic then $m_i = 1$. We now adapt Crawley-Boevey's result to the case where $\theta \neq 0$:


\begin{thm}\label{thm:decompCB2}
Let $\a = n_1 \sigma^{(1)} + \cdots + n_k \sigma^{(k)}$ be the canonical decomposition of $\alpha$ with respect to $\Sigma_{\lambda,\theta}$. Then, there is an isomorphism of \textbf{Poisson} varieties
$$
\phi : \ \prod_i   S^{n_i} \left( \Nak{\lambda}{\sigma^{(i)}}{\theta} \right) \stackrel{\sim}{\longrightarrow} \Nak\lambda\a\theta. 
$$
\end{thm}

The proof of Theorem \ref{thm:decompCB2} is given at the end of
section \ref{sec:decompCB2-pf}. In order to deduce Theorem
\ref{thm:decompCB2} from \cite[Theorem 1.1]{CBdecomp}, we use
hyperk\"ahler twists. By our main assumption \eqref{eq:assume},
$\lambda \in \R^{\Qo}$.

\begin{prop}\label{prop:diffeo}
Let $\nu = -\lambda - \mathbf{i} \theta$ and consider
$\Nak\lambda\a\theta$, $\Nak{\nu}{\a}{0}$ as complex analytic
spaces. Hyperk\"ahler twisting defines a homeomorphism of stratified
spaces
$$
\Psi : \Nak\lambda\a\theta \stackrel{\sim}{\longrightarrow} \Nak{\nu}{\a}{0},
$$ i.e. $\Psi$ restricts to a homeomorphism $
\Nak\lambda\a\theta_{(H)} \stackrel{\sim}{\longrightarrow}
\Nak{\nu}{\a}{0}_{(H)}$ for all classes $(H)$. In particular, the
homeomorphism maps stable representations to stable ($=$ simple)
representations.
\end{prop}

\begin{proof}
We follow the setup described in the proof of \cite[Lemma 3]{CBKleinian}. We have moment maps
$$ \mu_{\C}(x) = \sum_{a \in Q_1} [x_a,x_{a^*}], \quad \mu_{\R}(x) =
\frac{\sqrt{-1}}{2} \sum_{a \in Q_1} [x_a,x_{a}^{\dagger}] +
     [x_{a^*},x_{a^*}^{\dagger}].
$$ As shown in \cite[Corollary 6.2]{KingStable}, the Kempf-Ness
Theorem says that the embedding $\mu_{\C}^{-1}( \lambda) \cap
\mu^{-1}_{\R}( \mathbf{i} \theta) \hookrightarrow \mu_{\C}^{-1}(
\lambda)$ induces a bijection
\begin{equation}\label{eq:quo1}
\mu_{\C}^{-1}( \lambda) \cap \mu^{-1}_{\R}( \mathbf{i} \theta) / U(\alpha) \stackrel{\sim}{\longrightarrow} \Nak\lambda\a\theta. 
\end{equation}
Since the embedding is clearly continuous and the topology on the
quotients $\mu_{\C}^{-1}( \lambda) \cap \mu^{-1}_{\R}( i \theta) /
U(\alpha)$ and $\Nak\lambda\a\theta$ is the quotient topology (for the
latter space, see \cite[Corollary 1.6 and Remark
  1.7]{NeemanQuotient}), the bijection \eqref{eq:quo1} is continuous.

Define a stratification $\mu_{\C}^{-1}( \lambda) \cap \mu^{-1}_{\R}(
\mathbf{i} \theta) / U(\alpha)$ analogous to the stratification of
$\Nak\lambda\a\theta$ described in section \ref{sec:strata}. Let
$y \in \Nak\lambda\a\theta$, and $x = x_1^{e_1} \oplus \cdots \oplus x_k^{e_k} \in T(y)$ a $\theta$-polystable lift
 in $\mu_{\C}^{-1}(\lambda) \cap \mu^{-1}_{\R}( i \theta)$ (which exists by Proposition \ref{p:polystable}). Then Lemma \ref{lem:Gxred}
says that $G_x = \G(\mathbf{e})$ and \cite[Proposition 6.5]{KingStable}
implies that $U(\alpha)_{x} = U(\mathbf{e})$. Hence $\G(\alpha)_x =
U(\alpha)_x^{\C}$. Therefore the homeomorphism \eqref{eq:quo1} restricts to a bijection 
$$ 
(\mu_{\C}^{-1}( \lambda) \cap \mu^{-1}_{\R}( \mathbf{i} \theta) /
U(\alpha))_{(K)} \rightarrow \Nak\lambda\a\theta_{(K^{\C})}
$$
for each $(K)$. 

Let the quaternions $\mathbb{H} = \R \oplus \R \mathbf{i} \oplus \R
\mathbf{j} \oplus \R \mathbf{k}$ act on $\Rep(\overline{Q},\alpha)$ by
extending the usual complex structure so that $\mathbf{j} \cdot
(x_{a},x_{a^*}) = (- x_{a^*}^{\dagger}, x_a^{\dagger})$ and $\mathbf{k} \cdot
(x_{a},x_{a^*}) = ( -\mathbf{i} x_{a^*}^{\dagger}, \mathbf{i} x_a^{\dagger})$. Here the dagger denotes the Hermitian adjoint. In general,
$$
(z_1 + z_2 \mathbf{j}) \cdot (x_a,x_{a^*}) = (z_1 x_a - z_2 x_{a^*}^{\dagger}, z_1 x_{a^*} + z_2 x_a^{\dagger}).
$$
This action commutes with the action of $U(\alpha)$ and satisfies  
\begin{align}
\mu_{\R}(z \cdot x) & = ( ||z_1||^2 - || z_2||^2) \mu_{\R}(x) -
\mathbf{i} z_1 \overline{z}_2 \mu_{\C}(x) - \mathbf{i} z_2
\overline{z}_1 \mu_{\C}(x)^{\dagger}, \label{eq:zcomp} \\ \mu_{\C}(z
\cdot x) & = z_1^2 \mu_{\C}(x) - z_2^2 \mu_{\C}(x)^{\dagger} - 2
\mathbf{i} z_1 z_2 \mu_{\R}(x), \quad \forall \ z \in
\mathbb{H}. \label{eq:zreal}
\end{align}
Let $h = (\mathbf{i} - \mathbf{j}) / \sqrt{2}$. Then multiplication by $h$ defines a homeomorphism  
$$ \mu_{\C}^{-1}( \lambda) \cap \mu^{-1}_{\R}( \mathbf{i} \theta)
\stackrel{\sim}{\longrightarrow} \mu_{\C}^{-1}(- \lambda - \mathbf{i}
\theta) \cap \mu^{-1}_{\R}( 0)
$$
Since multiplication by $h$ commutes with the action of $U(\alpha)$, this homeomorphism descends to a homeomorphism 
$$ \left(\mu_{\C}^{-1}( \lambda) \cap \mu^{-1}_{\R}( \mathbf{i}
\theta) \right)/U(\alpha) \stackrel{\sim}{\longrightarrow} \left(
\mu_{\C}^{-1}(- \lambda - \mathbf{i} \theta) \cap \mu^{-1}_{\R}( 0)
\right) / U(\alpha)
$$
which preserves the stratification by stabilizer type.  

Thus, the map $\Psi$ is the composition of three homeomorphisms, each
of which preserves the stratification.
\end{proof}

\begin{remark}
Our general assumption that $\lambda \in \R^{Q_0}$ if $\theta \neq 0$
is required in the proof of Proposition \ref{prop:diffeo} to ensure
that multiplication by $h$ lands in $\mu_{\R}^{-1}(0)$. Equation
\eqref{eq:zcomp} implies that it would suffice to assume more
generally that there exists $z \in \C$ such that $|z| = 1$ and $z
\lambda \in \R^{Q_0}$. It is natural to expect that Theorem
\ref{thm:decompCB2} holds without the assumption $\lambda \in
\R^{Q_0}$.
\end{remark}

\begin{remark}
Using the notion of smooth structures on stratified symplectic spaces,
as defined in \cite{SjamaarL}, one can presumably strengthen
Proposition \ref{prop:diffeo} to the statement that there is a
diffeomorphism of stratified symplectic spaces $\Nak\lambda\a\theta
\stackrel{\sim}{\longrightarrow} \Nak{\nu}{\a}{0}$.
\end{remark}

\begin{prop}\label{prop:Mirrnormal}
The variety $\Nak\lambda\a\theta$ is irreducible and normal.
\end{prop}

\begin{proof}
We begin by showing that the variety $\Nak\lambda\a\theta$ is
connected. Proposition \ref{prop:diffeo} implies that
$\Nak\lambda\a\theta$ is connected if and only if $\Nak{\nu}{\a}{0}$
is connected. The latter is known to be connected (and nonempty) by \cite[Corollary
  1.4]{CBdecomp}.

Next, we show that $\Nak\lambda\a\theta$ is irreducible. Since
$\Nak\lambda\a\theta$ is connected, it suffices to show that, for each
$\C$-point $x\in \Nak\lambda\a\theta$, the local ring
$\mc{O}_{\Nak\lambda\a\theta,x}$ is a domain. This ring embeds into
the complete local ring
of $x$ in $\Nak\lambda\a\theta$. By Corollary
\ref{cor:formnbd}, the complete local ring
of $x$ in
$\Nak\lambda\a\theta$ is isomorphic to the
complete local ring
of $0$
in $\Nak{0}{e}{0}$. By \cite[Corollary 1.4]{CBdecomp}, this is a
domain. Finally, normality is an \'etale local property, \cite[Remark
  2.24 and Proposition 3.17]{Milne}. Therefore, as in the previous
paragraph this follows from Theorem \ref{thm:etalelocalgeneral} and
\cite[Theorem 1.1]{CBnormal}.
\end{proof}
We can now prove Theorem~\ref{t:stab-exists}.

\begin{proof}[Proof of Theorem~\ref{t:stab-exists}]  
	Let $\nu = -\lambda - \mathbf{i} \theta$. By the last statement of Proposition \ref{prop:diffeo}, there exists a $\theta$-stable representation of $\Pi^{\lambda}(Q)$ of dimension $\a$ if and only if there exists a simple representation of $\Pi^{\nu}(Q)$ of dimension $\a$. By \cite[Theorem~1.3]{CBmomap}, the latter happens if and only if $\a \in \Sigma_{\nu}$. Since $\lambda \in \R^{Q_0}$, the sets $\Sigma_{\nu}$ and $\Sigma_{\lambda,\theta}$ are equal. Thus, we deduce that there exists a $\theta$-stable representation of $\Pi^{\lambda}(Q)$ of dimension $\a$ if and only if $\a \in \Sigma_{\lambda,\theta}$. 
\end{proof}
  
As a consequence of Theorem~\ref{t:stab-exists} and Proposition \ref{prop:Mirrnormal}, we obtain:
  \begin{cor}\label{c:dense}
For $\alpha \in \Sigma_{\lambda,\theta}$, the locus of $\theta$-stable representations is dense in $\Nak\lambda\a\theta$.
  \end{cor}
\subsection{The proof of Theorem \ref{thm:decompCB2}}\label{sec:decompCB2-pf} Recall that $\alpha= n_1 \sigma^{(1)} + \cdots + n_k \sigma^{(k)}$ is the canonical decomposition of $\alpha$ in $R^+_{\lambda,\theta}$. The map $\phi$ is defined as follows. Let $H(\a)$ be the product $\G(\sigma^{(1)})^{n_1} \times \cdots \times \G(\sigma^{(k)})^{n_k}$, thought of as a subgroup of $\G(\a)$. There is a natural $H(\a)$-equivariant inclusion $\prod_i T^* \Rep(Q,\sigma^{(i)})^{n_i} \hookrightarrow T^* \Rep(Q,\a)$. This is an inclusion of symplectic vector spaces. Since the moment map for the action of $H(\a)$ on $T^* \Rep(Q,\a)$ is the composition of the moment map for $\G(\a)$ followed by projection from the Lie algebra of $\G(\a)$ to the Lie algebra of $H(\a)$, the above inclusion restricts to an inclusion $\prod_i (\mu_{\sigma^{(i)}}^{-1}(\lambda)^{\theta})^{n_i} \hookrightarrow \mu^{-1}_{\a}(\lambda)^{\theta}$, inducing a map of GIT quotients 
$$ \prod_i \Nak{\lambda}{\sigma^{(i)}}{\theta}^{n_i} \rightarrow
\Nak{\lambda}{\a}{\theta}.
$$
This map, which sends a tuple of representations $(M_{i,j})$ to the direct sum $\bigoplus_{i,j} M_{i,j}$ clearly factors through $\prod_i   S^{n_i} \left( \Nak{\lambda}{\sigma^{(i)}}{\theta} \right)$. It is this map that we call $\phi$. 

 Passing to the analytic topology, Proposition \ref{prop:diffeo}
 implies that we get a commutative diagram
\begin{equation}\label{eq:decomp1}
\xymatrix{ \prod_i S^{n_i} \left( \Nak{\lambda}{\sigma^{(i)}}{\theta}
  \right) \ar[rr] \ar[d] & & \Nak\lambda\a\theta \ar[d] \\ \prod_i
  S^{n_i} \left( \Nak{\lambda}{-\sigma^{(i)} - \mathbf{i} \theta}{0}
  \right) \ar[rr] & & \Nak{-\lambda-\mathbf{i} \theta}{\a}{0} .  }
\end{equation}
where both vertical arrows are homeomorphisms and the bottom horizontal arrow is an isomorphism by Theorem \ref{thm:CBdecomp}. Therefore, we conclude that $\phi$ is bijective. Since we are working over the complex numbers, and we have shown in Proposition \ref{prop:Mirrnormal} that $\Nak\lambda\alpha\theta$ is normal, we conclude by Zariski's main theorem that $\phi$ is an isomorphism. 

As a consequence, we can compute the dimension of
$\Nak\lambda\a\theta$, which in the case $\theta\neq0$ is
\cite[Corollary 1.4]{CBmomap}. We begin with the following basic lemma:
\begin{lem}\label{l:dimsigma}
If $\alpha \in \Sigma_{\lambda,\theta}$, then $\dim \Nak\lambda\a\theta = 2p(\alpha)$. Moreover $\dim \mu^{-1}(\lambda)^\theta \geq \alpha \cdot \alpha + 2p(\alpha)-1$.
\end{lem}
\begin{proof}
 Let $U$ be the subset of
  $\Nak\lambda\a\theta$ consisting of $\theta$-stable
  representations. Since $\alpha$ is assumed to be in
  $\Sigma_{\lambda,\theta}$, Corollary \ref{c:dense}
  implies
  that $U$ is a dense open subset of $\Nak{\lambda}{\a}{\theta}$. Let
  $V$ be the open subset of $\Rep(\overline{Q},\alpha)$ consisting of
  $\theta$-stable representations. Then $U$ is the image of
  $\mu^{-1}(\lambda) \cap V$ under the quotient map and hence $V$ is
  non-empty. The group $\PG(\alpha)$ acts freely on $V$
  and $\mu$ is smooth when restricted to $V$. Thus,
$$
\dim U = \dim \Rep(\overline{Q},\alpha) - 2 (\dim \G(\alpha) - 1) = 2 p(\alpha),
$$
as required. For the second statement, observe that $\dim (V \cap \mu^{-1}(\lambda)) = \dim U + \dim \PG(\a)$ since $\PG(\alpha)$ acts freely on $V$.
\end{proof}
Then we immediately conclude
\begin{cor}\label{cor:dimsympsing} For $\alpha\in R^+_{\lambda,\theta}$ with
canonical decomposition $\alpha = n_1 \sigma^{(1)} + \cdots + n_k \sigma^{(k)}$,
the variety $\Nak\lambda\a\theta$ has dimension $2 \sum_{i = 1}^k n_i
p(\sigma^{(i)})$.
\end{cor}
Finally, we need to check that the morphism $\phi$ is Poisson. Since
both varieties are normal by Proposition \ref{prop:Mirrnormal}, it
suffices to show that $\phi$ induces an isomorphism of smooth symplectic
varieties between the open leaf of $\Nak\lambda\a\theta$ and the open
leaf of $\prod_i S^{n_i} \left( \Nak{\lambda}{\sigma^{(i)}}{\theta}
\right)$. By Proposition \ref{prop:strata}, the symplectic leaves of
$\Nak\lambda\a\theta$ are connected components of the strata given by stabilizer
type. The explicit description of $\phi$ given at the start
of this section shows that $\phi$ restricts to an isomorphism between
strata. In particular, $\phi$ restricts to an isomorphism between the
open leaves.

The symplectic structure on the open leaf of $\Nak\lambda\a\theta$
comes from the symplectic structure on $T^* \Rep(Q,\a)$. More
specifically, the non-degenerate closed form on the latter space
restricts to a degenerate $\G(\a)$-invariant two-form on
$\mu^{-1}(\lambda)^{\theta}$. Hence it descends to a closed two-form
on $\Nak\lambda\a\theta$. The restriction of this two-form to the open
leaf is non-degenerate. The two-form on the open leaf of $\prod_i
S^{n_i} \left( \Nak{\lambda}{\sigma^{(i)}}{\theta} \right)$ is defined
similarly. Now the point is that under the embedding $\prod_i
(\mu_{\sigma^{(i)}}^{-1}(\lambda)^{\theta})^{n_i} \hookrightarrow
\mu^{-1}_{\a}(\lambda)^{\theta}$, the $H(\a)$-invariant closed
two-form on $\prod_i
(\mu_{\sigma^{(i)}}^{-1}(\lambda)^{\theta})^{n_i}$ is simply the
pull-back of the $\G(\a)$-invariant closed two-form on $
\mu^{-1}_{\a}(\lambda)^{\theta}$. This implies that the two-form on
the open leaf of $\prod_i S^{n_i} \left(
\Nak{\lambda}{\sigma^{(i)}}{\theta} \right)$ is the pull-back, under
$\phi$, of the symplectic two-form on the open leaf of
$\Nak\lambda\a\theta$.

Using Proposition \ref{prop:Mirrnormal}, we can now show that each stratum $\Nak{\lambda}{\a}{\theta}_{\tau}$ is connected. 

\begin{cor}\label{cor:strataNakaconnected}
	Each stratum $\Nak{\lambda}{\a}{\theta}_{\tau}$ of $\Nak{\lambda}{\a}{\theta}$ is irreducible (and thus connected) and nonempty. In particular, the strata $\Nak{\lambda}{\a}{\theta}_{\tau}$ are precisely the symplectic leaves of $\Nak{\lambda}{\a}{\theta}$.  
\end{cor}

\begin{proof}
	Writing $\tau = (e_1, \beta^{(1)}; \ds ; e_k, \beta^{(k)})$, we can repeat the construction of $\phi$ given above (even though $\tau$ is not the canonical decomposition of $\alpha$) to get a morphism 
	$$
	\phi : \prod_i S^{e_i} \Nak{\lambda}{\beta^{(i)}}{\theta} \rightarrow \Nak{\lambda}{\a}{\theta}.
	$$
	The stratum $\Nak{\lambda}{\a}{\theta}_{\tau}$ is contained in the image of $\phi$ and $\phi^{-1}(\Nak{\lambda}{\a}{\theta}_{\tau})$ is dense in the domain of $\phi$ by Corollary \ref{c:dense}. Since the domain is irreducible and nonempty by Proposition \ref{prop:Mirrnormal}, we deduce that $\Nak{\lambda}{\a}{\theta}_{\tau}$ is irreducible.   	
\end{proof}

\begin{remark}\label{rem:Martinoconnectedleaves}
	The statement of Corollary~\ref{cor:strataNakaconnected} (at least in the case $\theta = 0$) appears as Theorem 1.3 of \cite{MarsdenWeinsteinStratification}. However, it is falsely claimed in Proposition 4.5 of that paper that the irreducibility of the stratum $\Nak{\lambda}{\a}{0}_{\tau}$ follows from a result of G. Schwarz.  
\end{remark}

\subsection{Flatness of the moment map}
We need an additional result which follows from \cite{CBnormal}. Let
$\xi: \mu^{-1}(\lambda)^\theta \to \Nak{\lambda}{\a}{\theta}$ be the quotient map.
\begin{thm}\cite[Corollary 6.4]{CBnormal} \label{thm:dimmmf} For $\tau = (e_1, \beta^{(1)}; \ds ; e_k, \beta^{(k)})$ a representation type,
\begin{equation}
\dim \xi^{-1}(\Nak{\lambda}{\a}{\theta}_\tau) \leq \alpha \cdot \alpha - 1 +p(\alpha) + \sum_{i = 1}^k p(\beta^{(i)}).
\end{equation}
\end{thm}
\begin{proof} The proof follows verbatim as in \cite[Corollary 6.4]{CBnormal}, substituting $\theta$-stable representations for simple representations. Alternatively, \cite[Corollary 6.4]{CBnormal} as written together with Theorem \ref{thm:etalelocalgeneral} yields the result.
\end{proof}
\begin{prop}\label{p:flat-ci}
  The restricted moment map $\overline \mu: \Rep(\overline{Q},\alpha)^\theta \to \pg(\alpha)$
  is flat over an open subset $U$ containing
  $B_{\alpha,\theta} := \{\lambda \in \R^{Q_0} \mid \alpha \in
  \Sigma_{\lambda,\theta}\} \subseteq \R^{Q_0}$,
  with all fibers of dimension $\alpha \cdot \alpha + 2p(\alpha)-1$.
  In particular, if $\alpha \in \Sigma_{\lambda,\theta}$, then the
  variety $\mu^{-1}(\lambda)^\theta$ is a complete intersection in the
  open subset $\overline \mu^{-1}(U)^\theta \subseteq \Rep(\overline{Q},\alpha)^\theta$. 
\end{prop}
\begin{proof} 
  By Lemma \ref{l:dimsigma} and Theorem \ref{thm:dimmmf}, all of the fibers $\mu^{-1}(\lambda)$ for $\lambda \in B_{\alpha,\theta}$ have the same dimension, $\alpha \cdot \alpha + 2p(\alpha)-1$. Since this equals  the difference of dimensions
  $\dim \Rep(\overline{Q},\alpha) - \dim \pg(\alpha)$, it follows that the subset of the base where the fiber has this dimension is open, call it $U$.  
  Then, since $B_{\alpha,\theta}$ is smooth, and $\mu^{-1}(U)$ is open (hence smooth and therefore Cohen-Macaulay), it follows that the moment map is flat as stated, and therefore that every fiber is a complete intersection.  
\end{proof}



\section{Smooth vs. stable points}\label{sec:stabsmoothproof}

As usual, choose a deformation parameter $\lambda \in \R^{\Qo}$,
a stability parameter $\theta \in \Z^{\Qo}$, and a dimension vector $\alpha
\in \N R^+_{\lambda,\theta}$. The main goal of this section is to prove
Theorem \ref{thm:stablesmooth}, which says that $x \in
\Nak\lambda\a\theta$ is canonically $\theta$-polystable if and only if it
is in the smooth locus of $\Nak\lambda\a\theta$.

\subsection{Isotropic roots}

In this section, we briefly consider quiver varieties associated to isotropic roots. The subgroup of $\mathrm{GL}(\Z^{Q_0})$ generated by the reflection at loop free vertices is denoted $W(Q)$. 

\begin{lem}\label{lem:Kleiniancasequiver}
	Let $\alpha \in \Sigma_{\lambda,\theta}$ be an isotropic root. Then there exists $w \in W(Q)$ such that $\delta = w \alpha$ is in the fundamental domain, $Q' = \Supp \ \delta$ is an affine Dynkin quiver, $\delta |_{Q'}$ is the minimal imaginary root and $\mf{M}_{\lambda}(\alpha,\theta) \simeq \mf{M}_{w \lambda}(\delta, w \theta)$. 
\end{lem}

\begin{proof}
	As the name implies, the fundamental domain $\mc{F}(Q)$ is a fundamental domain for the action of the reflection group $W(Q)$ of $Q$ on the set of imaginary roots. Therefore there exists $w$ such that $w \alpha \in \mc{F}(Q)$. The fact that $Q'$ is affine Dynkin and $\delta |_{Q'}$ is the minimal imaginary root follows from \cite[Lemma 1.9 (d)]{KacThm}. 
	
	Thus, we show that $\delta \in \Sigma_{w \lambda,w \theta}$ and $\mf{M}_{\lambda}(\alpha,\theta) \simeq \mf{M}_{w \lambda}(\delta, w \theta)$. The Lusztig-Maffei-Nakajima reflection isomorphisms of quiver varieties (see in particular \cite[Theorem 26]{Maffei}) shows that if either $\lambda_i$ or $\theta_i$ is non-zero (equivalently, as explained in example \ref{ex:realrootdecomp}, if $e_i \notin \Sigma_{\lambda,\theta}$) then $\mf{M}_{\lambda}(\alpha,\theta) \simeq \mf{M}_{s_i \lambda}(s_i \alpha, s_i \theta)$. It is easily checked that if $e_i \in \Sigma_{\lambda,\theta}$ then $(\alpha,e_i) \le 0$ (otherwise $\alpha = (\alpha - e_i) + e_i$ with $p(\alpha) = p(\alpha - e_i)$). Moreover, the fact that $s_i$ permutes the set $R^+ \smallsetminus \{ e_i \}$ implies that $s_i \alpha \in \Sigma_{s_i \lambda, s_i \theta}$ if $e_i \notin \Sigma_{\lambda,\theta}$. Hence, we need to show that $w = s_{i_r} \cdots s_{i_1}$ can be chosen so that $(s_{i_l} \cdots s_{i_1} \lambda)_{i_{l+1}} \neq 0$ or $(s_{i_l} \cdots s_{i_1} \theta)_{i_{l+1}} \neq 0$ for all $l = 1, \ds, r-1$. Recall that every positive root $\beta = \sum_{i \in Q_0} k_i e_i$ has height $\mathrm{ht}(\beta) := \sum_{i \in Q_0} k_i \ge 1$. As in the proof of \cite[Proposition 16.10]{CarterBook}, the key thing to note is that $\delta$ is specified by the fact that it is the unique element of minimal height in the orbit $W(Q) \cdot \alpha$. Thus, if $\alpha \notin \mc{F}(Q)$, then there exists $i \in Q_0$ such that $(\alpha,e_i) > 0$. This implies that $e_i \notin \Sigma_{\lambda,\theta}$ and $\mathrm{ht}(s_i \alpha) < \mathrm{ht}(\alpha)$. Since every element in the orbit $W(Q) \cdot \alpha$ is a positive root (and hence has positive height) this cannot continue forever, and the result follows. 
\end{proof} 

In particular, we note that Lemma \ref{lem:Kleiniancasequiver} implies that if $\alpha \in \Sigma_{\lambda,\theta}$ is an isotropic root, then $\mf{M}_{\lambda}(\alpha,\theta)$ is the partial resolution of a partial deformation of a Kleinian singularity. Moreover, the type of the Kleinian singularity is specified by the support of $w \alpha \in \mc{F}$. 

\subsection{The proof of Theorem \ref{thm:stablesmooth}}\label{sec:stablesmooth-pf}

The proof of Theorem \ref{thm:stablesmooth} follows closely the
arguments given in \cite[Theorem 3.2]{LeBCoadjoint}. We provide the
necessary details that show that the arguments of \textit{loc.~cit.}
are valid in our setting. First, notice that, under the isomorphism of
Theorem \ref{thm:decompCB2}, the open subset of
canonically $\theta$-polystable points in $\Nak\lambda\a\theta$ is the
product of the canonically $\theta$-polystable points in the spaces $S^{n_i}
\Nak\lambda{\sigma^{(i)}}\theta$. Therefore it suffices to show that
the set of canonically $\theta$-polystable points in $S^{n_i}
\Nak\lambda{\sigma^{(i)}}\theta$ is precisely the smooth locus. If
$\sigma^{(i)}$ is real then $S^{n_i} \Nak\lambda{\sigma^{(i)}}\theta$
is a point. If $\sigma^{(i)}$ is an isotropic root then by Lemma \ref{lem:Kleiniancasequiver}, $\Nak\lambda{\sigma^{(i)}}\theta$ is a partial resolution of a du Val singularity. In particular, it is a $2$-dimensional
(quasi-projective) variety. This implies that the smooth locus of
$S^{n_i} \Nak\lambda{\sigma^{(i)}}\theta$ equals
$$ S^{n_i, \circ}
\ \mathfrak{M}_{\lambda}(\sigma^{(i)},\theta)_{\mathrm{sm}} := \left\{
  (p_j) \ \left| \ p_j \in
\mathfrak{M}_{\lambda}(\sigma^{(i)},\theta)_{\mathrm{sm}}, \ p_j \neq
p_k \textrm{ for $j \neq k$} \right. \right\}.
$$ On the other hand, the set of canonically $\theta$-polystable points in
$S^{n_i} \Nak\lambda{\sigma^{(i)}}\theta$ equals $S^{n_i, \circ} U$,
where $U \subset \Nak\lambda{\sigma^{(i)}}\theta$ is the set of
canonically $\theta$-polystable points. Therefore, in this case it
suffices to show that
$\mathfrak{M}_{\lambda}(\sigma^{(i)},\theta)_{\mathrm{sm}}$ equals
$U$. Finally, in the case where $\sigma^{(i)}$ is an anisotropic root, we have $n_i = 1$.

Thus, we are reduced to considering the situation where $\alpha \in
\Sigma_{\lambda,\theta}$ is an imaginary root. In this case, a point
$x$ is canonically $\theta$-polystable if and only if it is
$\theta$-stable. As in the proof of Corollary \ref{cor:dimsympsing},
it is clear from the definition of $\Nak\lambda\a\theta$ that the set
of $\theta$-stable points is contained in the smooth locus. Therefore
it suffices to show that if $x$ is not $\theta$-stable then it is a
singular point. As in section \ref{sec:etale}, let $x$ be the image of
a $\theta$-polystable representation
$y = y_1^{e_1} \oplus \cdots \oplus y_{\ell}^{e_{\ell}}$ (with the $y_i$ $\theta$-stable).
Let $\beta^{(i)} = \dim y_i$. Let $Q'$ be the quiver with $\ell$ vertices whose double
has $2 p(\beta^{(i)})$ loops at vertex $i$ and $-
(\beta^{(i)},\beta^{(j)})$ arrows between vertex $i$ and $j$. The
$\ell$-tuple $\mathbf{e} = (e_1, \ds, e_{\ell})$ defines a dimension
vector for the quiver $Q'$. By Theorem \ref{thm:etalelocalgeneral}, it
suffices to show that $0$ is contained in the singular locus of
$\mathfrak{M}_{0}(\mathbf{e},0)$.

In order to proceed, we require \cite[Proposition 1.1]{LeBSimple},
stated in our generality. The proof is identical to the proof given in
\textit{loc.~cit.}, this time using Theorem
\ref{thm:etalelocalgeneral}.

\begin{prop}\label{prop:simplevector}
Assume that $\alpha \in \Sigma_{\lambda,\theta}$ and let $x$ be a
geometric point of $\Nak{\lambda}{\alpha}{\theta}$, of representation
type $\tau = (e_1, \beta_1; \ds; e_k, \beta_k)$. Then $\mathbf{e}$ is
the dimension vector of a simple $\Pi^0(Q')$-module, i.e., $\mathbf{e}
\in \Sigma_0(Q')$.
\end{prop}

Returning to the proof of Theorem \ref{thm:stablesmooth}, with
Proposition \ref{prop:simplevector} in hand, the argument given in the
proof of \cite[Theorem 3.2]{LeBCoadjoint} goes through
\textit{verbatim}. This completes the proof of Theorem
\ref{thm:stablesmooth}.

\subsection{The proof of Corollary \ref{cor:smoothminimal}}\label{sec:smoothminimal-pf}

By Theorem \ref{thm:stablesmooth}, $\Nak\lambda\a\theta$ is smooth if
and only if every point is canonically $\theta$-polystable. As in the
reduction argument given at the start of the proof of Theorem
\ref{thm:stablesmooth}, this means that $n_i$ must be $1$ when
$\sigma^{(i)}$ is an isotropic root. Moreover, it is clear
that $\Nak\lambda{\sigma^{(i)}}\theta$ consists only of
$\theta$-stable points if and only if $\sigma^{(i)}$ is minimal.

\section{The $(2,2)$ case}\label{sec:Xnd}

In this section we will prove Theorem \ref{thm:blowup22Q}. First we restrict to the one vertex case.

\subsection{The variety $\Char(2,2)$}\label{sec:char22}
Recall that $\Char(g,n)$ denotes the quiver variety  
$$ \left\{ (X_1,Y_1, \ds, X_g,Y_g) \in \End_{\C}(\C^n) \ \left|
\ \sum_{i = 1}^d [X_i,Y_i] = 0 \right. \right\} \git \GL(n,\C).
$$
We note that $\Char(g,n)$ is an irreducible, normal affine variety of dimension $2 (n^2 (g - 1) + 1)$. 

Set $(g,n)=(2,2)$, so $\dim \Char(g,n) = 10$. We 
 recall results of Kaledin-Lehn  \cite{KaledinLehn}, see
also \cite{LSOGrady}, which explain that $\Char(2,2)$ admits a projective symplectic resolution.

Let $W = \mf{sl}_2$ and $(V,\omega)$ a $4$-dimensional symplectic
vector space. Let $\kappa$ denote the Killing form on $W$. Then
$\kappa \o \omega$ is a symplectic form on $W \o V$. We identify
$\mathfrak{sp}(V)^*$ with $\mathfrak{sp}(V)$ via its Killing
form. There is an action of $\PGL(2)$ on $W$ by conjugation and hence
on $W \o V$. This action is Hamiltonian and commutes with the natural
action of $\Sp(V)$ on $W \o V$. The moment map for the action of
$\PGL(2)$ is given by
\begin{align*}
\mu \left( \sum_i A_i \o v_i \right) & = \sum_{i,j} A_i A_j \omega(v_i,v_j) \\
 & = \sum_{i < j} [A_i,A_j] \omega(v_i,v_j). 
\end{align*}
The moment map for the action of $\Sp(V)$ is given by $\sum_i A_i \o
v_i \mapsto \nu( \sum_i A_i \o v_i )$, where
$$
\nu\left( \sum_i A_i \o v_i  \right)(u) = \sum_{i,j} \kappa(A_i,A_j) \omega(v_i,u) v_j. 
$$
 Since the actions of $\PGL(2)$ and $\Sp(V)$ on $\mu^{-1}(0)$ commute, the map $\nu$ descends to a map $\mu^{-1}(0) \git \PGL(2) \rightarrow
\mf{sp}(V)$, which we also denote by $\nu$. Let $\mc{N}_2^2 \subset \mf{sp}(V)$ be the set $\{ B \ | \ B^2 = 0, \ \mathrm{rk} B = 2 \}$. The set $\mc{N}^2_2$ is a
$6$-dimensional adjoint $\Sp(V)$-orbit. Its closure $\mc{N} :=
\overline{\mc{N}}^2_2 = \mc{N}^2_2 \cup\mc{N}^2_1 \cup \{
0\}$ consists of three $\Sp(V)$-orbits and one can check that
$\overline{\mc{N}}^2_1 \simeq \C^4 / \Z_2$, where $\Z_2$ acts on
$\C^4$ with weights $(-1,-1,-1,-1)$. The following result is proven
in \cite{KaledinLehn}.

\begin{thm}\cite{KaledinLehn}
  The map $\nu$ defines an isomorphism
  $\mu^{-1}(0) \git \PGL(2) \stackrel{\sim}{\longrightarrow} \mc{N}$
  of Poisson varieties. In particular, $\mu^{-1}(0) \git \PGL(2)$ is a
  symplectic singularity.
\end{thm}

Taking trace of the matrices $(X_1,X_2,Y_1,Y_2) \in \Char(2,2)$
defines an isomorphism of symplectic singularities
$\Char(2,2) \simeq \mu^{-1}(0)\git \PGL(2) \times \C^4$,
where $\C^4$ is given the usual symplectic structure. Thus,
$ \Char(2,2) \simeq \mc{N} \times \C^4$.


\subsection{Proof of Theorem \ref{thm:blowup22Q}}\label{sec:blowup22Q-pf}

We now prove Theorem \ref{thm:blowup22Q}, following the arguments of
 \cite[Remark
4.6]{KaledinLehn}; see also \cite{LSOGrady}.



  Since $\alpha$ is anisotropic, 
$2 \alpha$ is also an anisotropic root. Choose a generic stability parameter
  $\theta' \ge \theta$ with $\theta' \cdot \beta \neq 0$ for all
  nonzero $\beta \le 2 \alpha$, $\beta \neq \alpha$. Then the
  projective Poisson morphism
  $\Nak{\lambda}{2\alpha}{\theta'} \rightarrow \Nak{\lambda}{2
    \alpha}{\theta}$
  of Lemma \ref{lem:obvious} is a partial projective resolution. In fact it is birational, since it is an isomorphism over the $\theta$-stable locus (see the proof of Theorem \ref{thm:nonisonores} for more details).
  Thus,
  if $Y \rightarrow \Nak{\lambda}{2\alpha}{\theta'}$ is a projective
  symplectic resolution, then so is the composite
  $Y \rightarrow \Nak{\lambda}{2\alpha}{\theta}$, i.e., it is enough to
  show that we can resolve $\Nak{\lambda}{2\alpha}{\theta'}$
  symplectically. Fix $X = \Nak{\lambda}{2\alpha}{\theta'}$. Then
  $X = X_2 \sqcup X_1 \sqcup X_0$, where, by Theorem
  \ref{thm:stablesmooth}, $X_0$ is the smooth locus consisting of
  $\theta'$-stable representations, $X_1$ parameterizes
  representations $M = M_1 \oplus M_2$ with
  $\dim M_1 = \dim M_2 = \alpha$, $M_1 \not\simeq M_2$ are
  $\theta'$-stable representations and $X_2$ consists of all points
  $M^2$, with $\dim M = \alpha$. By Proposition \ref{prop:strata},
  $X_2$ and $X_2 \cup X_1$ are closed in $X$.

  Let $\widetilde{X}$ denote the blowup of $X$ the along the sheaf of
  ideals of the reduced singular locus $X_1 \sqcup X_0$.
The corollary will follow
  from the following claim: 
  $\widetilde{X} \rightarrow X$ is a projective symplectic
  resolution. 

  Clearly,  $\widetilde{X} \rightarrow X$ is a projective birational morphism,
  therefore we just need to show that $\widetilde{X}$ is smooth and
  the symplectic $2$-form on $X_0$ extends to a symplectic $2$-form on
  $\widetilde{X}$. We check this in a neighborhood of $x \in
  X_2$ and of $y \in X_1$. First consider $x \in X_2$.
  Replacing $X$ by some affine open neighborhood of $x$, Theorem
  \ref{thm:etalelocalgeneral} says that there is an affine $Z$ with
$$
\xymatrix{
  & Z \ar[dl]_{\pi} \ar[dr]^{\rho} & \\
  X & & \Char(2,2) }
$$ where $\pi$ and $\rho$ are \'etale. 
Let
$\widetilde{\Char}(2,2)
\rightarrow \Char(2,2)$, resp. $\widetilde{Z} \rightarrow Z$, denote the blowup along the reduced singular locus. Then
\begin{equation}\label{eq:isosquare}
\widetilde{Z} \simeq \widetilde{X} \times_X Z \simeq \widetilde{\Char}(2,2) \times_{\Char(2,2)} Z. 
\end{equation}
As noted in \cite[Remark 4.6]{KaledinLehn}, $\widetilde{\Char}(2,2)
\rightarrow \Char(2,2)$ is a projective symplectic resolution. Now
Lemma \ref{lem:etalecheck} below and \eqref{eq:isosquare} imply that
$\widetilde{X} \rightarrow X$ is a projective symplectic resolution.

For $y \in X_1$, Theorem \ref{thm:etalelocalgeneral} shows that there
is an \'etale equivalence between a neighborhood of $y$ and a
neighborhood of the origin in a certain quiver variety, independent of
the choice of $y \in X_1$.  In particular such a neighborhood is also
\'etale equivalent to a neighborhood of a point of $X_1$ inside the
neighborhood of $x \in X_2$ used above, so the result follows from the
previous statement. (One can also compute explicitly: the quiver
needed is the one with two vertices, one arrow in each direction
between the two vertices, and also two loops at each vertex, so the
quiver variety is isomorphic to $\C^8 \times \C^2/\Z_2$, which is an
$A_1$ singularity and hence  blowing up the reduced ideal
sheaf of the singular locus gives a projective symplectic resolution).

It remains to prove the following standard lemma:

\begin{lem}\label{lem:etalecheck}
  Let $X$ be a symplectic singularity and
  $\pi : \widetilde{X} \rightarrow X$ a proper morphism. Then
  $\pi$ is a symplectic resolution if and only if it is so
  after a surjective \'etale base change i.e. being a symplectic
  resolution is an \'etale local property.
\end{lem}

Notice that we are not making the (false) claim that $X$ admits a
symplectic resolution if and only if it does so \'etale locally.

\begin{proof}
  Passing to the generic points of $\widetilde{X}$ and $X$, the fact
  that a surjective \'etale morphism is faithfully flat implies that
  $\pi$ is birational if and only if it is so after base
  change. Therefore it suffices to check that the extension $\omega'$
  of the pullback $\pi^* \omega$ is non-degenerate. If
  $b: Z \rightarrow X$ is a surjective \'etale morphism, then so too
  is
  $\widetilde{b} : \widetilde{Z} = \widetilde{X} \times_X Z
  \rightarrow \widetilde{X}$.
  The form $\omega'$ will be non-degenerate if and only if
  $\widetilde{b}^* \omega'$ is non-degenerate.
\end{proof}

\section{Factoriality of quiver varieties}

In this section, which is the technical heart of the paper, we
consider the case of a
divisible anisotropic root. Fix
$\a \in \Sigma_{\lambda,\theta}$ to be an
indivisible anisotropic root, and let $n \ge 2$ such that such that $(p(\a),n) \neq (2,2)$. We prove the key result, Corollary
\ref{cor:factorial}, which says that if $\theta$ is generic
then $\Nak\lambda{n\a}{\theta}$ is a locally factorial variety.

\subsection{}\label{sec:nonisostrata} A \textit{weighted partition} $\nu$ of $n$ is a sequence $(\ell_1, \nu_1; \ds ; \ell_k,\nu_k)$, where $\nu_1 \ge \nu_2 \ge \cdots$ and $\sum_{i = 1}^k \ell_i \nu_i = n$. Recall from Proposition \ref{prop:strata} that the quiver variety $\Nak{\lambda}{\a}{\theta}$ has a finite stratification by representation type. Given a weighted partition $\nu$ of $n$ we can associate naturally a representation type of $n\alpha$:
\begin{equation}
\nu \alpha := (\ell_1, \nu_1 \alpha; \ds ; \ell_k, \nu_k \alpha).
\end{equation}
By the last statement of Theorem \ref{thm:Sigmadivisiblevsdivisible},
if $\alpha\in \Sigma_{\lambda,\theta}$ is anisotropic, then $\nu \alpha$ is indeed a representation type for all partitions $\nu$.
\begin{lem}\label{lem:stratadim}
  Let $\a \in \Sigma_{\lambda,\theta}$ be an indivisible anisotropic root. Let $n \ge 2$. 
\begin{enumerate}

\item We have the formula
$$
\dim \Nak{\lambda}{n\a}{\theta}_{\nu \alpha} = 2 \left( k + (p(\alpha)-1) \sum_{i = 1}^k \nu_i^2 \right).
$$ 
\item For $(p(\a),n) \neq (2,2)$, 
we have $\dim
  \Nak{\lambda}{n\a}{\theta} - \dim \Nak{\lambda}{n\a}{\theta}_{\nu \alpha} \ge 4$ for all
  $\nu \neq (1,n)$.
\item For $(p(\a),n) \neq (2,2)$ and $\nu \neq (1,n)$, we furthermore
  have
  $$
  \dim \Nak{\lambda}{n\a}{\theta} - \dim \Nak{\lambda}{n\a}{\theta}_{\nu \alpha} \ge 8
  $$
  unless one of the following holds: 
  \begin{itemize}
  	\item[(i)] $(p(\a),n)=(2,3)$ and
  	$\nu=(1,2;1,1)$; or 
  	\item[(ii)] $(p(\a),n)=(3,2)$ and $\nu=(1,1;1,1)$. 
  \end{itemize}
\item In the case that $\theta \cdot \beta \neq 0$ for all $\beta \le n \a$ not a multiple of $\a$, all strata  of $\Nak{\lambda}{n\a}{\theta}$ are
of the form 
$\Nak{\lambda}{n\a}{\theta}_{\nu}$, and they are
  parameterized by weighted partitions of $n$.  (It suffices to make the
weaker assumption that
$\Sigma_{\lambda,\theta} \cap \{\beta \mid \beta \le n \a\} \subseteq \{m \a \mid m \leq n\}$.)
\end{enumerate}
\end{lem}
\begin{proof}
We begin with the first claim.
Set $d := p(\a)$. Then $p(n \a) = n^2 (d - 1) + 1$. We have a finite surjective map $\prod_i \Nak{\lambda}{\nu_i \a}{\theta} \to \Nak{\lambda}{n\a}{\theta}_{\nu \alpha}$, so the dimension formula follows from Corollary \ref{cor:dimsympsing}.
%

For the second part, notice that 
\begin{align}\label{e:dimstt}
\dim \Nak{\lambda}{n\a}{\theta} - \dim \Nak{\lambda}{n\a}{\theta}_{\nu} & = 2 (n^2 (d - 1) +1) -  2 \sum_{i = 1}^k (\nu_i^2 (d-1 ) + 1) \notag \\
 & = 2 (d-1) \sum_{i,j = 1}^k (\ell_i \ell_i - \delta_{i,j}) \nu_i \nu_j - 2 (k-1).
\end{align}
Since $ \sum_{i,j = 1}^k (\ell_i \ell_j - \delta_{i,j}) \nu_i \nu_j - (k-1) \ge 1$, we clearly have  $\dim \Nak{\lambda}{n\a}{\theta} - \dim \Nak{\lambda}{n\a}{\theta}_{\nu} \ge 4$ when $d > 2$. When $d = 2$, a simple computation shows that $\dim  \Nak{\lambda}{n\a}{\theta} - \dim \Nak{\lambda}{n\a}{\theta}_{\nu} = 2$ if and only if $n = 2$ and $\nu = (1,1;1,1)$. 

For the third part, we use again \eqref{e:dimstt}, noticing the
following points: the RHS of \eqref{e:dimstt} is increasing in $d$;
the RHS is increased if we replace $(\ell_i,n_i)$ by
$(\ell_i-1,n_i);(1,n_i)$; the RHS is increased if we replace $(1,a)$
and $(1,b)$ by $(1,a+b)$ (when $a+b < n$); and for $a > b > 1$, the
RHS is increased if we replace $(1,a)$ and $(1,b)$ by $(1,a+1)$ and
$(1,b-1)$.  Since it suffices to prove the inequality after performing
operations that increase the RHS, the result follows once we observe
that the inequality holds in the following cases: (i)
$\nu=(1,n-1;1,1)$ whenever $n \geq 4$ as well as $(1,1;1,1;1,1)$; (ii)
for $\nu=(1,1;1,1)$ whenever $p(\a) \geq 4$, as well as $\nu=(2,1)$
for $p(\a)=3$.

For the final claim, observe that each stratum of $\Nak{\lambda}{n\a}{\theta}$ consists of representations of the form $x= x_1^{\oplus \ell_1} \oplus \cdots \oplus x_k^{\oplus \ell_k}$, where the $x_i$ are pairwise non-isomorphic $\theta$-stable representations of fixed dimension vectors $\alpha_i \in \Sigma_{\lambda,\theta}$. Under the assumptions given, each $\alpha_i$ must be a multiple of $\alpha$. Therefore the representation type is of the form $\nu \alpha$ for some weighted partition $\nu$ of $n$.
\end{proof}

Since $p(\alpha) > 1$, there exist infinitely many non-isomorphic $\theta$-stable $\Pi^{\lambda}(Q)$-modules of dimension $\alpha$. Therefore, for all representation types $\nu \alpha = (\ell_1,\nu_1 \alpha;  \ds; \ell_k , \nu_k \alpha)$ with $\sum_i \ell_i \nu_i = n$, the stratum $\Nak\lambda{n \a}\theta_{\nu \alpha}$ is non-empty. Let $U$ be the union of all strata of ``type $\nu \alpha$''. 

\begin{lem}\label{lem:Uopen} 
The subset $U$ is open in $\Nak{\lambda}{n\a}{\theta}$.  
\end{lem}

\begin{proof}
 Since the stratum of representation type $\rho = (n,\alpha)$ is contained in the closure of all the other strata of type $\nu \alpha$, it suffices to show that there is no stratum $\beta = (e_1,\beta^{(1)};\ds; e_l, \beta^{(l)})$ of any other type such that $\Nak\lambda{n \a}\theta_{\rho} \subset \overline{\Nak\lambda{n \a}\theta}_{\beta}$. Assume otherwise. If $G_{\rho} \simeq GL_n(\C)$ is the stabilizer of some $x \in \Nak\lambda{n \a}\theta_{\rho}$, then the Hilbert-Mumford criterion implies that there exists some $y \in \Nak\lambda{n \a}\theta_{\beta}$ whose stabilizer $G_{\beta}$ is contained in $G_{\rho}$. Let $V_i$ be the $n \alpha_i$-dimensional vector space at the vertex $i$ on which $\G(n \alpha)$ acts. Then, for each $g \in \G(n \alpha)$ and $u \in \Cs$, the $u$-eigenspace of $g$ is the direct sum over the $u$-eigenspaces $g |_{V_i}$. In particular, it has a well-defined dimension vector. Now the elements $g$ of $G_{\rho}$ all have the property that the dimension vector of the $u$-eigenspace of $g$ is of the form $r \alpha$ for some $r\in \Z_{\ge 0}$. On the other hand, since $\beta$ is not ``of type $\nu \alpha$'', there is some $i$ such that $e_i \beta^{(i)} \neq r \alpha$ for any $r$. Take $u \neq 1$ and $g \in G_{\beta}$ that rescales the summand of $y$ of dimension $e_i \beta^{(i)}$ by $u$ and is the identity on all other summands. Then the $u$-eigenspace of $g$ has dimension vector $e_i \beta^{(i)}$ which implies that $G_{\beta} \not\subset G_{\rho}$---a contradiction. Thus, $U$ is open. 
\end{proof}

The open subset of $\mu^{-1}(\lambda)^\theta$ consisting of stable representations is denoted $\mu^{-1}(\lambda)_{s}^{\theta}$, and its image in $\Nak{\lambda}{n\a}{\theta}$ is denoted $\Nak{\lambda}{n\a}{\theta}^s$. Note that $\Nak{\lambda}{n\a}{\theta}^s$ is an open subset of $U$, and the quotient map $\mu^{-1}(\lambda)_{s}^{\theta} \rightarrow \Nak{\lambda}{n\a}{\theta}^s$ is a principal $\PG(n\alpha)$-bundle.

\subsection{Factoriality of
  $\Nak\lambda{n\a}{\theta}$}\label{sec:localfact62}

A closed point $x \in X$ is said to be factorial if the local ring
$\mc{O}_{X,x}$ is a unique factorization domain. We say that $X$ is
locally factorial if every closed point of $X$ is factorial.
If
$\xi : \mu^{-1}(\lambda)^{\theta} \rightarrow \Nak\lambda{n\a}\theta$
is the quotient map, then let $V = \xi^{-1}(U)$, where $U$ is the open
subset of Lemma \ref{lem:Uopen}. We will need the following result from \cite{CBnormal}:

\begin{thm} \cite[Theorem 6.3, Corollary
  6.4]{CBnormal} \label{t:cbnormal} Consider a stratum $Z$ in
  $\Nak\lambda\beta\theta$ of representation type
  $(k_1,\beta^{(1)}; \ldots; k_r,\beta^{(r)})$. Then for all
  $z \in Z$, $\xi^{-1}(z) \subseteq \mu^{-1}(\lambda)^\theta$ has
  dimension at most
  $\beta \cdot \beta - 1 + p(\beta) - \sum_t p(\beta^{(t)})$, so the
  dimension of $\xi^{-1}(Z)$ is at most
  $\beta \cdot \beta - 1 + p(\beta) + \sum_t p(\beta^{(t)})$.
\end{thm}
We note that in \cite{CBnormal}, this is stated and proved for $\lambda$ and $\theta$
equal to zero, but the proof and result extends verbatim to the general case, replacing simple modules by $\theta$-stable modules. In the case that $\beta \in \Sigma_{\lambda,\theta}$, applying Proposition \ref{p:flat-ci} immediately yields
\begin{cor}\label{c:cbnormal2}
  The codimension of $\xi^{-1}(Z)$ is at least
  $\frac{1}{2} \codim(Z) = p(\beta)-\sum_t p(\beta^{(t)})$.
\end{cor}

\begin{prop}\label{prop:Vproperties}
	$V$ is a local complete intersection, locally factorial and
	normal. 
\end{prop} 

\begin{proof}
	Since $\alpha \in \Sigma_{\lambda,\theta}$, Proposition \ref{p:flat-ci} implies that $\mu^{-1}(\lambda)^\theta$ and hence
  $V$ is a local complete intersection of dimension $n^2\alpha \cdot \alpha - 1 + 2p(n\alpha)$. 

The main step is the second assertion. For this we will show that $V$ is smooth outside a subset of codimension four, i.e., it satisfies the $R_3$ property. For any $\G(n\alpha)$-stable subset $X$ of $\mu^{-1}(\lambda)^{\theta}$, we write $X_{\free}$ for the subset of all points where $\PG(n\alpha)$ acts freely. The assertion will follow from estimating the codimension of the complement to $V_{\free}$ in $V$. Note that the free locus is the same as the locus of representations whose endomorphism algebra has dimension one, i.e., the ``bricks''. We represent $\mu^{-1}(\lambda)^\theta$ as the union of preimages of the (finitely many) strata, and consider over each such preimage the non-free locus. 

If the preimage of the stratum has codimension at least four, it can be ignored. Thus, we just need to show that the complement to $\xi^{-1}(Z)_{\mathrm{free}}$ has codimension at least four for those strata $Z$ with $\mathrm{codim} \ \overline{\xi^{-1}(Z)} \le 3$. Since we are explicitly excluding the case $(p(\alpha),n) = (2,2)$, Corollary \ref{c:cbnormal2}, together with Lemma \ref{lem:stratadim} (3), imply that we are reduced to considering the cases $(p(\alpha),n) = (2,3)$ and $\nu = (1,2;1,1)$, or $(p(\alpha),n) = (3,2)$ and $\nu = (1,1;1,1)$.     

Observe first that if $Z$ is a stratum, then the polystable part of the preimage $\xi^{-1}(Z)$ has codimension (in $\mu^{-1}(\lambda)^{\theta}$) at least the codimension of $Z$ itself (in $\Nak\lambda{n\a}\theta$), since the fiber over a polystable representation $M$ has dimension $\alpha \cdot \alpha - \dim \End(M)$, which is maximized
when $M$ is stable. Thus if $Z$ has codimension at least four (which is the case for us), then we can
ignore the polystable part of $\xi^{-1}(Z)$.

Next, if we consider a stratum $Z$ of type $(1,a;1,b)$, note that every representation in this stratum is either polystable or an indecomposable extension of two non-isomorphic representations. The latter type is a brick, since there is a unique stable quotient and a unique stable subrepresentation and the two are nonisomorphic. Therefore applying the previous paragraph together
with  Lemma \ref{lem:stratadim} (2) shows that we can ignore $\xi^{-1}(Z)$ (the non-free locus has overall codimension at least four). This proves the final assertion.

  Since $\mu$ is regular on the locus where $\PG(n\alpha)$ acts freely,
  $\mu^{-1}(\lambda)_{\free}^{\theta}$ lies in the smooth locus of
  $V$.  We conclude from the last assertion of the proposition that
  the singular locus of $V$ has codimension at least $4$ (i.e.,
  property $R_3$ holds).  Since $V$ is a local complete intersection,
  and hence Cohen-Macaulay, it satisfies Serre's condition $S_2$, so
  it is normal.

  Finally, it follows from a theorem of Grothendieck, \cite[Theorem
  3.12]{KaledinLehnSorger}, that since $V$ is a complete intersection
  and satisfies $R_3$, the local ring $\mc{O}_{V,x}$ of any point $x \in V$ is a unique factorization domain. 
\end{proof}

The result that allows us to descend local factoriality from $V$ to the
quotient $U$ is the following theorem by Drezet. Since the version
given in \cite{Drezet} concerns the moduli space of semistable
sheaves on a smooth surface, we provide full details to ensure the
arguments are applicable in our situation. Let $G$ be a connected reductive group.  

\begin{lem}\label{lem:extendPG}
Let $V$ be a locally factorial normal affine $G$-variety and $V_s \subset V$ a dense open subset of $V$, whose complement has codimension at least two in $V$. Then every $G$-equivariant line bundle on $V_s$ extends to a $G$-equivariant line bundle on $V$.  
\end{lem}

\begin{proof}
The fact that $V$ is normal and locally factorial implies that
$$
\mathrm{Pic} (V) = \mathrm{Div}(V) = \mathrm{Div} \left(V_s \right) = \mathrm{Pic} \left( V_s \right). 
$$ 
Hence if $L_0$ is a $G$-equivariant line bundle on $V_s$, forgetting the equivariant structure, there is an extension $L$ to $V$. To show that the extension $L$ has a $G$-equivariant structure, one repeats the
argument of \cite[Lemme 5.2]{DrezetNarasimhan}, which uses the fact that the codimension of $V \smallsetminus V_s$ is at least two. 
\end{proof}

\begin{thm}[\cite{Drezet}, Theorem A]\label{thm:Drezet}
	Let $V$ be a locally factorial, normal $G$-variety, with good quotient $\xi: V \rightarrow U := V \git \, G$. Assume that there exists an open subset $U_s \subset U$ such that 
	\begin{enumerate}
		\item[(a)] the complement to $U_s$ has codimension at least two in $U$,
		\item[(b)] $V_s := \xi^{-1}(U_s) \rightarrow U_s$ is a principal $G$-bundle; and
		\item[(c)] the complement to $V_s$ has codimension at least two in $V$. 
	\end{enumerate} 
	Let $x \in U$ and $y \in T(x)$ a lift in $V$ (so that $G \cdot y$ is closed in $V$). The following are equivalent:
\begin{enumerate}
\item[(i)] The local ring $\mc{O}_{U,x}$ is a unique factorization domain.
\item[(ii)] For every line bundle $M_0$ on $U_s$, there exists an open subset
   $U_0 \subset U$ containing both $x$ and
  $U_s$ such that $M_0$ extends to a
  line bundle $M$ on $U_0$.
\item[(iii)] For every $G$-equivariant line bundle $L$ on $V$,
  the stabilizer of $y$ acts trivially on the fiber $L_y$.
\end{enumerate}
\end{thm} 

\begin{proof}
Recall that $\mc{O}_{U,x}$ is a unique factorization domain if and
only if every height one prime is principal. Geometrically, this means
that for every hypersurface $Y$ of $U$, the sheaf of ideals $\mc{I}_Y$
is free at $x$.

(i) implies (ii). It suffices to assume that $M_0 = \mc{I}_Y$, where
$Y$ is a hypersurface in $U_s$. If $\overline{Y}$ is the closure of $Y$ in $U$, then $M =
\mc{I}_{\overline{Y} \cap U_0}$ is the required extension.

(ii) implies (i). Let $Y$ be a hypersurface in $U$. We wish to show
that $\mc{I}_Y$ is free at $x$. Let $M$ be the extension of $\mc{I}_Y|_{U_s}$
to $U_0$. The line bundle $M$ corresponds to a Cartier divisor $D$ on
$U_0$; $M= \mc{O}_{U_0}(D)$. Then,
$$
\mc{I}_{Y} |_{U_s} = \mc{O}_{U_s}(D \cap U_s), 
$$ and the divisors $Y$ and $-D \cap
U_s$ are linearly equivalent. Since, by assumption, the
codimension of the complement to $U_s$ in $U$ has codimension at least two and $U$ is normal, $\overline{Y}
\simeq - D$. Hence $M = \mc{I}_{\overline{Y}}$ is free at $x$.

(ii) implies (iii). Suppose that $L$ is a $G$-equivariant line
bundle on $V$. Since $G$ acts freely on
$V_s$, the restriction $L
|_{V_s}$ descends to the line bundle
$M_0 = (L |_{V_s}) / G$ on
$U_s$. Let $M$ be the extension of $M_0$
to $U_0$. Then the $G$-equivariant line bundle $\xi^* M$ agrees
with $L$ on $V_s$. This implies, as in
the previous paragraph, that $\xi^* M = L$ on $\xi^{-1}(U_0)$. In
particular, since $y \in \xi^{-1}(U_0)$, the stabilizer of $y$ acts
trivially on $L_y$.

(iii) implies (ii). Let $M_0$ be a line bundle on
$U_s$. By Lemma \ref{lem:extendPG},
$\xi^* M_0$ extends to a $G$-equivariant line bundle $L$ on
$V$. Recall by definition of lift that $G \cdot y$ is closed in
$V$. Therefore Lemma \ref{lem:closedfiberaction} below says that there is an affine open
neighborhood $U'$ of $x$ such that $G_{y'}$ acts trivially on
$L_{y'}$ for all $y' \in \xi^{-1}(U')$ such that $G \cdot y'$ is
closed in $V$. Let $U_0 = U' \cup
U_s$. Then, by descent \cite[Theorem
  1.1]{Drezet}, there exists a line bundle $M$ on $U_0$ such that
$\xi^* M \simeq L$. In particular, $M$ extends $M_0$.
\end{proof}

Let $Y$ be a variety admitting an algebraic action of a reductive
group $G$. Assume that there exists a good quotient $\xi : Y
\rightarrow X = Y \git \, G$. The following result, which says that the
descent locus of an equivariant line bundle is open, is presumably
well-known, but we were unable to find it in the literature.

\begin{lem}\label{lem:closedfiberaction}
Let $L$ be a $G$-equivariant line bundle on $Y$ and $y \in Y$ a closed
point such that the orbit $\mc{O} = G \cdot y$ is closed and the
stabilizer $G_y$ of $y$ acts trivially on the fiber $L_y$. Then there
exists an affine open neighborhood $U$ of $\xi(y)$ such that the
stabilizer $G_{y'}$ acts trivially on $L_{y'}$ for all $y' \in
\xi^{-1}(U)$ such that $G \cdot y'$ is closed.
\end{lem}

\begin{proof}
The proof of the lemma can be easily deduced from the proof of
\cite[Theorem 2.3]{DrezetNarasimhan}. It is shown there
that one can find a $G$-invariant section $s' : \mc{O} \rightarrow L
|_{\mc{O}}$, which trivializes $L |_{\mc{O}}$. As explained in
\emph{loc.~cit.}, the fact that $\mc{O}$ is closed in $Y$ implies
that one can lift $s'$ to a $G$-invariant section $s \in
\Gamma(\xi^{-1}(U'), L)$, where $U'$ is some affine open neighborhood
of $\xi(y)$. Let $W$ be the (non-empty) open subset of $\xi^{-1}(U')$
consisting of all points $y'$ such that $s(y') \neq 0$ i.e. $s$
trivializes $L$ over $W$. Then it suffices to show that there is some
affine neighborhood $U$ of $\xi(y)$ such that $\xi^{-1}(U) \subset
W$. Again, following \cite[Theorem 2.3]{DrezetNarasimhan},
the sets $\xi^{-1}(U')
\smallsetminus W$ and $\mc{O}$ are $G$-stable closed subsets of
$\xi^{-1}(U')$. Therefore the fact that $\xi$ is a good quotient
implies that $\xi(\xi^{-1}(U') \smallsetminus W)$ and $\xi(\mc{O}) =\{
\xi(y) \}$ are closed, disjoint subsets of $U'$. Thus, there exists an
affine neighborhood $U$ of $\xi(y)$ such that $U \cap \xi(\xi^{-1}(U')
\smallsetminus W) = \emptyset$, as required.
\end{proof}  

\begin{cor}\label{cor:factorial}
Assume that $(p(\a),n) \neq (2,2)$. Then the variety $U$ is locally factorial. 
\end{cor}

\begin{proof}
	Let $G = \PG(\a)$ and $U_s = \Nak\lambda{n\a}{\theta}_s$. Proposition \ref{prop:Vproperties} implies that $V$ is normal and locally factorial. This implies that $U$ is normal. Moreover, by Lemma \ref{lem:stratadim}, the codimension of the complement to $U_s$ in $U$ has codimension at least two. Thus, assumptions (a) and (b) of Theorem \ref{thm:Drezet} are satisfied. By Lemma \ref{lem:stratadim} (2) and Corollary \ref{c:cbnormal2}, the complement to $\mu^{-1}(\lambda)_{s}^{\theta}$ in $V$ has codimension at least $4$. In particular, assumption (c) of Theorem \ref{thm:Drezet} is also satisfied.  
	
	Next, recall from the proof of Lemma \ref{lem:Uopen}, the stratum of
type $\rho = (n,\a)$ is contained in the closure of all other strata
in $U$. If $y$ is a lift in $\mu^{-1}(\lambda)^{\theta}$ of a point of
$\Nak\lambda{n\a}{\theta}_{\rho}$ then $y$ corresponds to a
representation $M_0^{\oplus n}$, where $M_0 \in
\Nak\lambda{\a}{\theta}$ is a stable
$\Pi^{\lambda}(Q)$-module. Therefore $\PG(n\a)_y = PGL_n$ has no
non-trivial characters. In particular, $\PG(n\a)_y$ will act trivially
on $L_y$ for any $\PG(n\a)$-equivariant line bundle on $V$. Hence, we
deduce from Theorem \ref{thm:Drezet} that $\Nak\lambda{n\a}{\theta}$
is factorial at every point of $\Nak\lambda{n\a}{\theta}_{\rho}$.

Now consider an arbitrary stratum $\Nak\lambda{n\a}{\theta}_{\nu \alpha}$ in
$U$. If $\Nak\lambda{n\a}{\theta}$ is factorial at one point of the
stratum then it will be factorial at every point in the stratum (for a
rigorous proof of this fact, repeat the argument given in the proof of
\cite[Theorem 5.3]{KaledinLehnSorger}).  On the other hand, a
theorem of Boissi\`ere, Gabber and Serman \cite{OpenFactorial} says
that the subset of factorial points of $U$ is an open subset. Since
this open subset is a union of strata and contains the unique closed
stratum, it must be the whole of $U$.
\end{proof}

\begin{rem}
Notice that if $\theta$ is generic then $U =
\Nak\lambda{n\a}{\theta}$. Hence Corollary \ref{cor:factorial}
says that $\Nak\lambda{n\a}{\theta}$ is a locally factorial variety. This is precisely the statement of Theorem \ref{t:fact}. In the special case where $Q$ has a single vertex, $g$ loops and $\a = e_1$, the parameters $\lambda = \theta = 0$ are generic. Then $\Nak0{n}{0}$ is an affine cone and thus has trivial Picard group. Corollary \ref{cor:factorial} implies that when $(g,n) \neq (2,2)$, $\Nak0{n}{0}$ is actually factorial. That is, $\C[\Nak0{n}{0}]$ is a unique factorization domain.  
\end{rem}




\subsection{The proof of Theorem \ref{thm:sympsing}}\label{sec:thmsympsing}

By Proposition \ref{prop:Mirrnormal}, we know that
$\Nak\lambda\a\theta$ is irreducible and normal. Therefore, it
suffices to show that it admits symplectic singularities. Since the
isomorphism of Theorem \ref{thm:decompCB2} is Poisson, it suffices to
show that the varieties $S^{n_i} \Nak\lambda{\sigma^{(i)}}\theta$
admit symplectic singularities. If $\sigma^{(i)}$ is real there is
nothing to check.

\begin{lem}\label{lem:pre}
Let $X$ be a smooth irreducible Poisson variety and $Y$ a smooth symplectic variety. If $\pi : Y \rightarrow X$ is a birational, surjective Poisson morphism, then it is an isomorphism. 
\end{lem}

\begin{proof}
  Since the morphism is birational, there is a dense open subset $U \subset X$ over which it is an isomorphism. We claim that we can take $U$ to have complement of codimension at least two.  Indeed, if $Z \subseteq X$ is any irreducible subset of codimension one, then localising at $Z$, by \cite[I, Corollary 6.12]{Hartshorne}, the birational map is an isomorphism. So $U$ can be taken to have nontrivial intersection with $Z$.
  Now,
  since $X$ is smooth, the locus where the Poisson structure on $X$ is degenerate has codimension one. Therefore, $X$ is symplectic too. This implies that $d_y \pi$ is an isomorphism for all $y \in Y$. Thus, by Zariski's Main Theorem, $\pi$ is an isomorphism. 
\end{proof}

\begin{lem}\label{lem:YtoXsymp}
Let $X$ be a normal irreducible Poisson variety and assume that $\pi : Y \rightarrow X$ is a proper birational Poisson morphism from a variety $Y$ with symplectic singularities. Then $X$ has symplectic singularities. 
\end{lem}
 Observe that one can drop the assumption of the lemma that $X$ is Poisson and $\pi$ is a Poisson morphism: since $R^0 \pi_* \mc{O}_Y = \mc{O}_X$, $Y$ endows $X$ with a unique Poisson structure making $\pi$ Poisson.
\begin{proof}
  Let $\rho : Z \rightarrow Y$ be a resolution of
  singularities. If $\varpi$ is the symplectic $2$-form on the smooth
  locus of $Y$ then $\rho^* \varpi$ extends to a regular form on
  $Z$. Let $U = \pi^{-1}(X_{\mathrm{sm}})$. Then, since
  $\pi : U \rightarrow X_{\mathrm{sm}}$ is proper and birational and $X_{\mathrm{sm}}$ is irreducible, $\pi$ 
  is surjective. Now, let $U' := U \cap Y_{\mathrm{sm}}$.  
  Lemma \ref{lem:pre} implies that $\pi|_{U'}: U' \to \pi(U')$ is an
  isomorphism. Thus $\pi(U') \subseteq X_{\mathrm{sm}}$ has a symplectic $2$-form induced by the Poisson structure.  Since $\pi(U') \subseteq X_{\mathrm{sm}}$ has complement of codimension two, the Poisson structure cannot degenerate here, and hence $X_{\mathrm{sm}}$ is itself symplectic.  Now suppose that $y \in U$ is singular.  There is a Hamiltonian vector field at $\pi(y)$ which flows into $\pi(U')$. The Hamiltonian vector field of the pullback function must then flow from $y$ into $U'$.  This contradicts the fact that the singular locus of $U$ is a union of symplectic leaves, hence closed under the Hamiltonian flow.  So $U=U'$, which is symplectic. By Lemma \ref{lem:pre}, $\pi: U \to X_{\mathrm{sm}}$ is an isomorphism.  
  Let $\omega$ be the symplectic two-form on $X_{\mathrm{sm}}$. Then
  $\pi^* \omega = \varpi|_U$. Thus,
  $(\pi \circ \rho)^* \omega = \rho^* \varpi$ extends to a regular
  form, and hence $X$ has symplectic singularities.
  \end{proof}

If $\sigma^{(i)}$ is an indivisible anisotropic root then
$n_i = 1$ by Corollary \ref{c:ni=1}. Choosing a generic stability parameter $\theta' \ge \theta$ defines a projective, Poisson resolution
$\Nak\lambda{\sigma^{(i)}}{\theta'} \rightarrow
\Nak\lambda{\sigma^{(i)}}\theta$
with $\Nak\lambda{\sigma^{(i)}}{\theta'}$ a smooth symplectic variety; see
\cite[Section 8]{CBnormal}. Similarly, if $\sigma^{(i)}$ is isotropic
imaginary then it is well-know that one can frame the quiver so that
there exists a projective, Poisson resolution of singularities from a
quiver variety that is a smooth symplectic variety. Thus, Lemma
\ref{lem:YtoXsymp} implies that
$S^{n_i} \Nak\lambda{\sigma^{(i)}}\theta$ admits symplectic
singularities in these two cases.

Therefore we may assume that there exists an
indivisible anisotropic root $\beta$ such that $\alpha = n \beta$, for some $n > 1$. Let $g = p(\beta)$. If $(g,n) = (2,2)$, then Theorem \ref{thm:blowup22Q} and Lemma \ref{lem:YtoXsymp}
imply that $\Nak\lambda{\sigma^{(i)}}\theta$ has symplectic
singularities. Therefore, it suffices to show that
$\Nak\lambda{\sigma^{(i)}}\theta$ has symplectic singularities when
$(g,n) \neq (2,2)$. Again, choose a generic stability parameter
$\theta' \ge \theta$. Then
$\pi :\Nak\lambda{\sigma^{(i)}}{\theta'} \rightarrow
\Nak\lambda{\sigma^{(i)}}\theta$
is projective and Poisson by Lemma \ref{lem:obvious}. Moreover, since
both $\Nak\lambda{\sigma^{(i)}}{\theta'}$ and
$\Nak\lambda{\sigma^{(i)}}\theta$ are irreducible by Proposition
\ref{prop:Mirrnormal}, and a generic element of
$\Nak\lambda{\sigma^{(i)}}\theta$ is $\theta$-stable, the map $\pi$ is
birational (see the proof of Theorem \ref{thm:nonisonores} below for more details). Thus, by Lemma \ref{lem:YtoXsymp}, it suffices to show
that $\Nak\lambda{\sigma^{(i)}}{\theta'}$ admits symplectic
singularities. This follows from Flenner's Theorem \cite{Flenner},
once we show that the singular locus of
$\Nak\lambda{\sigma^{(i)}}{\theta'}$ has codimension at least four. By
Theorem \ref{thm:stablesmooth}, the singular locus of
$\Nak\lambda{\sigma^{(i)}}{\theta'}$ is the union of all strata except
the open stratum. By Lemma \ref{lem:stratadim} (2) each of these strata
has codimension at least $4$ in $\Nak\lambda{\sigma^{(i)}}{\theta'}$.

\subsection{The proof of Theorems \ref{thm:main0}, \ref{thm:main1}, and \ref{thm:proper}}\label{sec:Thm01proof}

We begin by considering the case of a divisible anisotropic root. Recall that $\alpha$ is $\Sigma$-divisible if $\alpha = m\beta$ for some $\beta \in \Sigma_{\lambda,\theta}$ and $m \geq 2$.

Recall that a normal variety $X$ with $\Q$-Cartier canonical divisor $K_X$ is said to have \textit{terminal singularities} if $K_Y = f^*(K_X) + \sum_i a_i E_i$ with $a_i > 0$, where $f : Y \rightarrow X$ is any resolution of singularities, and the sum is over all exceptional divisors of $f$.


\begin{thm}\label{thm:nonisonores}
  Let $\alpha \in \Sigma_{\lambda,\theta}$ be a
  divisible
  anisotropic root which is not of the form $\alpha=2\beta$ for $p(\beta)=2$.
  Then the symplectic singularity
  $\Nak\lambda{\a}\theta$ does not admit a projective symplectic
  resolution. If, moreover, $\alpha$ is $\Sigma$-divisible, then it does not admit a proper symplectic resolution.
\end{thm}

\begin{proof}
  Write $\alpha=n\beta$ for $n \geq 2$.  First suppose that
  $\beta \in \Sigma_{\lambda,\theta}$.  If
  $\Nak\lambda{n \beta}\theta$ admits a proper symplectic resolution
  then so too by restriction does the open subset $U$ of Lemma
  \ref{lem:Uopen}. Recall from Theorem \ref{thm:stablesmooth} that the
  singular locus of $U$ is the complement of the open stratum. The
  singular locus is also nonempty since the stratum $(n,\beta)$ is
  nonempty.  By Lemma \ref{lem:stratadim}, it has codimension at
  least four in $U$. Therefore, since $U$ has symplectic singularities
  by Theorem \ref{thm:sympsing}, \cite{NamikawaNote} says that $U$ has
  terminal singularities. This implies that if $f : Y \rightarrow U$
  is a proper symplectic resolution then the exceptional locus of $f$
  has codimension at least two in $Y$. On the other hand, we have
  shown in Corollary \ref{cor:factorial} that $U$ is locally factorial. This
  implies by van der Waerden purity, see \cite[Section 1.40]{Debarre},
  that the exceptional locus of $f$ is a divisor. This is a
  contradiction.

  Finally suppose that $\beta \notin \Sigma_{\lambda,\theta}$. Since it is indivisible, it is clear that for generic $\theta'$, we have $\beta \in \Sigma_{\lambda,\theta'}$. It follows from the previous paragraph that $\Nak\lambda\a{\theta'}$ does not admit a proper symplectic resolution, hence not a projective symplectic resolution.

  We claim that, for generic $\theta' > \theta$, $\Nak\lambda\a{\theta'} \to \Nak\lambda\a{\theta}$ is a projective birational Poisson morphism. For generic $\theta'$, note that every $\theta$-stable representation is also $\theta'$-stable: this is because we can assume that for $\beta < \a$ such that $\theta \cdot \beta < 0$, we also have $\theta' \cdot \beta < 0$. Now, since every $\theta'$-semistable representation is also $\theta$-semistable, and the $\theta$-stable representations are dense in the $\theta$-semistable ones (by Corollary \ref{c:dense}), it follows that the $\theta$-stable locus is dense in $\Nak\lambda\a{\theta'}$.
  Since stable orbits in $\mu^{-1}(\lambda)$ are closed, in this case the locus of
  $\theta$-stable representations in $\Nak\lambda\a{\theta'}$ maps isomorphically
  onto the stable locus in  $\Nak\lambda\a{\theta}$. The latter being dense as well, and stability being an open condition, we conclude that the map $\Nak\lambda\a{\theta'} \to \Nak\lambda\a{\theta}$ is birational. The other 
  conditions follow from Lemma \ref{lem:obvious}.

  Now, suppose that $\Nak\lambda\a\theta$ admitted a projective symplectic resolution. If $\lambda=\theta=0$, then $\Nak\lambda\a\theta$ is conical. By \cite[Theorem 2.2]{BSnon} (which is based on \cite{Namikawa}), it would then follow that $\Nak\lambda\a{\theta'}$ also admitted a projective symplectic resolution, which is a contradiction.

  For the general case, if $\Nak\lambda\a\theta$ admitted a projective symplectic resolution, it would also do so 
  \'etale-locally. By first decomposing $\beta$ into elements of $\Sigma_{\lambda,\theta}$, we could find a representation type $\tau = (ne_1, \beta^{(1)}; \ds ; ne_k, \beta^{(k)})$ for $\alpha$ with coefficients multiples of $n$.  Then, the \'etale-local quiver $(Q',\a')$ at this stratum would have $n \mid \gcd(\a')$.  Moreover, the anisotropic root $\a'$ belongs to $\Sigma_{0,0}$, as the generic representation remains stable (one can also prove this directly). It would follow that $\Nak{0}{\a'}{0}$ admits a projective resolution, contradicting the previous paragraph.
\end{proof} 
The proof also shows, more generally, that any singular open subset of
$U$ in Lemma \ref{lem:Uopen} (in the $\Sigma$-divisible case)
does not admit a symplectic resolution.  If $\theta$ is generic,
then $U = \Nak\lambda{n\a} \theta$.  This implies Corollary
\ref{cor:openfact} (as well as the stronger result discussed
afterwards).

\begin{cor}\label{c:proper=proj}  Let $\alpha\in \Sigma_{\lambda,\theta}$ be
  indivisible and $n \geq 1$. Then $\Nak\lambda{n \a}\theta$
  admits a projective symplectic resolution if one of the following
  conditions hold: 
  \begin{enumerate}
  	\item[(o)] $\a$ is a real root ($p(\alpha)=0$); 
  	\item[(i)] $n=1$;
  	\item[(ii)] $p(\alpha)=1$; or
  	\item[(iii)] $(n,p(\alpha)) = (2,2)$.
  \end{enumerate}
   If none of these conditions hold, then $\Nak\lambda{n \a}\theta$ does not admit a proper symplectic resolution. In particular, existence of
  projective and proper symplectic resolutions is equivalent for
  $\Nak\lambda{n \a}\theta$.
\end{cor}

\begin{proof} In case (o), $\Nak \lambda{n\a}\theta$ is a point, so
  there is nothing to show. In case (i), let $D$ be the open subset of $\{ \theta' \in \Q^{Q_0} \ | \ \theta'(\alpha) = 0\}$ consisting of all stability conditions vanishing on $\alpha$, but not on any other $\beta < \alpha$ with $\beta \in \Sigma_{\lambda,\theta}$. Since $\alpha$ is
  indivisible, the set $D$ is non-empty. Its closure is the whole space, thus there exists a connected component $C$ such that $\theta \in \overline{C}$. Choose $\theta' \in C$ (rescaling if necessary, we may assume that $\theta' \in \Z^{Q_0}$). Then, just as shown in \cite[Section 8]{CBnormal}, the morphism   $\Nak \lambda{ \alpha}{\theta'} \to \Nak \lambda{ \alpha}\theta$ is
  a projective symplectic resolution. For case (ii), first note that case (i) implies that $X:=\Nak\lambda{\a}{\theta'} \to \Nak\lambda{\a}{\theta}$ is a projective symplectic resolution of (du Val) singularities for some $\theta' \ge \theta$. In particular, $X$ is a smooth symplectic surface. Next, $\Nak \lambda{n\alpha}\theta \cong S^n \Nak \lambda{\alpha} \theta$ by Theorem \ref{thm:CBdecomp} because the canonical decomposition of $n\alpha$ is $\alpha+ \cdots + \alpha$. We therefore obtain a partial resolution $S^n X \to \Nak\lambda{n\alpha}\theta$. Now recall that the natural map
$\Hilb^n  \Nak \lambda \a {\theta'} \to S^n \Nak \lambda \a{\theta'}$ is a projective symplectic resolution; see \cite[Theorem 1.8, Theorem 1.10]{NakajimaBook}.  
Finally in case (iii) the resolution is given in Theorem \ref{thm:blowup22Q}.   If none of the conditions hold, then $\alpha$ is an anisotropic root and the non-existence of a proper symplectic resolution is a consequence of Theorem \ref{thm:nonisonores}.
\end{proof}

We now proceed to the proof of Theorem \ref{thm:main0}.
The isomorphism of Theorem \ref{thm:main0} follows directly from
Theorem \ref{thm:decompCB2}. Therefore, it suffices to show that
$\Nak\lambda\a\theta$ admits a projective symplectic resolution if and only if
each $\Nak\lambda{\sigma^{(i)}}\theta$ admits a projective symplectic resolution.

First note that if $\sigma^{(i)}$ is a real root or an isotropic
root, then $\sigma^{(i)}$ is 
indivisible (for the latter
property, a 
divisible isotropic root is not in
$\Sigma_{\lambda,\theta}$).  In these cases, as recalled in Corollary
\ref{c:proper=proj}, $S^{n_i}\Nak\lambda{\sigma^{(i)}}\theta$ admits a
projective symplectic resolution, as does
$\Nak\lambda{\sigma^{(i)}}\theta$.

On the other hand, if $\sigma^{(i)}$ is an anisotropic root, then $n_i = 1$, by
Corollary \ref{c:ni=1}.
Moreover, by Corollary \ref{c:proper=proj}, $\Nak\lambda{\sigma^{(i)}}\theta$ admits a projective symplectic resolution if $\sigma^{(i)}$ is 
indivisible or if $\sigma^{(i)}$ is twice a 
root $\beta \in \Sigma_{\lambda,\theta}$ satisfying $p(\beta)=2$, and otherwise it does not.

Therefore, if $\Nak\lambda{\sigma^{(i)}} \theta$ admits a projective
symplectic resolution for all $i$, it follows that each
$S^{n_i}\Nak\lambda{\sigma^{(i)}}\theta$ admits a projective symplectic
resolution, and hence so does $\Nak\lambda\a\theta$. 




If, on the other hand, some $\Nak\lambda{\sigma^{(i)}}\theta$ did not
admit a projective symplectic resolution, then Corollary \ref{c:proper=proj} implies that $\sigma^{(i)}$ is a divisible anisotropic root which is
not twice a root $\beta \in \Sigma_{\lambda,\theta}$ satisfying $p(\beta)=2$.  In
this case, by the proof of Theorem \ref{thm:nonisonores},
\'etale-locally $\Nak\lambda{\sigma^{(i)}}\theta$ is a cone which admits a partial projective Poisson resolution $\Nak\lambda{\sigma^{(i)}}{\theta'}$ which itself is terminal, locally factorial, and singular. Taking products with the other factors, we see that \'etale-locally $\Nak\lambda\a\theta$ is itself a cone with a partial crepant resolution by a terminal, locally factorial, and singular variety.  Such a variety does not admit a symplectic resolution, as explained in the proof of Theorem \ref{thm:nonisonores}. By \cite[Theorem 2.2]{BSnon}, $\Nak\lambda \a \theta$ itself does not admit a projective symplectic resolution. This completes the proof.


Notice that Theorems \ref{thm:main1} and \ref{thm:proper} also follow from the above argument.

\subsection{Formal resolutions}

Let $\a$ be a divisible anisotropic root, and assume that $\alpha$ is not
of the form $\alpha=2\beta$ with $p(\beta)=2$. Though it might not be
obvious from Corollary \ref{cor:openfact}, the nature of the
obstructions to the existence of a symplectic resolution of
$\Nak\lambda{\a}\theta$ is quite subtle. We have shown that Zariski
locally no resolution exists.
But then one
can ask if a resolution exists \'etale locally, or in the formal
neighborhood of a point? In this section we give a precise answer to
this question.

\begin{defn}
The closed point $x \in \Nak{\lambda}{\a}{\theta}$ is said to be \emph{formally resolvable} if the formal neighborhood $\widehat{\mathfrak{M}}_{\lambda}(\a,\theta)_x$ of $x$ in $ \Nak{\lambda}{\a}{\theta}$ admits a projective symplectic resolution. 
\end{defn}

\begin{lem}\label{lem:formal0}
If $0 \in \mathfrak{M}_0(\alpha,0)$ is formally resolvable, then $\mathfrak{M}_0(\alpha,0)$ also admits a projective symplectic resolution, and conversely.
\end{lem}

\begin{proof}
  Let $\Cs$ act on $\Rep(\overline{Q},\alpha)$ by dilations. Then the
  moment map $\mu$ is homogeneous of degree two and the action of
  $\G(\a)$ commutes with the action of $\Cs$. This implies that
  $\C[\Nak{\lambda}{\a}{\theta}]$ is an $\N$-graded, connected
  algebra. Note also that the Poisson bracket on
  $\C[\Nak{\lambda}{\a}{\theta}]$ has degree $-2$. The lemma
  follows from standard arguments; see \cite[Proposition 5.2]{GK},
  \cite[Theorem 1.4]{KaledinDynkin}, and the references therein. The
  idea is that: 1) The induced action of $\Cs$ on
  $\widehat{\mathfrak{M}}_{0}(\a,0)_0$ lifts to the resolution. 2) The
  $\Cs$-action allows one to globalize the resolution of the formal
  neighborhood of $0$ to a resolution of the whole of
  $\mathfrak{M}_0(\alpha,0)$.  For the converse statement, we restrict
  a symplectic resolution of $\mathfrak{M}_0(\alpha,0)$ to the formal
  neighborhood of zero.
\end{proof}

By Corollary \ref{cor:formnbd}, if one point in a stratum
$\Nak{\lambda}{\a}{\theta}_{\tau} \subset \Nak{\lambda}{\a}{\theta}$
is formally resolvable, then so too is every other point in the
stratum. If $\tau = (e_1, \beta^{(1)}; \ds ; e_k, \beta^{(k)})$, then
define the greatest common divisor $\gcd (\tau)$ of $\tau$ to be the
greatest common divisor of the $e_i$. If the greatest common divisor
of $\tau$ is $k$, then each point in
$\Nak{\lambda}{\a}{\theta}_{\tau}$ corresponds to a representation of
the form $Y^{\oplus k}$ for some $\theta$-polystable representation $Y$. Let $U_{\text{fr}} \subset \Nak{\lambda}{\a}{\theta}$ be the union
of all strata $\Nak{\lambda}{\a}{\theta}_{\tau}$ such that
$\gcd(\tau) = 1$.

\begin{lem}
	Let $\alpha \in \Sigma_{\lambda,\theta}$. Then $U_{\text{fr}}$ is a dense open subset of $\Nak{\lambda}{\a}{\theta}$.
\end{lem}

\begin{proof}
The set $U_{\text{fr}}$ is dense because it contains the open stratum $\Nak{\lambda}{\a}{\theta}_{(1,\a)}$, consisting of stable representations. We will show that the complement to $U_{\text{fr}}$ is closed in $\Nak{\lambda}{\a}{\theta}$. It suffices to show that if the greatest common divisor of $\rho $ is greater than one and $\Nak{\lambda}{\a}{\theta}_{\tau} \subset \overline{\Nak{\lambda}{\a}{\theta}}_{\rho}$ then $\gcd (\tau) > 1$ too. The argument is similar to the proof of Lemma \ref{lem:Uopen}. Let $x \in \Nak{\lambda}{\a}{\theta}_{\rho}$ and $G_{\rho} \subset \G(\a)$ its stabilizer. By Proposition \ref{prop:strata}, there exists $x' \in \Nak{\lambda}{\a}{\theta}_{\tau} $ such that its stabilizer $G_{\tau}$ contains $G_{\rho}$. Let $\gcd(\rho) = k$, so that $x$ corresponds to a representation $Y \otimes V$ for some $\theta$-polystable representation $Y$, and $k$-dimensional vector space $V$. Notice that $\a = k \dim Y$. Then $GL(V)$ is a subgroup of $G_{\rho}$, and hence of $G_{\tau}$ too. An elementary argument shows that this implies that $x'$ corresponds to a representation $Y' \otimes V$ for some $\theta$-polystable representation $Y'$. Thus, $\gcd (\tau) > 1$. In fact, we have shown that if $\Nak{\lambda}{\a}{\theta}_{\tau} \subset \overline{\Nak{\lambda}{\a}{\theta}}_{\rho}$, then $\gcd(\rho)$ divides $\gcd(\tau)$. Thus, $U_{\text{fr}}$ is open in $\Nak{\lambda}{\a}{\theta}$.
\end{proof}

\begin{thm}
  Let $\a \in \Sigma_{\lambda,\theta}$ be an anisotropic root, $\a = n \beta$ for some
  indivisible root $\beta$ and some $n > 1$. Assume that $(n,p(\beta)) \neq (2,2)$. Then a point $x$ is formally resolvable if and only if $x \in U_{\text{fr}}$. 
\end{thm}

\begin{proof}


Let $x \in \Nak{\lambda}{\a}{\theta}$ have representation type $\tau = (e_1, \beta^{(1)}; \ds ; e_k, \beta^{(k)})$, where
$m := \gcd(\tau)$. By Corollary \ref{cor:formnbd}, $\widehat{\mathfrak{M}}_{\lambda}(\a,\theta)_x \simeq
\widehat{\mathfrak{M}}_{0}(\mathbf{e},0)_0$ and hence Lemma \ref{lem:formal0} says that $x$ is formally
resolvable if and only if $\mathfrak{M}_{0}(\mathbf{e},0)$ admits a projective symplectic resolution. By definition, the greatest common divisor of
$\mathbf{e}$ is $m$.
Proposition \ref{prop:simplevector} says that $\mathbf{e}$ belongs to $\Sigma_{0,0}$ for the quiver underlying
$\mathfrak{M}_{0}(\mathbf{e},0)$. Moreover, by remark \ref{rem:etalep}, we have $p(\a) = p(\mathbf{e})$ which implies that
$\mathbf{e} = m \mathbf{f}$ with both $\mathbf{e}$ and $\mathbf{f}$ anisotropic. If $m = 1$ then Corollary \ref{c:proper=proj} (i) implies that $x$ is formally resolvable. Thus, we just need to show that if $m > 1$ then $x$ is not formally resolvable. 

First, we show: 
\begin{equation}\label{eq:LR22}
  \left( m, p(\mathbf{f}) \right) = (2,2) \quad \Leftrightarrow \quad \left( n,p\left( \beta \right) \right) = (2,2). 
\end{equation}
Assume that the left hand side of \eqref{eq:LR22} holds. Then $p(\alpha) = n^2 (p(\beta) -1) + 1$ implies that   
$$
n^2 (p(\beta) -1) + 1 = p(\a) = p(\mathbf{e}) = 5
$$
and hence $n^2 (p(\beta) -1) = 4$. But $p(\beta) > 1$ since $\beta$ is anisotropic, and $2$ divides $n$. Thus, $n = 2$ and $p(\beta) = 2$. Conversely, assume that the right hand side of \eqref{eq:LR22} holds. Then $5 = p(\a) = p(\mathbf{e})$ implies that $m^2 (p(\mathbf{f}) - 1) = 4$. Since $\mathbf{f}$ is anisotropic and we have assumed that $m > 1$, we deduce that $  \left( m, p(\mathbf{f}) \right) = (2,2)$. 

Notice that we have assumed in the statement of the theorem that $\left( n,p\left( \beta \right) \right) \neq (2,2)$. Thus, $\left(m,p(\mathbf{f})\right)$ cannot equal $(2,2)$.
Then Theorem \ref{thm:nonisonores} says that $\mathfrak{M}_{0}(\mathbf{e},0)$ does not admit a projective symplectic resolution because $m > 1$.
\end{proof}

In the case where $\a \in \Sigma_{\lambda,\theta}$ equals $2 \beta$
for some
 root $\beta \in \Sigma_{\lambda,\theta}$ with
$p(\beta) = 2$, every point in $\Nak{\lambda}{\a}{\theta}$ is formally
resolvable. Similarly, if $\a$ is  indivisible (or a multiple of an isotropic root), then
every point in $\Nak{\lambda}{\a}{\theta}$ is formally resolvable.

\begin{remark}\label{r:not-resolv-form-resolv}
If $U_{\text{fr}} \subsetneq \Nak{\lambda}{\a}{\theta}$ then Corollary \ref{cor:openfact} implies that any open subset of $U_{\text{fr}}$ not contained in the smooth locus of $\Nak{\lambda}{\a}{\theta}$ does not admit a symplectic resolution, i.e. the singular locus of $U_{\text{fr}}$ consists of points that cannot be resolved Zariski locally, but do admit a resolution in a formal neighborhood (in fact \'etale locally). 
\end{remark}

\section{Namikawa's Weyl group}\label{sec:Namikawa}

In the paper \cite{Namikawa2}, Namikawa defined a finite group $W$ associated to any conic affine symplectic singularity $X$ such that the symplectic form on $X$ has weight $\ell > 0$ with respect to the torus action. The group $W$ acts as a reflection group on $H^2(Y,\mathbb{R})$, where $Y \rightarrow X$ is any $\Q$-factorial terminalization of $X$, whose existence is guaranteed by the minimal model program. The group $W$ plays a key role in the birational geometry of $X$; see \cite{Namikawa3} and \cite{BellamyNamikawa}.  

One computes $W$ as follows: let $\mc{L}$ be a codimension $2$ leaf of $X$ and $x \in \mc{L}$. Then the formal neighborhood of $x$ in $X$ is isomorphic to the formal neighborhood of $0$ in $\C^{2(n-1)} \times \C^2 / \Gamma$, where $2 n = \dim X$ and $\Gamma \subset SL_2(\C)$ is a finite group; see \cite[Lemma 1.3]{Namikawa}. Associated to $\Gamma$, via the McKay correspondence, is a Weyl group $W_{\mc{L}}$ of type $A,D$ or $E$. The fundamental group $\pi_1(\mc{L})$ acts on $W_{\mc{L}}$ via Dynkin automorphisms. Let $W_{\mc{L}}'$ denote the centralizer of $\pi_1(\mc{L})$  in $W_{\mc{L}}$. Then 
$$
W := \prod_{\mc{L}} W_{\mc{L}}'. 
$$
Thus, in order to compute $W$, it is essential to classify the codimension $2$ leaves of $X$, and describe $\pi_{1} (\mc{L})$. This is the goal of this section.

\subsection{The proof of Theorem \ref{t:codim-two-strata}}
\label{sec:codim-two-strata-pf} We assume
throughout that $\a \in \Sigma_{\lambda,\theta}$, hence it is a
root. Therefore the support of $\a$ on the quiver is connected.  We
can assume, up to replacing the quiver by the subquiver whose vertices
are the support of $\a$, and whose arrows are the ones with endpoints
in the support, that $\a$ is sincere.  Then, the quiver is
connected. We may assume that $\a$ is imaginary, otherwise the
statement is vacuous.


Our goal is to compute the codimension two leaves of
$\Nak{\lambda}{\a}{\theta}$, proving Theorem
\ref{t:codim-two-strata}. Recall from Definition \ref{defn:isodecomp}
that
$\alpha = \beta^{(1)} + \cdots + \beta^{(s)} + m_1 \gamma^{(1)} +
\cdots m_t \gamma^{(t)}$ is an isotropic decomposition if
\begin{enumerate}
	\item[(a)] $\beta^{(i)}, \gamma^{(j)} \in \Sigma_{\lambda,\theta}$. 
	\item[(b)] The $\beta^{(i)}$ are imaginary roots.
	\item[(c)] The $\gamma^{(i)}$ are \textit{pairwise distinct} real roots.
	\item[(d)] If $\overline{Q}''$ is the quiver with $s + t$ vertices without
	loops and $-(\alpha^{(i)}, \alpha^{(j)})$ arrows from vertex $i$ to vertex $j   \neq i$, where $\alpha^{(i)},\alpha^{(j)} \in \{ \beta^{(1)}, \ds,
	\beta^{(s)}, \gamma^{(1)} , \ds , \gamma^{(t)} \}$, then $Q''$ is an
	affine Dynkin quiver.
	\item[(e)] The dimension vector $(1,\ds, 1, m_1, \ds, m_t)$ of $Q''$ (where
	there are $s$ ones) equals $\delta$, the minimal imaginary root.
\end{enumerate}
To prove this, first let us consider a general stratum 
$\tau$ of $\Nak\lambda\alpha\theta$, 
\begin{equation} \label{e:tau-defn}
\tau = \left(n_1, \beta^{(1)}; \ds ; n_s, \beta^{(s)}; m_1,
\gamma^{(1)}; \ds ; m_t, \gamma^{(t)} \right), \quad \text{$\beta^{(i)}$ are imaginary, and $\gamma^{(i)}$ are real}.
\end{equation}
Since there is only one $\theta$-stable representation of dimension equal to each real root in $\Sigma_{\lambda,\theta}$, it follows that the $\gamma^{(i)}$ are all distinct.

Let us now set $\alpha^{(i)} := \beta^{(i)}$ for $1 \leq i \leq s$ and
$\alpha^{(i)} = \gamma^{(i-s)}$ for $s+1 \leq i \leq s+t$. 
Let $k_i := n_i$ 
for $1 \leq i \leq s$ and $k_{i} = m_{i-s}$ for
$s+1 \leq i \leq s+t$; let $\mathbf{k} = (k_1,\ds, k_{s+t})$. 
By Theorem \ref{thm:etalelocalgeneral}, at a point of 
this stratum, $\Nak\lambda\alpha\theta$ is \'etale-equivalent to
$\Nak{Q'}{0}{\mathbf{k}}_{0}$, where the notation means we use the
quiver $Q'$ instead of $Q$. Recall
that $\overline{Q}'$ is the quiver with $s+t$
vertices, $2 p(\alpha^{(i)})$ loops at the $i$th vertex and
$-(\alpha^{(i)},\alpha^{(j)})$ arrows between $i$ and $j$.  

Note that $Q''$ is obtained from $Q'$ by discarding all loops at
vertices. We will prove that, in the case that the stratum has
codimension two, $\Nak{Q''}{0}{\mathbf{k}}_{0}$ \'etale-locally
describes a transverse slice to the stratum.

\begin{lem}\label{l:q''-tr} Suppose that $\tau$ is as in \eqref{e:tau-defn} and moreover
	$n_i=1$ for all $i$.
	Then
	at every point of the stratum 
	there is an \'etale-local transverse slice isomorphic to a neighborhood
	of zero in $\Nak{Q''}{0}{\mathbf{k}}$.
\end{lem}
\begin{proof}
	A neighborhood of a point of the stratum is \'etale-equivalent to
	a neighborhood of zero in $\Nak{Q''}{0}{\mathbf{k}}$. Inside the latter,
	the stratum containing zero consists of the representations which are a direct sum of simple representations, one at each vertex. At the vertices $1,\ldots,s$, this representation has dimension one; at the other vertices there are no loops and hence the simple representations are the standard ones.  The stratum has dimension $2 \sum_{i=1}^s p(\beta^{(i)})$, where $p(\beta^{(i)})$ equals the number of loops at the vertex $i$. A transverse slice is thus given by the representations which assign zero to all of the loops, which obviously identifies
	with $\Nak{Q''}{0}{\mathbf{k}}$.
\end{proof}
\begin{lem}\label{l:codim2} 
	Suppose that $\tau$ has codimension two. Then $n_i = 1$ for all $i$.
	Moreover, the anisotropic $\beta^{(i)}$ are pairwise distinct,
        except
        in the case where $Q''$ is of affine type $A_1$, so that $s=2$, $t=0$, and $\beta:=\beta^{(1)}=\beta^{(2)}$ has self pairing $(\beta,\beta)=-2$.
\end{lem}
\begin{proof}
	The codimension two condition can be written as:
	\begin{equation}\label{e:dimcount}
	1=p(\alpha) - \sum_i  p \left(\beta^{(i)} \right) - \sum_i p
	\left(\gamma^{(i)}\right) = p(\alpha) - \sum_i 
	p\left(\beta^{(i)}\right),
	\end{equation}
	since $\dim \Nak{\lambda}{\a}{\theta}_{\tau} = 2\sum_i p
	\left(\beta^{(i)} \right) + 2\sum_i p \left(\gamma^{(i)}\right)$. 
	Note that the $n_i \beta^{(i)}$ are
	themselves imaginary roots.  Let $I_{\text{ani}} \subseteq I := \{1,\ldots,s\}$ be the set
	of indices such that $\beta^{(i)}$ is anisotropic, and let $I_{\text{iso}}:=I\setminus I_{\text{ani}}$; so $p(\beta^{(i)})=1$ if and only if $i \in I_{\text{iso}}$. Note the identity
	\[
	p(m\alpha) = m^2(p(\alpha) - 1) + 1.
	\]
	Since $\alpha \in
	\Sigma_{\lambda,\theta}$,
	\[
	p(\alpha) \geq 
	\sum_{i \in I_{\text{ani}}} p\left(n_i\beta^{(i)}\right) + \sum_{i \in I_{\text{iso}}}
	n_i p\left( \beta^{(i)} \right) + \sum_i 
	p\left(\gamma^{(i)}\right) = \sum_{i \in I_{\text{ani}}} (n_i^2 p\left(\beta^{(i)}\right) + (1-n_i^2)) + \sum_{i \in I_{\text{iso}}} n_i,
	\]
	with equality holding only if $s=1$ and $\beta^{(1)}$ is anisotropic.
	Therefore, the RHS of \eqref{e:dimcount} is greater than or equal to
	\[
	\sum_{i \in I_{\text{ani}}} (n_i^2-1) (p\left(\beta^{(i)}\right)-1)
	+ \sum_{i \in I_{\text{iso}}} (n_i-1),
	\]
	again with equality only if $s=1$ and $\beta^{(1)}$ is anisotropic.
	Therefore, if $n_i >
	1$ for any $i$, 
	then the RHS of \eqref{e:dimcount} is strictly greater than one, a
	contradiction.

	Next we show that the anisotropic $\beta^{(i)}$ are
	all distinct, except in the case given.  Group the anisotropic roots 
        that are not distinct together: let $I' \subseteq I_{\text{ani}}$ be an index set so that $\beta^{(i)}, i \in I'$ gives all of the distinct anisotropic roots once each,
	and let $\ell_i := |\{j: \beta^{(j)} = \beta^{(i)}\}|$. Then, since $p(\alpha) \geq \sum_{i \in I_{\text{iso}}} p\left( \beta^{(i)} \right) + \sum_{i \in I'} p\left( \ell_i \beta^{(i)} \right)$, we conclude that
	\begin{multline*}
	p(\alpha) - \sum_{i \in I} p\left(\beta^{(i)}\right) \geq
	\sum_{i \in I'} p \left( \ell_i \beta^{(i)} \right) - \ell_i p \left( \beta^{(i)} \right)
	= \sum_{i \in I'} (\ell_i^2-\ell_i) (p\left(\beta^{(i)}\right)-1)-(\ell_i-1) \\ = \sum_{i \in I'} (\ell_i - 1)(\ell_i (p\left( \beta^{(i)} \right) - 1) - 1),
	\end{multline*}
	which        is greater than one if any $\ell_i > 1$ with $i \in I_{\text{ani}}$, unless this happens for a unique $i \in I_{\text{ani}}$ with $\ell_i=2$ and $p\left( \beta^{(i)} \right) = 2$. To have the strict equality $p(\alpha) =  \sum_{i \in I_{\text{iso}}} p\left( \beta^{(i)} \right) + \sum_{i \in I'} p\left( \ell_i \beta^{(i)} \right)$, we also need to have $\alpha=\ell_i \beta^{(i)}$.  We obtain
      exactly the exceptional case stated.

       We also give an alternative argument to the preceding paragraph: if some of the roots $\beta_i, i \in J \subseteq I$ are the same root, call it $\beta$, for $|J|>1$, then the subset of $Q''$ supported on $J$ is a complete graph with $-(\beta, \beta)$ arrows between each pair of vertices. Given $\beta$ is anisotropic and  $Q''$ is affine Dynkin, we must have  $(\beta,\beta)=-2$ and $|J| = 2$. In this case this subgraph is already affine Dynkin of type $A_1$, so it is the entire quiver $Q''$. \qedhere
        
\end{proof}

We can now proceed with the proof of the theorem:

\begin{proof}[Proof of Theorem \ref{t:codim-two-strata}]
	Consider a general stratum $\tau$ as in \eqref{e:tau-defn}. By
	Lemmas \ref{l:q''-tr} and \ref{l:codim2}, we know that $\tau$ has
	codimension two if and only if $n_i = 1$ for all $i$ and
	$\dim \Nak{Q''}{0}{\mathbf{k}} = 2$. The latter is certainly
	true if $\tau$ is given by an isotropic decomposition. Moreover, in this
	case $\Nak{Q''}{0}{\mathbf{k}}$ is a du Val singularity of type given by the quiver $Q''$.  
	
	It remains only to show that, if $\tau$ has codimension two, then
	$Q''$ is affine Dynkin (ADE), with $\mathbf{k}$ the minimal
	imaginary root. Consider the canonical decomposition of
	$\mathbf{k}$, say
	$\mathbf{k} = \mathbf{k}^{(1)} + \cdots + \mathbf{k}^{(r)}$.  Then
	$\Nak{Q''}{0}{\mathbf{k}} \cong \prod_{i=1}^r
	\Nak{Q''}{0}{\mathbf{k}^{(i)}}$.
	The dimension of the latter is $2\sum_{i=1}^r
	p(\mathbf{k}^{(i)})$.
	Hence exactly one of the $\mathbf{k}^{(i)}$ is isotropic, i.e.,
	$p(\mathbf{k}^{(i_0)})=1$ for some $i_0$, and the others are real,
	i.e., $p(\mathbf{k}^{(i)}) = 0$ for $i \neq i_0$. Let
	$\mathbf{k}' := \mathbf{k}^{(i_0)}$. Since
	$\mathbf{k}' \in \Sigma_{0,0}(Q'')$, it follows that it is the
	minimal imaginary root of some affine Dynkin subquiver of $Q''$ (by
	the argument of the proof of \cite[Proposition 1.2.(2)]{CBdecomp}).  
	
	We claim that $\mathbf{k}=\mathbf{k}'$. Given this, since every component of $\mathbf{k}$ is nonzero, $Q''$ is indeed affine Dynkin, which completes the proof.  
	
	To prove the claim, write $\mathbf{k}' = (k'_1, \ldots, k'_{s+t})$ with $k'_i \leq k_i$ for all $i$.  Let $\alpha' := \sum_{i=1}^{s+t} k'_i \alpha^{(i)}$.  Then
	Lemmas \ref{l:q''-tr} and \ref{l:codim2} applied to $\alpha'$ show
	also that the stratum $\tau'$ corresponding to $\mathbf{k}'$ in $\Nak\lambda{\alpha'}\theta$ has codimension two. That is:
	\begin{equation}\label{eq:WeylNameq1}
	2\sum_{i: k'_i \neq 0} p(\alpha^{(i)}) = \dim \Nak\lambda{\alpha'}\theta - 2.
	\end{equation}
	By Lemma \ref{l:cd-ineq} below, the RHS of \eqref{eq:WeylNameq1} is at most $2p(\alpha')-2$.  Now adding
	$\sum_{i: k'_i=0} p(\alpha^{(i)})$ to both sides, we obtain:
	\begin{equation}\label{eq:WeylNameq2}
	\sum_i p(\alpha^{(i)}) \leq p(\alpha') + \sum_{i: k'_i = 0} p(\alpha^{(i)}) - 1.
	\end{equation}
	The LHS of \eqref{eq:WeylNameq2} equals $p(\alpha)-1$ by assumption. Therefore we obtain:
	\begin{equation}\label{eq:WeylNameq3}
	p(\alpha) \leq p(\alpha') + \sum_{i: k'_i = 0} p(\alpha^{(i)}).
	\end{equation}
	Now, replace $\alpha'$ and each of the $\alpha^{(i)}$ in the RHS of \eqref{eq:WeylNameq3} by their canonical
	decompositions and let $\eta_1, \ldots, \eta_q$ be the resulting
	elements of $\Sigma_{\lambda,\theta}$ with multiplicity.
	By Lemma \ref{l:cd-ineq} again, we obtain that $p(\alpha) \leq \sum_{i=1}^q p(\eta_i)$. Since  $\alpha \in \Sigma_{\lambda,\theta}$, this can only happen if $q=1$, i.e., $\alpha=\alpha'=\eta_1$. This is true if and only if $\mathbf{k}=\mathbf{k}'$.
\end{proof}

\begin{lem}\label{l:cd-ineq}
	Suppose $\alpha \in \N R^+_{\lambda,\theta}$ has canonical
	decomposition $\alpha = \sum_i n_i \sigma^{(i)}$ with respect to
	$\lambda$ and $\theta$.  Then
	$p(\alpha) \leq \sum_i n_i p(\sigma^{(i)})$. 
\end{lem}
\begin{proof}
	Let $\lambda'$ be such that $R^+_{\lambda'} = R^+_{\lambda,\theta}$.
	As $\alpha \in \N R^+_{\lambda,\theta}=\N R^+_{\lambda'}$, we know
	that $\mu^{-1}_\alpha(\lambda')$ is nonempty.  The latter is a fiber
	of a map $\bigoplus_{a \in Q_1} \Hom(\C^{\alpha_{t(a)}},
	\C^{\alpha_{h(a)}}) \to \pg(\alpha)$, where $\pg(\alpha)$ is the Lie
	algebra of $\PG(\alpha)$.  All of the irreducible components of the
	latter must have dimension at least $\sum_{a \in \overline{Q_1}}
	\alpha_{t(a)} \alpha_{h(a)} - \sum_{i \in Q_0} \alpha_i^2 +1 = \alpha
	\cdot \alpha - 2\langle \alpha, \alpha \rangle +1 = \alpha \cdot
	\alpha +2p(\alpha) - 1$.
	
	On the other hand, by \cite[Theorem 4.4]{CBmomap}, $\dim
	\mu^{-1}_{\alpha}(\lambda') = \alpha \cdot \alpha - \langle \alpha, \alpha \rangle + m= \alpha \cdot \alpha + p(\alpha) + (m-1)$
	where $m$ is the maximum value of $\sum_{i} p(\alpha^{(i)})$ with
	$\alpha^{(i)} \in R^+_{\lambda'}$ and $\alpha = \sum \alpha^{(i)}$; as
	remarked at the top of page 3 in \cite{CBdecomp}, we have $m = \sum_i
	n_i p(\sigma^{(i)})$ (it is a direct consequence of \cite[Theorem
	1.1]{CBdecomp} which we discussed before Theorem
	\ref{thm:decompCB2}).\footnote{Another interpretation of these facts 
		is that $\mu^{-1}_\alpha(\lambda')$ has some
		irreducible component of maximum dimension whose generic element is
		semisimple with the canonical decomposition. The same fact can be
		deduced for $\mu^{-1}_\alpha(\lambda)^\theta$, replacing semisimple by canonically polystable.}  We conclude that
	$\alpha \cdot \alpha + p(\alpha) + (m-1) \geq \alpha \cdot \alpha +
	2p(\alpha) - 1$.  Therefore, $m \geq p(\alpha)$, as desired.
	%
\end{proof}
\begin{remark}
	The lemma can be strengthened to prove: for any decomposition
	$\alpha = \sum_j \alpha^{(j)}$ with
	$\alpha^{(j)} \in \N R^+_{\lambda,\theta}$, we have
	$\sum_j p(\alpha^{(j)}) \leq \sum_i n_i p(\sigma^{(i)})$.  This
	generalizes an observation on \cite[p.~3]{CBdecomp} (dealing with
	the case where the $\alpha^{(j)}$ are roots).  To prove this, for
	arbitrary $\alpha^{(j)}$, we can apply the lemma to each of the
	$\alpha^{(j)}$, and then we get that
	$\sum_j p(\alpha^{(j)}) \leq \sum_j p(\beta^{(j)})$ for some roots
	$\beta^{(j)} \in R^+_{\lambda,\theta}$ with $\alpha=\beta^{(j)}$;
	then we are back in the case of roots so that
	$\sum_j n_i p(\sigma^{(i)}) \geq \sum_j p(\beta^{(j)})$.
\end{remark}
\begin{remark}
	The arguments of \cite{CBmomap,CBdecomp} can be generalized to the
	context of the pair $(\lambda,\theta)$, which as we pointed out in \S
	\ref{sec:candec} would eliminate the need of picking a $\lambda'$ as in
	the proof of the lemma above.
\end{remark}

\def\cprime{$'$} \def\cprime{$'$} \def\cprime{$'$} \def\cprime{$'$}
  \def\cprime{$'$} \def\cprime{$'$} \def\cprime{$'$} \def\cprime{$'$}
  \def\cprime{$'$} \def\cprime{$'$} \def\cprime{$'$} \def\cprime{$'$}
  \def\cprime{$'$} \def\cprime{$'$}

\end{document}